\documentclass[a4paper,UKenglish,]{lipics-v2018}
\pdfoutput=1
\nolinenumbers

\title{Automatic semigroups \emph{vs} automaton semigroups}
\titlerunning{Automatic semigroups \emph{vs} automaton semigroups}

\author{Matthieu Picantin}
{IRIF, UMR 8243 CNRS \& Univ. Paris Diderot, 75013 Paris, France}
{picantin@irif.fr}
{https://orcid.org/0000-0002-7149-1770}
{}%

\authorrunning{M. Picantin}
\Copyright{Matthieu Picantin}

\subjclass{F.1.1 Models of Computation, F.4.3 Formal Languages.}
\keywords{Mealy machine, semigroup, rewriting system, automaticity, self-similarity.}

\funding{This work was partially supported by the French \emph{Agence Nationale pour la~Recherche},
through the project \textsf{MealyM} ANR-JS02-012-01.}

\usepackage{ifdraft,textcomp,float,amsmath}
\usepackage{xcolor,etex,upgreek}
\usepackage{marginnote}

\usepackage{pgf,tikz,braids}
\usetikzlibrary{arrows,automata,shapes,backgrounds,shapes.symbols}
\usetikzlibrary{positioning}
\usetikzlibrary{decorations.pathmorphing}
\usetikzlibrary{snakes}
\usetikzlibrary{trees,decorations}
\usetikzlibrary{calc}
\usetikzlibrary{shapes.callouts,shapes.arrows}

\colorlet{couleurLettre}{green!65!black}
\tikzstyle{etatPPrimal}=[minimum size=12pt,inner sep=0pt,color=black,draw=black,fill=white]
\tikzstyle{etatDualDual}=[minimum size=12pt,inner sep=0pt,color=couleurLettre,draw=couleurLettre,fill=white]

\theoremstyle{plain}
\newtheorem{proposition}[theorem]{Proposition}

	\newcommand{\bzer}{{\color{black} 0}}
	\newcommand{\bone}{{\color{black} 1}}
	\newcommand{\bmon}{{\color{black} \overline{1}}}
	
	\newcommand{\czer}{0}
	\newcommand{\cone}{1}	
	
	\newcommand{\cmon}{\overline{1}}

\newcommand{\aut}[1]{{\mathcal #1}}

\newcommand{\dz}{\mathfrak d}
\newcommand{\word}[1]{{\mathbf {#1}}}
\def\eref#1{(\ref{#1})}

\newcommand{\pres}[1]{\langle{#1}\rangle}
\newcommand{\presm}[1]{\pres{{#1}}^{\ttunit}_{\hspace*{-1pt}+}}
\newcommand{\press}[1]{\pres{{#1}}_{\hspace*{-1pt}+}}
\def\MM{S}
\def\ff{a}
\def\gg{b}
\def\hh{c}

\newcommand\Gar{\mathcal Q}
\newcommand\Fin{\mathcal J}
\newcommand\Bic{\mathbf B}
\newcommand\Malcev{\mathbf T}
\def\ss{s}
\def\tt{t}
\def\pp{p}
\newcommand\EV{\textsc{ev}}
\newcommand\NF{\textsc{nf}}
\newcommand\ID{\textsc{i}\hbox{\small d}}
\newcommand\NM{\textsc{n}}

\newcommand\NMbar{\hbox{$\overline{\textsc{n}}$}}
\newcommand{\card}[1]{|#1|}

\newcommand\comp{\mathbin{\vcenter{\hbox{$\scriptstyle\circ$}}}}
\newcommand\Mea{\mathcal M}
\newcommand\Thu{\mathcal T}

\makeatletter
   \def\equalsfill{$\m@th\mathord=\mkern-7mu
     \cleaders\hbox{$\!\mathord=\!$}\hfill
     \mkern-7mu\mathord=$}
\makeatother
\newcommand\BB{\hbox{$\mathbf{B}$}}

\newcommand\BS{\hbox{$\mathbf{B\!\!\;S}$}}
\newcommand\AK{\hbox{$\mathbf{A\!\!\;K}$}}

\newcommand\WW{\hbox{$\mathbf{W}$}}
\def\ttunit{\mathtt{1}}
\def\tta{\mathtt{a}}
\def\ttb{\mathtt{b}}
\def\ttc{\mathtt{c}}
\def\ttd{\mathtt{d}}

\def\ttx{\mathtt{x}}
\def\tty{\mathtt{y}}
\def\ttz{\mathtt{z}}

\def\resp{\hbox{\emph{resp.} }}

\ifdraft{
\newcommand{\colorM}[1]{{\color{blue} #1}}
\newcommand{\itemM}[1]{\colorS{\item[M:] #1}}
\newcommand{\commM}[1]{{\marginpar{\tiny \colorM{#1}}}}
}
{
\newcommand{\colorM}[1]{}
\newcommand{\itemM}[1]{}
\newcommand{\commM}[1]{}
}
\usepackage{xcolor}
\definecolor{darkgreen}{rgb}{0.2,0.55,0}

\EventEditors{John Q. Open and Joan R. Access}
\EventNoEds{2}
\EventLongTitle{42nd Conference on Very Important Topics (CVIT 2016)}
\EventShortTitle{CVIT 2016}
\EventAcronym{CVIT}
\EventYear{2016}
\EventDate{December 24--27, 2016}
\EventLocation{Little Whinging, United Kingdom}
\EventLogo{}
\SeriesVolume{42}
\ArticleNo{23}

\begin{document}

\maketitle

\begin{abstract} We develop an effective and natural approach to interpret any
semigroup admitting a special language of greedy normal forms as an automaton semigroup,
namely the semigroup generated by a Mealy automaton encoding the behaviour of such a language
of greedy normal forms under one-sided multiplication.
The framework embraces many of the well-known classes of (automatic) semigroups:
free semigroups, free commutative semigroups, trace or divisibility monoids,
braid or Artin--Tits or Krammer or Garside monoids, Baumslag--Solitar semigroups, etc.
Like plactic monoids or Chinese monoids,
some neither left- nor right-cancellative automatic semigroups are also investigated,
as well as some residually finite variations of the bicyclic monoid. 
It provides what appears to be the first known connection from a class of automatic semigroups
to a class of automaton semigroups. 
It is worthwhile noting that, 
"being an automatic semigroup" and "being an automaton semigroup"
become dual properties in a very automata-theoretical sense.
Quadratic rewriting systems and associated tilings appear as the cornerstone of our construction.
\end{abstract}

\section{Introduction}
\label{sec-intro}

The half century long history of the interactions between (semi)group theory and automata theory
went through a pivotal decade from the mid-eighties to the mid-nineties.
Contemporaneously but independently, two new theories truly started to develop and thrive:
automaton (semi)groups on the one hand with the works of~Aleshin~\cite{aleshin,aleshin:free} and~Grigorchuk~\cite{grigorch:burnside,grigorch:degrees} and the book~\cite{nekrash:self-similar},
and automatic (semi)groups on the other hand with the work of~Cannon and~Thurston and the book~\cite{EpsteinWord}.
We refer to~\cite{silva} for a clear and short survey on the known interactions between groups and automata.
A deeper and more extended survey by~Bartholdi and~Silva can be found out in two chapters~\cite{BS10,BS10b} of the forthcoming \textsf{AutoMathA} handbook.
We can refer to~\cite{Cain,McCune} for automaton semigroups and to~\cite{CRRT,HT} for automatic semigroups.

\subparagraph*{Remote siblings}
As their very name indicates, automaton (semi)groups and automatic (semi)groups share a same defining object:
the \emph{automaton} or the letter-to-letter transducer in this case.
Beyond this common origin, these two topics until now happened to remain largely distant both in terms of community and in terms of tools or results.
Typically, any paper on one or the other topic used to contain a sentence like "it should be emphasised that, despite their similar names,
the notions of automaton (semi)groups are entirely separate from the notions of automatic (semi)groups".  
This was best evidenced by the above-mentioned valuable handbook chapter~\cite{BS10} which splits into exactly two sections
(automatic groups and automaton groups) without any reference between one and the other appearing explicitly.

\subparagraph*{Related open problems}
A significant problem is to recognise whether a given (semi)group is self-similar, that is, an automaton (semi)group.
Amongst the thirty-odd listed problems from~\cite{aim}, we can pick the one with the number~1.1:
\begin{quote}
\vspace*{-2.5pt}
{\sffamily\bfseries\raggedright Problem A.} It seems quite difficult to show whether a given group is self-similar. Are Gromov hyperbolic groups self-similar?
Find obstructions to self-similarity.
\vspace*{-2.5pt}
\end{quote}
Amongst the unsolved problems in group theory from the Kourovka Notebook \cite{kou},
the one (with number~16.84) asked by~Sushchanskii (see also~\cite{lmos}) can be formulated as follows:
\begin{quote}
\vspace*{-2.5pt}
{\sffamily\bfseries\raggedright Problem B.} Is the $n$-strand braid group~$\BB_n$ a subgroup of some automaton group? 
\vspace*{-2.5pt}
\end{quote}
All these questions can be meaningfully rephrased in terms of semigroups or monoids.

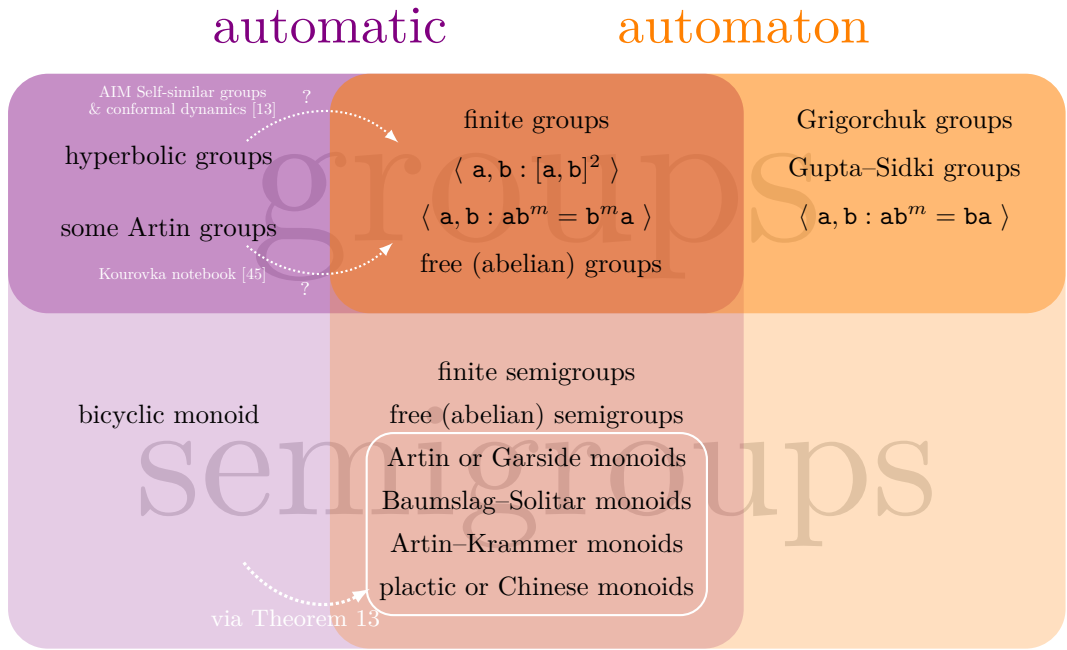
\begin{figure}
\centering
\begin{tikzpicture}[scale=.605,yscale=1.05]
	\node[thick,scale=7.4,opacity=.15,black] at (8,1) {groups};
	\node[thick,scale=6.2,opacity=.15,black] at (8,-4.45) {semigroups};
	\draw[rounded corners=15pt,draw opacity=0,fill=violet,fill opacity=.2] (-3.5,-8) rectangle (12.5,4);		
	\draw[rounded corners=15pt,draw opacity=0,fill=violet,fill opacity=.3] (-3.5,-1) rectangle (12.5,4);		
	\draw[rounded corners=15pt,draw opacity=0,fill=orange,fill opacity=.25] (3.5,-8) rectangle (19.5,4);		
	\draw[rounded corners=15pt,draw opacity=0,fill=orange,fill opacity=.4] (3.5,-1) rectangle (19.5,4);		
	\node[thick,scale=2,draw opacity=.5,violet] at (3.5,5) {automatic};
	\node[thick,scale=2,draw opacity=1,orange] at (12.5,5) {automaton};
		\node (hyperb)	at (0,2.25) {hyperbolic groups};
		\node (hyperbE)	[right of=hyperb,node distance=9mm] {};
		\node (hyperbNE)	[above of=hyperbE,node distance=1mm] {};
			\node[white,scale=.6] (kourovka)	at (0.3,3.60) {AIM Self-similar groups};	
			\node[white,scale=.6] (kourovka)	at (0.3,3.25) {\& conformal dynamics \cite{aim}};
		\begin{scope}[xshift=80mm]
		\node (fin)		at (0,3) {\hspace*{21pt}finite groups\hspace*{21pt}};
		\node (orbi) 	at (0,2) {$\pres{~\tta,\ttb:[\tta,\ttb]^2~}$};
		\end{scope}
	\path[->,>=latex,densely dotted,thick] (hyperbNE) edge[white,bend left=40] node[white,pos=.4,above]{\scriptsize\(?\)} (fin.south west);
		\node (artin)	at (0,0.75) {some Artin groups};
		\node (artinE)	[right of=artin,node distance=9mm] {};
		\node (artinSE)	[below of=artinE,node distance=1mm] {};
			\node[white,scale=.6] (kourovka)	at (0.3,-.2) {Kourovka notebook \cite{kou}};		
		\begin{scope}[xshift=80mm]
		\node (bsmm) 	at (0,1) {$\pres{~\tta,\ttb:\tta\ttb^m=\ttb^m\tta~}$};
		\node (free) 	at (0,0) {~~free (abelian) groups~~};
		\end{scope}
	\path[->,>=latex,densely dotted,thick] (artinSE) edge[white,bend right=40] node[thick,pos=.4,below]{\scriptsize\(?\)} (free.north west);
	\begin{scope}[xshift=160mm]
		\node (grig)	at (0,3) {Grigorchuk groups};
		\node (gsid)	at (0,2) {Gupta--Sidki groups};
			\node (bsolv)	at (0,1) {$\pres{~\tta,\ttb:\tta\ttb^m=\ttb\tta~}$};
	\end{scope}
	\begin{scope}[yshift=-70mm]
		\node (bicy)	at (0,3.85) {bicyclic monoid\vphantom{g}};
		\begin{scope}[xshift=80mm]
			\node (sfin)	at (0,4.75) {finite semigroups};
			\node (sfree) 	at (0,3.85) {free (abelian) semigroups};
		\end{scope}
	\begin{scope}[xshift=80mm]
		\node (sartin) 	at (0,2.95) {Artin or Garside monoids\vphantom{g}};
		\node (bsmn) 	at (0,2.05) {Baumslag--Solitar monoids};%
		\node (kram) 	at (0,1.15) {Artin--Krammer monoids\vphantom{g}};
		\node (plactic) 	at (0,0.25) {plactic or Chinese monoids};
		\end{scope}
		\node (chin)	at (0,1.15) {};
		\node (chinE)	[right of=chin,node distance=9mm] {};
		\node (chinSE)	[below of=chinE,node distance=1mm] {};
		\path[->,>=latex,densely dotted,very thick] (chinSE) edge[white,bend right=40]
				node[white,pos=.45,below]{\small via Theorem~\ref{thm-main}} (plactic.west);
		\draw[rounded corners=10pt,white,thick] (4.3,3.5) rectangle (11.7,-0.3);
	\end{scope}
\end{tikzpicture}
\caption{The big picture: comparing the classes of automatic \emph{vs} automaton (semi)groups.}%
\label{fig-picture}
\end{figure}

\subparagraph*{Our contributions}
The aim here is to establish a possible connection between being an automatic semigroup and being an automaton semigroup.
Preliminary observations are that these classes intersect non trivially and that neither is included in the other (see Figure~\ref{fig-picture}).
Like the Grigorchuk group for instance, many automaton groups are infinite torsion groups, hence cannot be automatic groups.
By contrast, it is an open question whether every automatic group is an automaton group. The latter is related to the question
whether every automatic group is residually finite, which remains open despite the works by~Wise~\cite{Wise} and~Elder~\cite{ElderPhD}.
Like the bicyclic monoid, some automatic semigroups are not residually finite, hence cannot be automaton semigroups (see~\cite{Cain} for instance).
As for the intersection, we know that at least finite semigroups, free semigroups
(of rank at least~2, see~\cite{BroughCain15,BroughCain16}),
free abelian semigroups happen to be both automatic semigroups and automaton semigroups.

\smallbreak We propose here a new and natural way to interpret algorithmically each semigroup from a wide class of automatic semigroups---encompassing all the above-mentioned classes---as an automaton semigroup (Theorem~\ref{thm-main}). Furthermore, it is worthwhile noting that, in all these cases, "being an automatic semigroup" and "being an automaton semigroup" become dual properties in a very automata-theoretical sense (Corollary~\ref{cor-duality}).

\smallbreak Occurring as the very first bridge between two hitherto irreconcilable research areas, Theorem~\ref{thm-main} allows us to also provide a (more than) positive answer to the monoidal version of Problem~B. While the $n$-strand braid monoid~$\BB^{\ttunit}_{\hspace*{-1pt}n+}$ is the paradigmatic example of an automatic monoid from~\cite[Chapter~9]{EpsteinWord}, Theorem~\ref{thm-main} implies that $\BB^{\ttunit}_{\hspace*{-1pt}n+}$ is (not only a submonoid of) an automaton monoid as well. From this significant milestone arise various new questions, that will hopefully swarm into the both areas.

\subparagraph*{Organisation}
The structure of the paper is as follows. As a simple preliminary, Section~\ref{sec-prelim} illustrates in a deliberately informal manner
how a single Mealy automaton can be used in order to define both self-similar structures and automatic structures
(via a principle of duality).
In Section~\ref{sec-mealy}, we set up the notations for Mealy automata and recall necessary notions of dual automaton, cross-diagram, and self-similar structure. In Section~\ref{sec-garside}, we recall basics about normal forms and automatic structures, and we give necessary notions
of quadratic normalisation, square-diagram, and Garside family.
Section~\ref{sec-main} is devoted to our main results (Theorem~\ref{thm-main} and Corollary~\ref{cor-duality}),
while Section~\ref{sec-ex} finally gathers several carefully selected examples, counterexamples, and open problems.
Proofs and comments omitted due to space constraints have been put into a clearly marked appendix.
 
\section{A preliminary example}\label{sec-prelim}

As their very name indicates, automaton (semi)groups and automatic (semi)groups share a same defining object.
In both cases, a \emph{Mealy automaton} (see Definition~\ref{def-mealy}) basically transforms words into words.

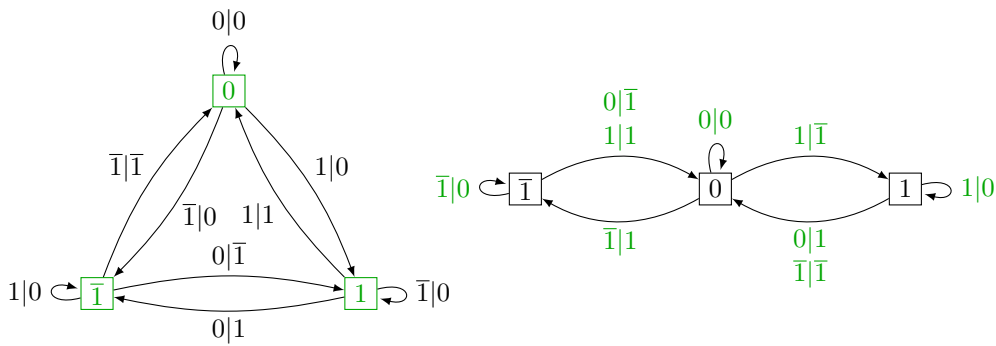
\begin{figure}[h]
\centering
\scalebox{1}{
\begin{tikzpicture}[->,>=latex,node distance=25mm]
	  \node[etatDualDual] (one) {\(0\)};
	  \node[etatDualDual] (two)	[below right=22.5mm and 13mm of one] {\(1\)};
	  \node[etatDualDual] (thr) 	[below left=22.5mm and 13mm of one] {\(\overline{1}\)};
	\path (one) edge [bend left=12] 	node[label,above right=.5pt,inner sep=1pt,rounded corners]{\(\bone|\bzer\)} (two);
	\path (two) edge [bend left=12] 	node[label,below left=.5pt,inner sep=1pt,rounded corners]{\(\bone|\bone\)} (one);
	\path (two) edge [bend left=12] 	node[label,below=1pt,inner sep=1pt,rounded corners]{\(\bzer|\bone\)} (thr);
	\path (thr) edge [bend left=12] 	node[label,above=1pt,inner sep=1pt,rounded corners]{\(\bzer|\bmon\)} (two);
	\path (one) edge [bend left=12] 	node[label,below right=.5pt,inner sep=1pt,rounded corners]{\(\bmon|\bzer\)} (thr);
	\path (thr) edge [bend left=12] 	node[label,above left=.5pt,inner sep=1pt,rounded corners]{\(\bmon|\bmon\)} (one);
	\path (one) edge [loop above=15] 	node[label,above=3pt,inner sep=1pt,rounded corners,align=center]{\(\bzer|\bzer\)} (one);
	\path (two) edge [loop right=15] 	node[label,right=3pt,inner sep=1pt,rounded corners,align=center]{\(\bmon|\bzer\)} (two);
	\path (thr) edge [loop left=15] 	node[label,left=3pt,inner sep=1pt,rounded corners,align=center]{\(\bone|\bzer\)} (thr);
\begin{scope}[xshift=64mm,yshift=-13mm]
	\node[etatPPrimal] (one) {\(0\)};
	\node[etatPPrimal] (two) [right of=one] {\(1\)};
	\node[etatPPrimal] (thr) [left of=one] {\(\overline{1}\)};
	\path (one) edge [bend left] 	node[label,above=1pt,inner sep=1pt,rounded corners]{\color{couleurLettre}\(\cone|\cmon\)} (two);
	\path (two) edge [bend left] 	node[label,below=1pt,inner sep=1pt,rounded corners,align=center]
							{\color{couleurLettre}\(\czer|\cone\)\\\color{couleurLettre} \(\cmon|\cmon\)} (one);
	\path (one) edge [bend left] 	node[label,below=1pt,inner sep=1pt,rounded corners]{\color{couleurLettre}\(\cmon|\cone\)} (thr);
	\path (thr) edge [bend left] 	node[label,above=1pt,inner sep=1pt,rounded corners,align=center]
							{\color{couleurLettre}\(\czer|\cmon\)\\\color{couleurLettre} \(\cone|\cone\)} (one);
	\path (one) edge [loop above] 	node[label,above=3pt,inner sep=1pt,rounded corners,align=center]{\color{couleurLettre}\(\czer|\czer\)} (one);
	\path (two) edge [loop right] 	node[label,right=3pt,inner sep=1pt,rounded corners,align=center]{\color{couleurLettre}\(\cone|\czer\)} (two);
	\path (thr) edge [loop left] 		node[label,left=3pt,inner sep=1pt,rounded corners,align=center]{\color{couleurLettre}\(\cmon|\czer\)} (thr);
\end{scope}
\end{tikzpicture}
}
\caption{Two (dual) Mealy automata: the left-hand one computing the division by~3 in base~2 (most significant digit first)
\emph{vs} the right-hand one computing the multiplication by~2 in base~3 (least significant digit first).}%
\label{fig-avizienis}
\end{figure}

\medbreak
The Mealy automaton displayed on Figure~\ref{fig-avizienis} (left) is some signed-digit version of one of the most classical examples of a transducer
(see~\cite[Prologue]{saka} for a delightfully alternate history). Signed-digit numeration systems \cite{Avizienis61,Cauchy1840} are not the topic, now they provide a special opportunity to illustrate our purpose.
When starting from the state~$\color{couleurLettre}0$ and reading any binary word or any $\{\cmon,\czer,\cone\}$-word~$\word{u}$ (most significant digit first),
it computes the division by~3 in base~2  by outputting the (quotient) $\{\cmon,\czer,\cone\}$-word~$\word{v}$ (most significant digit first) satisfying
\begin{equation}
(\word{u})_2=3\times(\word{v})_2+f\label{eq-division}\tag{\lstinline{\%}}
\end{equation}
where the (remainder)~$f\in\{{\color{couleurLettre}\overline{1}},{\color{couleurLettre}0},{\color{couleurLettre}1}\}$ corresponds to the arrival state of the run,
and where $(\word{w})_b$ denotes by convention the number that is represented by~$\word{w}$ in base~$b$.

\medbreak
For the current preliminary section, let us now focus on this basic example and consider the two different viewpoints described as follows.
On the one hand, it seems natural to consider the set of those functions (from $\{\cmon,\czer,\cone\}$-words to $\{\cmon,\czer,\cone\}$-words) thus associated with each state,
then to compose them with each other, and finally to study the (semi)group which is generated by such functions.

For instance, the function~$\word{u}\mapsto\word{v}$ associated with the state~$\color{couleurLettre}0$ (satisfying Equation~\eref{eq-division}) can be squared, cubed, and so on,
to obtain functions, which can be again interpreted as the division by~$9, 27, \ldots$ (in base~2 with most significant digit first), or can be composed with the functions induced by the other two states. The generated semigroup happens to be the rank~3 free semigroup~$\{{\color{couleurLettre}\overline{1}},{\color{couleurLettre}0},{\color{couleurLettre}1}\}^+$
(provided that the three states and their induced functions are identified).
This simple idea coincides with the notion of automaton (semi)groups or self-similar structures (see Definition~\ref{def-automatonsem}).
With this crucial standpoint, we can compute (semi)group operations by manipulating the corresponding Mealy automaton (see~\cite{bartholdi:fr,BS10,KMP12,muntyan_s:automgrp}),
and hopefully foresee some combinatorial and dynamical properties by examining its shape
(see~\cite{spol,BZsolvable,bondarenko2_sz:conjugacy13,DGKPR16,gillibert:finiteness14,godin_kp:jir,klimann:finiteness,kpsDLT,kpsORB,SilvaSteinberg} for instance).

\medbreak
On the other hand, it may be also natural to simply iterate the runs. The starting language is again over the (input/output) alphabet,
now the images of the transformations are some languages over the stateset.

For instance, restarting again from the state~$\color{couleurLettre}0$,
the previously output word~$\word{v}$ (satisfying Equation~\eref{eq-division}) can be read in turn, and so on.
The successive arrival states can be then collected and concatenated in order to obtain here the decomposition 
of~$(\word{u})_2$ in base~3 (least significant digit first). The whole process is thus inherently quadratic.

This second idea coincides with the fundamental notion of automatic (semi)groups (see Definition~\ref{def-automaticsem}), for which Mealy automata can compute normal forms.

\medbreak We give some so-called \emph{cross-diagrams} on Figures~\ref{fig-cross-selfsim} and~\ref{fig-cross-thurston}. The ways these tilings can be organised and read illustrate the dual facets: self-similarity \emph{vs} automaticity.

\begin{figure}[h]
\begin{center}
\begin{tikzpicture}[thick,node distance=14mm]
\begin{scope}[xshift=0mm,yshift=0mm]
  \node (arrA1)									{};
  \node (arrA2) [node distance=83mm,right of=arrA1]		{};
  \draw (arrA1) edge[->,line width=12pt,>=triangle 90 cap,violet!30] (arrA2);
\end{scope}
\begin{scope}[xshift=-2.55mm,yshift=0mm]
  \node (arrA1)									{};
  \node (arrA2) [node distance=83mm,right of=arrA1]		{};
  \draw (arrA1) edge[->,line width=12pt,>=triangle 90 cap,violet!10] (arrA2);
\end{scope}
\begin{scope}[xshift=-5mm,yshift=0mm]
  \node (arrA1)									{};
  \node (arrA2) [node distance=83mm,right of=arrA1]		{};
  \draw (arrA1) edge[->,line width=12pt,>=triangle 90 cap,violet!30] (arrA2);
\end{scope}
\begin{scope}[xshift=-0mm,yshift=-42mm]
  \node (arrA1)									{};
  \node (arrA2) [node distance=83mm,right of=arrA1]		{};
  \draw (arrA1) edge[->,line width=12pt,>=triangle 90 cap,violet!30] (arrA2);
\end{scope}
\begin{scope}[xshift=-2.5mm,yshift=-42mm]
  \node (arrA1)									{};
  \node (arrA2) [node distance=83mm,right of=arrA1]		{};
  \draw (arrA1) edge[->,line width=12pt,>=triangle 90 cap,violet!10] (arrA2);
\end{scope}
\begin{scope}[xshift=-5mm,yshift=-42mm]
  \node (arrA1)									{};
  \node (arrA2) [node distance=83mm,right of=arrA1]		{};
  \draw (arrA1) edge[->,line width=12pt,>=triangle 90 cap,violet!30] (arrA2);
\end{scope}
  \node (gridA1) 			  {\(\czer\)};
  \node (gridA2) [right of=gridA1] {\(\czer\)};
  \node (gridA3) [right of=gridA2] {\(\cone\)};
  \node (gridA4) [right of=gridA3] {\(\czer\)};
  \node (gridA5) [right of=gridA4] {\(\cone\)};
  \node (gridA6) [right of=gridA5] {\(\cone\)};
  \node (gridB0) [below left of=gridA1,node distance=9.9mm,couleurLettre]{0};
  \node (gridB1) [right of=gridB0,couleurLettre] {0};
  \node (gridB2) [right of=gridB1,couleurLettre] {0};
  \node (gridB3) [right of=gridB2,couleurLettre] {1};
  \node (gridB4) [right of=gridB3,couleurLettre] {$\overline{1}$};%
  \node (gridB5) [right of=gridB4,couleurLettre] {$\overline{1}$};
  \node (gridB6) [right of=gridB5,couleurLettre] {$\overline{1}$};
  \node (gridC1) [below right of=gridB0,node distance=9.9mm]{\(\czer\)};
  \node (gridC2) [right of=gridC1] {\(\czer\)};
  \node (gridC3) [right of=gridC2] {\(\czer\)};
  \node (gridC4) [right of=gridC3] {\(\cone\)};
  \node (gridC5) [right of=gridC4] {\(\czer\)};
  \node (gridC6) [right of=gridC5] {\(\czer\)};
  
  \node (gridD0) [below left of=gridC1,node distance=9.9mm,couleurLettre]{$\overline{1}$};
  \node (gridD1) [right of=gridD0,couleurLettre]{1};
  \node (gridD2) [right of=gridD1,couleurLettre]{$\overline{1}$};
  \node (gridD3) [right of=gridD2,couleurLettre]{1};
  \node (gridD4) [right of=gridD3,couleurLettre]{0};%
  \node (gridD5) [right of=gridD4,couleurLettre]{0};
  \node (gridD6) [right of=gridD5,couleurLettre]{0};
  \node (gridE1) [below right of=gridD0,node distance=9.9mm]{\(\cmon\)};
  \node (gridE2) [right of=gridE1] {\(\cone\)};
  \node (gridE3) [right of=gridE2] {\(\cmon\)};
  \node (gridE4) [right of=gridE3] {\(\cone\)};
  \node (gridE5) [right of=gridE4] {\(\czer\)};
  \node (gridE6) [right of=gridE5] {\(\czer\)};

  \node (gridF0) [below left of=gridE1,node distance=9.9mm,couleurLettre]{1};
  \node (gridF1) [right of=gridF0,couleurLettre]{1};
  \node (gridF2) [right of=gridF1,couleurLettre]{0};
  \node (gridF3) [right of=gridF2,couleurLettre]{$\overline{1}$};
  \node (gridF4) [right of=gridF3,couleurLettre]{$\overline{1}$};
  \node (gridF5) [right of=gridF4,couleurLettre]{1};
  \node (gridF6) [right of=gridF5,couleurLettre]{$\overline{1}$};
  \node (gridG1) [below right of=gridF0,node distance=9.9mm]{\(\czer\)};
  \node (gridG2) [right of=gridG1] {\(\cone\)};
  \node (gridG3) [right of=gridG2] {\(\czer\)};
  \node (gridG4) [right of=gridG3] {\(\czer\)};
  \node (gridG5) [right of=gridG4] {\(\cmon\)};
  \node (gridG6) [right of=gridG5] {\(\cone\)};

  \begin{scope}[->,>=latex]
	\draw (gridA1) edge (gridC1);
	\draw (gridA2) edge (gridC2);
	\draw (gridA3) edge (gridC3);
	\draw (gridA4) edge (gridC4);
	\draw (gridA5) edge (gridC5);
	\draw (gridA6) edge (gridC6);
	\draw (gridB0) edge (gridB1);
	\draw (gridB1) edge (gridB2);
	\draw (gridB2) edge (gridB3);
	\draw (gridB3) edge (gridB4);
	\draw (gridB4) edge (gridB5);
	\draw (gridB5) edge (gridB6);
	\draw (gridC1) edge (gridE1);
	\draw (gridC2) edge (gridE2);
	\draw (gridC3) edge (gridE3);
	\draw (gridC4) edge (gridE4);
	\draw (gridC5) edge (gridE5);
	\draw (gridC6) edge (gridE6);
	\draw (gridD0) edge (gridD1);
	\draw (gridD1) edge (gridD2);
	\draw (gridD2) edge (gridD3);
	\draw (gridD3) edge (gridD4);
	\draw (gridD4) edge (gridD5);
	\draw (gridD5) edge (gridD6);
	\draw (gridE1) edge (gridG1);
	\draw (gridE2) edge (gridG2);
	\draw (gridE3) edge (gridG3);
	\draw (gridE4) edge (gridG4);
	\draw (gridE5) edge (gridG5);
	\draw (gridE6) edge (gridG6);
	\draw (gridF0) edge (gridF1);
	\draw (gridF1) edge (gridF2);
	\draw (gridF2) edge (gridF3);
	\draw (gridF3) edge (gridF4);
	\draw (gridF4) edge (gridF5);
	\draw (gridF5) edge (gridF6);
  \end{scope}
\end{tikzpicture}
\end{center}
\vspace*{-10pt}
\caption{A cross-diagram for self-similarity: the action induced by the word~$\color{couleurLettre}\czer\cmon\cone$ (left here) on the word~$\bzer\bzer\bone\bzer\bone\bone$ (top) gives the word~$\bzer\bone\bzer\bzer\bmon\bone$ (bottom). The actions can be composed, just as the cross-diagrams can be stacked. The self-similarity context requires actually to consider actions on right-infinite words.}
\label{fig-cross-selfsim}
\end{figure}

\begin{figure}[h]
\begin{center}
\begin{tikzpicture}[thick,node distance=14mm]
\begin{scope}[xshift=-8mm,yshift=0mm]
  \node (arrA1)									{};
  \node (arrA2) [node distance=83mm,right of=arrA1]		{};
  \draw (arrA2) edge[->,line width=12pt,>=triangle 90 cap,violet!30] (arrA1);
\end{scope}
\begin{scope}[xshift=77mm,yshift=-2mm]
  \node (arrA1)									{};
  \node (arrA2) [node distance=41mm,below of=arrA1]		{};
  \draw (arrA1) edge[->,line width=12pt,>=triangle 90 cap,violet!30] (arrA2);
\end{scope}
  \node (gridA1) 			  {\(\czer\)};
  \node (gridA2) [right of=gridA1] {\(\czer\)};
  \node (gridA3) [right of=gridA2] {\(\cone\)};
  \node (gridA4) [right of=gridA3] {\(\czer\)};
  \node (gridA5) [right of=gridA4] {\(\cone\)};
  \node (gridA6) [right of=gridA5] {\(\cone\)};
  \node (gridB0) [below left of=gridA1,node distance=9.9mm,couleurLettre]{0};
  \node (gridB1) [right of=gridB0,couleurLettre] {0};
  \node (gridB2) [right of=gridB1,couleurLettre] {0};
  \node (gridB3) [right of=gridB2,couleurLettre] {1};
  \node (gridB4) [right of=gridB3,couleurLettre] {$\overline{1}$};%
  \node (gridB5) [right of=gridB4,couleurLettre] {$\overline{1}$};
  \node (gridB6) [right of=gridB5,couleurLettre] {$\overline{1}$};
  \node (gridC1) [below right of=gridB0,node distance=9.9mm]{\(\czer\)};
  \node (gridC2) [right of=gridC1] {\(\czer\)};
  \node (gridC3) [right of=gridC2] {\(\czer\)};
  \node (gridC4) [right of=gridC3] {\(\cone\)};
  \node (gridC5) [right of=gridC4] {\(\czer\)};
  \node (gridC6) [right of=gridC5] {\(\czer\)};
  
  \node (gridD0) [below left of=gridC1,node distance=9.9mm,couleurLettre]{0};
  \node (gridD1) [right of=gridD0,couleurLettre]{0};
  \node (gridD2) [right of=gridD1,couleurLettre]{0};
  \node (gridD3) [right of=gridD2,couleurLettre]{0};
  \node (gridD4) [right of=gridD3,couleurLettre]{1};%
  \node (gridD5) [right of=gridD4,couleurLettre]{$\overline{1}$};
  \node (gridD6) [right of=gridD5,couleurLettre]{1};
  \node (gridE1) [below right of=gridD0,node distance=9.9mm]{\(\czer\)};
  \node (gridE2) [right of=gridE1] {\(\czer\)};
  \node (gridE3) [right of=gridE2] {\(\czer\)};
  \node (gridE4) [right of=gridE3] {\(\czer\)};
  \node (gridE5) [right of=gridE4] {\(\cone\)};
  \node (gridE6) [right of=gridE5] {\(\cmon\)};

  \node (gridF0) [below left of=gridE1,node distance=9.9mm,couleurLettre]{0};
  \node (gridF1) [right of=gridF0,couleurLettre]{0};
  \node (gridF2) [right of=gridF1,couleurLettre]{0};
  \node (gridF3) [right of=gridF2,couleurLettre]{0};
  \node (gridF4) [right of=gridF3,couleurLettre]{0};
  \node (gridF5) [right of=gridF4,couleurLettre]{1};
  \node (gridF6) [right of=gridF5,couleurLettre]{1};
  \node (gridG1) [below right of=gridF0,node distance=9.9mm]{\(\czer\)};
  \node (gridG2) [right of=gridG1] {\(\czer\)};
  \node (gridG3) [right of=gridG2] {\(\czer\)};
  \node (gridG4) [right of=gridG3] {\(\czer\)};
  \node (gridG5) [right of=gridG4] {\(\czer\)};
  \node (gridG6) [right of=gridG5] {\(\czer\)};

  \begin{scope}[->,>=latex]
	\draw (gridA1) edge (gridC1);
	\draw (gridA2) edge (gridC2);
	\draw (gridA3) edge (gridC3);
	\draw (gridA4) edge (gridC4);
	\draw (gridA5) edge (gridC5);
	\draw (gridA6) edge (gridC6);
	\draw (gridB0) edge (gridB1);
	\draw (gridB1) edge (gridB2);
	\draw (gridB2) edge (gridB3);
	\draw (gridB3) edge (gridB4);
	\draw (gridB4) edge (gridB5);
	\draw (gridB5) edge (gridB6);
	\draw (gridC1) edge (gridE1);
	\draw (gridC2) edge (gridE2);
	\draw (gridC3) edge (gridE3);
	\draw (gridC4) edge (gridE4);
	\draw (gridC5) edge (gridE5);
	\draw (gridC6) edge (gridE6);
	\draw (gridD0) edge (gridD1);
	\draw (gridD1) edge (gridD2);
	\draw (gridD2) edge (gridD3);
	\draw (gridD3) edge (gridD4);
	\draw (gridD4) edge (gridD5);
	\draw (gridD5) edge (gridD6);
	\draw (gridE1) edge (gridG1);
	\draw (gridE2) edge (gridG2);
	\draw (gridE3) edge (gridG3);
	\draw (gridE4) edge (gridG4);
	\draw (gridE5) edge (gridG5);
	\draw (gridE6) edge (gridG6);
	\draw (gridF0) edge (gridF1);
	\draw (gridF1) edge (gridF2);
	\draw (gridF2) edge (gridF3);
	\draw (gridF3) edge (gridF4);
	\draw (gridF4) edge (gridF5);
	\draw (gridF5) edge (gridF6);
  \end{scope}
\end{tikzpicture}
\end{center}
\vspace*{-10pt}
\caption{A cross-diagram for automaticity: according to Thurston, repeated runs starting from the unit state allows to collect the successives symbols of the normal form. Reading the word~$001011$ from the state~$\color{couleurLettre}\czer$ leads to the state~$\color{couleurLettre}\cmon$ and outputs the word~$000100$.
Restarting with the just-outputting word~$000100$ from the state~$\color{couleurLettre}\czer$ leads to the state~$\color{couleurLettre}\cone$ and outputs the word~$00001\overline{1}$. The runs can be iterated, just as the cross-diagrams can be stacked: the normal form appears on the side (right here). Considering the initial word~$001011$ (as some base~2 representation of eleven, most significant digit first), we find its normal form~$\color{couleurLettre}\cmon\cone\cone$ (in base~3, least significant digit first).}
\label{fig-cross-thurston}
\end{figure}

To conclude this preliminary section, let us mention that states and letters of any Mealy automaton play a symmetric role,
and that several properties can be beneficially derived from the so-called \emph{dual (Mealy) automaton}, obtained
by exchanging the stateset and the alphabet (see~\cite{cant} for an overview).

For instance, Figure~\ref{fig-avizienis} displays a pair of dual automata.
While the left-hand automaton allows to compute the division by~3 in base~2 (most significant digit first) as we have seen just above,
its dual automaton (right) essentially computes the multiplication by~2 in base~3 (least significant digit first).
More precisely, its state~$0$ induces the function~$x\mapsto 2\times x$ on~$\mathbb{Z}$,
while its states~$\overline{1}$ and~$1$ induce the functions~$x\mapsto 2\times x-1$ and~$x\mapsto 2\times x+1$ respectively:
they together generate the semigroup~$\press{~\overline{1},0,1:01=1\overline{1},0\overline{1}=\overline{1}1~}$.
Let us mention that the induced functions happen to be invertible and to generate a group
which is isomorphic with the so-called Baumslag--Solitar
group~$\BS(2,1)=\pres{~\tta,\ttb:\tta\ttb^2=\ttb\tta~}$ (see Figure~\ref{fig-picture}, Example~\ref{ex-bsThreeTwo}, and~\cite{pHdR}).
\nocite{BZsolvable}

Besides, such a Mealy automaton (right) can be used to compute the base~2 from the base~3 representation 
of the fractional part of any rational number, by iterating runs as explained above.
For instance, finitely iterated runs from the state~$0$ and the initial word~$\color{couleurLettre}00001\overline{1}$
produce the infinite word~$(0\overline{1}0010)^\omega$,
both words representing (the fractional part of) the rational~$-\frac{2}{9}$ 
in base~3 (least significant digit first) and in base~2 (most significant digit first) respectively,
see Figure~\ref{fig-cross-thurston-dual}.

\begin{figure}[h]
\begin{center}
\begin{tikzpicture}[thick,node distance=14mm]
\begin{scope}[xshift=-8mm,yshift=-42mm]
  \node (arrA1)									{};
  \node (arrA2) [node distance=83mm,right of=arrA1]		{};
  \node (arrA2bis) [node distance=1.5mm,left of=arrA2]		{};
  \node (arrA3) [node distance=13mm,right of=arrA2]		{};
  \draw (arrA2) edge[->,line width=12pt,>=triangle 90 cap,violet!30] (arrA1);
  \draw (arrA2bis) edge[line width=12pt,violet!30,dash pattern=on 1mm off 1mm] (arrA3);
\end{scope}
\begin{scope}[xshift=-7mm,yshift=-2mm]
  \node (arrA1)									{};
  \node (arrA2) [node distance=41mm,below of=arrA1]		{};
  \draw (arrA1) edge[->,line width=12pt,>=triangle 90 cap,violet!30] (arrA2);
\end{scope}
  \node (gridA1) 			  {\(\czer\)};
  \node (gridA2) [right of=gridA1] {\(\czer\)};
  \node (gridA3) [right of=gridA2] {\(\czer\)};
  \node (gridA4) [right of=gridA3] {\(\czer\)};
  \node (gridA5) [right of=gridA4] {\(\czer\)};
  \node (gridA6) [right of=gridA5] {\(\czer\)};
  \node (gridB0) [below left of=gridA1,node distance=9.9mm,couleurLettre]{0};
  \node (gridB1) [right of=gridB0,couleurLettre] {0};
  \node (gridB2) [right of=gridB1,couleurLettre] {0};
  \node (gridB3) [right of=gridB2,couleurLettre] {0};
  \node (gridB4) [right of=gridB3,couleurLettre] {0};%
  \node (gridB5) [right of=gridB4,couleurLettre] {0};
  \node (gridB6) [right of=gridB5,couleurLettre] {0};
  \node (gridC1) [below right of=gridB0,node distance=9.9mm]{\(\czer\)};
  \node (gridC2) [right of=gridC1] {\(\czer\)};
  \node (gridC3) [right of=gridC2] {\(\czer\)};
  \node (gridC4) [right of=gridC3] {\(\czer\)};
  \node (gridC5) [right of=gridC4] {\(\czer\)};
  \node (gridC6) [right of=gridC5] {\(\czer\)};
  
  \node (gridD0) [below left of=gridC1,node distance=9.9mm,couleurLettre]{1};
  \node (gridD1) [right of=gridD0,couleurLettre]{$\overline{1}$};
  \node (gridD2) [right of=gridD1,couleurLettre]{1};
  \node (gridD3) [right of=gridD2,couleurLettre]{$\overline{1}$};
  \node (gridD4) [right of=gridD3,couleurLettre]{1};%
  \node (gridD5) [right of=gridD4,couleurLettre]{$\overline{1}$};
  \node (gridD6) [right of=gridD5,couleurLettre]{1};
  \node (gridE1) [below right of=gridD0,node distance=9.9mm]{\(\cone\)};
  \node (gridE2) [right of=gridE1] {\(\cmon\)};
  \node (gridE3) [right of=gridE2] {\(\cone\)};
  \node (gridE4) [right of=gridE3] {\(\cmon\)};
  \node (gridE5) [right of=gridE4] {\(\cone\)};
  \node (gridE6) [right of=gridE5] {\(\cmon\)};

  \node (gridF0) [below left of=gridE1,node distance=9.9mm,couleurLettre]{$\overline{1}$};
  \node (gridF1) [right of=gridF0,couleurLettre]{$\overline{1}$};
  \node (gridF2) [right of=gridF1,couleurLettre]{0};
  \node (gridF3) [right of=gridF2,couleurLettre]{1};
  \node (gridF4) [right of=gridF3,couleurLettre]{1};%
  \node (gridF5) [right of=gridF4,couleurLettre]{0};
  \node (gridF6) [right of=gridF5,couleurLettre]{$\overline{1}$};
  \node (gridG1) [below right of=gridF0,node distance=9.9mm]{\(\czer\)};
  \node (gridG2) [right of=gridG1] {\(\cmon\)};
  \node (gridG3) [right of=gridG2] {\(\czer\)};
  \node (gridG4) [right of=gridG3] {\(\czer\)};
  \node (gridG5) [right of=gridG4] {\(\cone\)};
  \node (gridG6) [right of=gridG5] {\(\czer\)};

  \begin{scope}[->,>=latex]
	\draw (gridA1) edge (gridC1);
	\draw (gridA2) edge (gridC2);
	\draw (gridA3) edge (gridC3);
	\draw (gridA4) edge (gridC4);
	\draw (gridA5) edge (gridC5);
	\draw (gridA6) edge (gridC6);
	\draw (gridB0) edge (gridB1);
	\draw (gridB1) edge (gridB2);
	\draw (gridB2) edge (gridB3);
	\draw (gridB3) edge (gridB4);
	\draw (gridB4) edge (gridB5);
	\draw (gridB5) edge (gridB6);
	\draw (gridC1) edge (gridE1);
	\draw (gridC2) edge (gridE2);
	\draw (gridC3) edge (gridE3);
	\draw (gridC4) edge (gridE4);
	\draw (gridC5) edge (gridE5);
	\draw (gridC6) edge (gridE6);
	\draw (gridD0) edge (gridD1);
	\draw (gridD1) edge (gridD2);
	\draw (gridD2) edge (gridD3);
	\draw (gridD3) edge (gridD4);
	\draw (gridD4) edge (gridD5);
	\draw (gridD5) edge (gridD6);
	\draw (gridE1) edge (gridG1);
	\draw (gridE2) edge (gridG2);
	\draw (gridE3) edge (gridG3);
	\draw (gridE4) edge (gridG4);
	\draw (gridE5) edge (gridG5);
	\draw (gridE6) edge (gridG6);
	\draw (gridF0) edge (gridF1);
	\draw (gridF1) edge (gridF2);
	\draw (gridF2) edge (gridF3);
	\draw (gridF3) edge (gridF4);
	\draw (gridF4) edge (gridF5);
	\draw (gridF5) edge (gridF6);
  \end{scope}
\end{tikzpicture}
\end{center}
\vspace*{-10pt}
\caption{A second cross-diagram for automaticity: according to Thurston, repeated runs starting from the unit state allows to collect the successives symbols of the normal form. Reading the (vertical) word~$\color{couleurLettre}\czer\cone\cmon$ from the (top) state~$0$ leads to the state~$0$ and outputs the (vertical) word~$\color{couleurLettre}\czer\cmon\cmon$.
Restarting with the just-outputting word~$\color{couleurLettre}\czer\cmon\cmon$ from the state~$0$ leads to the state~$\overline{1}$ and outputs the word~$\color{couleurLettre}\czer\cone\czer$. The runs can be iterated until completing a loop: the (eventually periodic) normal form appears below. Considering the initial word~$\color{couleurLettre}\czer\cone\cmon$ (as some base~2 representation of~$-\frac{2}{9}$, least significant digit first), we find its normal form~$(0\overline{1}0010)^\omega$ (in base~3, most significant digit first).}
\label{fig-cross-thurston-dual}
\end{figure}

\bigbreak
This innocuous example allows to illustrate the quite simple machineries associated both with automaton semigroups and with automatic semigroups.
It also aims to give an informal glimpse of their behaviours through the duality principle: for instance, division \emph{vs} multiplication, factor \emph{vs}
base, least \emph{vs} most significant digit first, integer part \emph{vs} fractional part.

\newpage
\section{Mealy automata and self-similar structures}\label{sec-mealy}

We first recall the formal definition of an automaton. Possible references are~\cite{nekrash:self-similar,BS10,Cain,McCune}.

\begin{definition}\label{def-mealy}
A {\em (finite, deterministic, and complete) automaton} is a
triple \(\bigl( Q,\Sigma,\tau = (\tau_i\colon Q\rightarrow Q )_{i\in \Sigma} \bigr)\),
where the \emph{stateset}~$Q$
and the \emph{alphabet}~$\Sigma$ are non-empty finite sets, and
where the~\(\tau_i\)'s
are functions.

A \emph{Mealy automaton} is a quadruple
\(\bigl( Q, \Sigma, \tau = (\tau_i\colon Q\rightarrow Q )_{i\in \Sigma},
\sigma = (\sigma_x\colon \Sigma\rightarrow \Sigma  )_{x\in Q} \bigr)\)
such that both~\((Q,\Sigma,\tau)\) and~\((\Sigma,Q,\sigma)\) are
automata.
\end{definition}

In other terms, a Mealy automaton is a complete, deterministic,
letter-to-letter transducer with the same input and output alphabet.

The graphical representation of a Mealy automaton is
standard, see Figures~\ref{fig-avizienis}, \ref{fig-bsOneZero}, and \ref{fig-braid}.

In a Mealy automaton $\aut{A}=(Q,\Sigma, \tau, \sigma)$, the sets $Q$ and
$\Sigma$ play dual roles. So we may consider the \emph{dual (Mealy)
  automaton} defined by
\(\dz(\aut{A}) = (\Sigma,Q, \sigma, \tau)\):
\begin{center}
\begin{tikzpicture}[thick,node distance=13mm]
	\node[draw,circle,text width=3mm,align=center,inner sep=1.2pt] (quad-n-transi-left) at (-1,-.8) {$x$};
	\node[draw,circle,text width=3mm,align=center,inner sep=1.2pt,node distance=18mm] (quad-n-transi-right) [right of=quad-n-transi-left] {$y$};
	\path[->,>=latex] (quad-n-transi-left) edge[bend left] node[align=center,above]{\(\!i\,|\,j\!\)} (quad-n-transi-right);
	\node (eq) at (1.7,-.8) {$ \in \aut{A}$};
	\node (eq) at (3.3,-.8) {$\iff$};
	\node (eq) at (8,-.8) {$ \in \dz(\aut{A})$.};
	\node[draw,circle,text width=3mm,align=center,inner sep=1.2pt] (quad-n-transi-left) at (5,-.8) {$i$};
	\node[draw,circle,text width=3mm,align=center,inner sep=1.2pt,node distance=18mm] (quad-n-transi-right) [right of=quad-n-transi-left] {$j$};
	\path[->,>=latex] (quad-n-transi-left) edge[bend left] node[align=center,above]{\(\!x\,|\,y\!\)} (quad-n-transi-right);
\end{tikzpicture}
\end{center}
We view~\(\aut{A}=(Q,\Sigma,\tau,\sigma)\) as an automaton with an input and an output tape, thus
defining mappings from input words over~$\Sigma$ to output words
over~$\Sigma$.
Formally, for~\(x\in Q\), the map~$\sigma_x\colon\Sigma^* \rightarrow \Sigma^*$,
extending~$\sigma_x\colon\Sigma \rightarrow \Sigma$, is defined recursively by:
\begin{equation*}\label{eq-rec-def}
\forall i \in \Sigma, \ \forall \word{s} \in \Sigma^*, \qquad
\sigma_x(i\word{s}) = \sigma_x(i)\sigma_{\tau_i(x)}(\word{s}) \:.
\end{equation*}

The above equation can be easier to understood when depicted by a \emph{cross-diagram} (see~\cite{AKLMP12}):
\begin{center}
\begin{tikzpicture}[thick,node distance=16mm]
	\node (01) at (.75,2) {$i$};
	\node (03) [node distance=22mm,right of=01] {$\word{s}$};
	\node (05) [node distance=20mm,right of=03] {};
	\node (10) at (0,1.25) {$x$};
	\node (12) [node distance=19mm,right of=10] {$\tau_i(x)$};
	\node (14) [node distance=25mm,right of=12] {$\tau_{\word{s}}(\tau_i(x))$};
	\node (21) [below of=01] {$\sigma_x(i)$};
	\node (23) [node distance=22mm,right of=21] {$ \sigma_{\tau_i(x)}(\word{s})$};
	\path[->,>=latex]
		(01)	edge  (21)
		(10)	edge  (12)
		(03)	edge  (23)
		(12)	edge  (14);
\end{tikzpicture}
\end{center}
By convention, the image of the empty word is itself. 
The mapping~\(\sigma_x\) for each~$x\in Q$ is length-preserving and prefix-preserving.
We say that~\(\sigma_x\) is the \emph{production
function\/} associated with~\((\aut{A},x)\).
For~$\word{x}=x_1\cdots x_n \in Q^n$ with~$n>0$, set
\(\sigma_\word{x}\colon\Sigma^* \rightarrow \Sigma^*, \sigma_\word{x} = \sigma_{x_n}
\circ \cdots \circ \sigma_{x_1} \:\).
Denote dually by~\(\tau_i\colon Q^*\rightarrow Q^*,
i\in \Sigma\), the production functions associated with
the dual automaton
$\dz(\aut{A})$. For~$\word{s}=s_1\cdots s_n
\in \Sigma^n$ with~$n>0$, set~\(\tau_\word{s}\colon Q^* \rightarrow Q^*,
\ \tau_\word{s} = \tau_{s_n}\circ \cdots \circ \tau_{s_1}\).

\begin{definition}\label{def-automatonsem} The semigroup of mappings from~$\Sigma^*$ to~$\Sigma^*$ generated by
$\{\sigma_x, x\in Q\}$ is called the \emph{semigroup generated
  by~$\aut{A}$} and is denoted by~$\press{~\aut{A}~}$.
When~\(\aut{A}\) is invertible,
its production functions are
permutations on words of the same length and thus we may consider
the corresponding group instead; this group is the \emph{group generated
by~$\aut{A}$} and is denoted by~$\pres{~\aut{A}~}$.
A (semi)group is called an \emph{automaton (semi)group} whenever it can be generated by some Mealy automaton.
The term \emph{self-similar} is used as a synonym.
\end{definition}

\DeclareRobustCommand{\home}{{\rm(\scalebox{.06}{\rotatebox{-15}{
 \begin{tikzpicture}[thick]
	\fill[color=lightgray] (0,0) -- (-3,1.5) -- (-3,3) -- (0,4.5) -- (0,0);
	\fill[color=lightgray] (0,0) -- (3,3) -- (0,4.5) -- (0,0);
\end{tikzpicture}
}})}}

\DeclareRobustCommand{\unit}{{\rm\hspace*{1pt}\raisebox{-.3ex}{\scalebox{.08}{
 \begin{tikzpicture}[line width=7pt]
	\draw (0,3)	edge[color=lightgray,line width=3pt,densely dotted]
				node[text=black,fill=white,scale=7]{\hspace*{-3.3pt}$\ttunit$\hspace*{-3.3pt}}	(3,0);
	\draw[color=lightgray] (0,3) -- (3,3) -- (3,0) -- (0,0) -- (0,3)-- (3,3);
	\draw[color=black,cap=round] (3,0) -- (3,3);
\end{tikzpicture}
}}\hspace*{1pt}}}
\DeclareRobustCommand{\cunit}{{\rm(\hspace*{1pt}\raisebox{-.3ex}{\scalebox{.08}{
 \begin{tikzpicture}[line width=7pt]
	\draw (0,3)	edge[color=lightgray,line width=3pt,densely dotted]
				node[text=black,fill=white,scale=7]{\hspace*{-3.3pt}$\ttunit$\hspace*{-3.3pt}}	(3,0);
	\draw[color=lightgray] (0,3) -- (3,3) -- (3,0) -- (0,0) -- (0,3)-- (3,3);
	\draw[color=black,cap=round] (3,0) -- (3,3);
\end{tikzpicture}\
}}\hspace*{1pt})}}

\DeclareRobustCommand{\squa}{{\rm\scalebox{.8}{$\blacksquare$}}}
\DeclareRobustCommand{\sqtwo}{{\rm\raisebox{-.3ex}{\scalebox{.065}{
 \vspace*{-5pt}
 \begin{tikzpicture}[thick]
	\fill[color=black] (0,0) -- (4,0) -- (4,4) -- (0,4) -- (0,0);
	\node[color=white,scale=6] at (2,2) {\LARGE2};
\end{tikzpicture}
}}}}
\DeclareRobustCommand{\sqk}{{\rm\raisebox{-.3ex}{\scalebox{.065}{
 \vspace*{-5pt}
 \begin{tikzpicture}[thick]
	\fill[color=black] (0,0) -- (4,0) -- (4,4) -- (0,4) -- (0,0);
	\node[color=white,scale=6] at (2,2) {\LARGE k};
\end{tikzpicture}
}}}}

\section{Quadratic normalisations and automatic structures}\label{sec-garside}

This section gathers the definitions of some classical notions like normal form or automatic structure (see \cite{EpsteinWord,CRRT,HT}),
together with the slighly more specific notion of a quadratic normalisation (see~\cite{dehDLT,dehgui})
and a Garside family (see~\cite{dehDLT,dehornoyTheory}).

For any set~$\Gar$, we denote by~$\Gar^+$ the free semigroup over~$\Gar$ (\resp by~$\Gar^*$ the free monoid and by~$\ttunit$ its unit element) and call its elements \emph{$\Gar$-words}.
We write $\card{\word{w}}$ for the length of a $\Gar$-word~$\word{w}$, and $\word{w}\word{w}'$ for the product of two $\Gar$-words~$\word{w}$ and~$\word{w}'$.

\begin{definition}\label{def-automaticsem} Let~$\MM$ be a semigroup with a generating set~$\Gar$. 
A \emph{normal form} for~$(\MM,\Gar)$ is a (set-theoretic) section
of the canonical projection~$\EV$ from the language of $\Gar$-words onto~$\MM$,
that is, a map~$\NF$ that assigns to each element of~$\MM$ a~distinguished representative $\Gar$-word
with~$\EV\circ\NF=\ID_S$:
\begin{center}
\vspace*{-1pt}
\begin{tikzpicture}[scale=1,inner sep=2pt]
	\node  (EV) at (0,0) {$\EV:\Gar^+$};
	\node  (S) at (2.2,0) {$\MM$};
	\path[->,thick]
		(EV)	edge[thin,black,scale=2] (S);
	\path[->,thick]
		(S)	edge[thin,black!70,bend left=40] node[below=.08,black]{$\NF$} (EV.south east);
\end{tikzpicture}
\end{center}
Whenever $\NF(\MM)$ is regular, it provides a \emph{right-automatic structure} for~$\MM$
if the language ${\mathcal L}_q=\{~(\ \NF(a)\#^{\card{\NF(aq)}},\ \NF(aq)\#^{\card{\NF(a)}}\ ):a\in\MM~\}$ over the alphabet~$(\Gar\sqcup\{\#\})^2$ is regular for each~$q\in\Gar$,
where the normal forms of a pair are right-padded with an extra symbol~$\#\not\in\Gar$ to equalise the lengths.
The semigroup~$\MM$ can then be called a \emph{(right-)automatic semigroup}.
\end{definition}
We mention here the thorough and precious study in~\cite{HoffmannPhD} of the different notions (right- or left-reading-padding \emph{vs} right- or left-multiplication) of automaticity for semigroups.

\begin{remark}\label{rem-thurston}In his seminal work~\cite[Chapter~9]{EpsteinWord},
Thurston shows how the whole set of these different automata recognizing
the multiplication---that is, recognizing the languages~${\mathcal L}_q$
---in Definition~\ref{def-automaticsem}
can be replaced with advantage by a single letter-to-letter transducer over the alphabet~$\Gar$ (see Definition~\ref{def-thurston})
that computes the normal forms via iterated runs:
each run both provides one symbol of the final normal form and outputs a word still to be normalised.
\end{remark}

One will often consider the associated normalisation~$\NM=\NF\circ\EV$ over~$\Gar$.

\begin{definition}\label{def-norm} A \emph{normalisation} is a pair~$(\Gar,\NM)$, where~$\Gar$ is a set and~$\NM$ is a map from~$\Gar^+$ to itself satisfying, for all $\Gar$-words~$\word{u},\word{v},\word{w}$:
\begin{itemize}
\item $\card{\NM(\word{w})} = \card{\word{w}}$,
\item $\card{\word{w}} = 1 \Rightarrow \NM(\word{w}) = \word{w}$,
\item $\NM(\word{u}\,\NM(\word{w})\,\word{v})=\NM(\word{u}\word{w}\word{v})$.
\end{itemize}
A $\Gar$-word~$\word{w}$ satisfying $\NM(\word{w})=\word{w}$ is called \emph{$\NM$-normal}.
If~$\MM$ is a semigroup, we say that $(\Gar, \NM)$ is a \emph{normalisation for~$\MM$} if~$\MM$ admits the presentation\begin{equation*}
\press{~\Gar:\{\,\word{w} = \NM(\word{w}) \mid \word{w} \in \Gar^+\,\}~}.
\end{equation*}
\end{definition}

We associate with every element~$q\in\Gar$ a $q$-labeled edge
and with a product the concatenation of the associated edges, 
and represent equalities in the ambient semigroup using commutative diagrams, that we shall often
organise as tilings and that we call here \emph{square-diagram}. For instance, the following square illustrates an equality $q_1q_2 =q_1'q_2'$.
\begin{center}
\begin{tikzpicture}[thick,node distance=10mm]
	\node (00) at (0,0) {};
	\node (01) [right of=00] {};
	\node (10) [below of=00] {};
	\node (11) [right of=10] {};
	\path[->,>=latex]
		(01)	edge 		node[above]{$q_1$}	(00)
		(00)	edge 		node[left]{$q_2$}	(10)
		(01)	edge 		node[right]{$q_1'$}	(11)
		(11)	edge 		node[below]{$q_2'$}	(10);
\end{tikzpicture}
\end{center}

For a normalisation~$(\Gar,\NM)$, we denote by~$\NMbar$ the restriction of~$\NM$ to~$\Gar^2$ and, for~$i\ge 1$, by~$\NMbar_i$ the (partial) map from~$\Gar^+$ to itself that consists in applying~$\NMbar$ to the entries in position~$i$ and~$i+1$. For any finite sequence
$\word{i}= i_1 \cdots i_n$ of positive integers, we write~$\NMbar_{\word{i}}$ for the composite map $\NMbar_{i_{n}}\comp\cdots\comp \NMbar_{i_1}$ (so $\NMbar_{i_1}$ is applied first).

\begin{definition}\label{def-quad}A normalisation $(\Gar,\NM)$ is \emph{quadratic} if the two following conditions hold:
\begin{itemize}
\item a $\Gar$-word~$\word{w}$ is $\NM$-normal if, and only if, every length-two factor of~$\word{w}$ is;
\item for every $\Gar$-word~$\word{w}$, there exists a finite sequence~$\word{i}$ of positions, depending on~$\word{w}$,
such that $\NM(\word{w})$ is equal to~$\NMbar_\word{i}(\word{w})$.
\end{itemize}
\end{definition}

\begin{definition}\label{def-breadth}As illustrated in~Figure~\ref{fig-breadth},
with any quadratic normalisation~$(\Gar,\NM)$ is associated its \emph{breadth}~$(d,p)$
(called minimal left and right classes in~\cite{dehDLT,dehgui}) defined as:%
\[d=\max_{(q_1,q_2,q_3)\in\Gar^3}\min\{\,\ell:\NM(q_1q_2q_3)=\NMbar_{\underbrace{\scriptstyle212\cdots}_{{\rm length}\, \ell}}(q_1q_2q_3)\},\] and
\[p=\max_{(q_1,q_2,q_3)\in\Gar^3}\min\{\,\ell:\NM(q_1q_2q_3)=\NMbar_{\underbrace{\scriptstyle121\cdots}_{{\rm length}\, \ell}}(q_1q_2q_3)\}.\]%
Such a breadth need to be finite provided that~$\Gar$ is finite, and then satisfies~$|d-p|\leq 1$.
For~$p\leq 3$ (and~$d\leq 4$), the quadratic normalisation~$(\Gar,\NM)$ is said to satisfy Condition~\home\
(its corresponds with the so-called \emph{domino rule} in~\cite{dehornoyTheory} but with a different reading direction).
\end{definition}

\begin{figure}
\vspace*{-45pt}
\centering
 \rotatebox{-15}{
 \begin{tikzpicture}[thick,node distance=15mm,inner sep=1.3pt]
	\node[inner sep=2pt] (00) at (0,0) {~~~~};
	\node[inner sep=4pt] (00) at (0,0) {};
	\node (10) [above of=00] {};
	\node (11) [left of=10] 	{};
	\node (12) [right of=10] 	{};
	\node (13) [left of=11] 	{};
	\node (14) [right of=12] 	{};
	\node (15) [left of=13] 	{};
	\node (16) [right of=14] 	{};
	\node (17) [node distance=7.5mm,left of=15] 	{$\cdots$};
	\node (18) [node distance=7.5mm,right of=16] 	{$\cdots$};
	\node (20) [above of=10] {};
	\node (21) [above of=11] {};
	\node (11b) [node distance=13mm,left of=10] {};
	\node (21b) [node distance=13mm,left of=20] {};
	\node (22) [above of=12] {};
	\node (12b) [node distance=13mm,right of=10] {};
	\node (22b) [node distance=13mm,right of=20] {};
	\node (23) [above of=13] {};
	\node (13b) [node distance=11mm,left of=11] {};
	\node (23b) [node distance=11mm,left of=21] {};
	\node (24) [above of=14] {};
	\node (14b) [node distance=13mm,right of=12] {};
	\node (24b) [node distance=11mm,right of=22] {};
	\node (25) [above of=15] {};
	\node (15b) [node distance=13mm,left of=13] {};
	\node (25b) [node distance=13mm,left of=23] {};
	\node (26) [above of=16] {};
	\node (16b) [node distance=13mm,right of=14] {};
	\node (26b) [node distance=13mm,right of=24] {};
	\node (27) [above of=17] {$\cdots$};
	\node (28) [above of=18] {$\cdots$};
	\node (07) [node distance=7.5mm,above of=17] {$\cdots$};
	\node (007) [rotate=15,below left=13mm and -50pt of 07] {$\NMbar_{\scriptstyle212\cdots}(q_1q_2q_3)$};
	\node (08) [node distance=7.5mm,above of=18] {$\cdots$};
	\node (008) [rotate=15,above right=13mm and -50pt of 08] {$\NMbar_{\scriptstyle121\cdots}(q_1q_2q_3)$};
	\node (30) [above of=20] {};
	\fill[color=lightgray] (0,0) -- (-3,1.5) -- (-3,3) -- (0,4.5) -- (0,0);
	\fill[color=lightgray] (0,0) -- (3,3) -- (0,4.5) -- (0,0);
	\draw[color=white,line width=9pt] (0,0) -- (-3,1.5) -- (-3,3) -- (0,4.5) -- (0,0);
	\draw[color=white,line width=9pt] (0,0) -- (3,3) -- (0,4.5) -- (0,0);
	\draw[color=white,line width=9pt] (0,1.5) -- (1.5,3);
	\draw[color=white,line width=9pt] (1.5,1.5) -- (1.5,3) -- (0,4.5);
	\draw[color=white,line width=9pt] (-3,1.5) -- (-1.5,3) -- (0,4.5);
	\draw[color=white,line width=9pt] (-1.5,1.5) -- (0,3);
	\draw[color=white,line width=9pt] (0,0) -- (-1.5,1.5) -- (-1.5,3);
	\path[->,>=latex,line width=1.4pt]	(30)	edge 		node[pos=.42,fill=white,rotate=15]{\footnotesize$\!q_1\!$} (20);
	\path[->,>=latex,line width=1.4pt]	(20)	edge			node[pos=.42,fill=white,rotate=15]{\footnotesize$\!q_2\!$} (10);
	\path[->,>=latex,line width=1.4pt]	(10)	edge			node[pos=.42,fill=white,rotate=15]{\footnotesize$\!q_3\!$} (00);
	\path[->,>=latex] 	(30)	edge			(21);
	\path[->,>=latex]	(21)	edge			(11);
	\path[->,>=latex]	(11)	edge			(00);
	\path[->,>=latex]	(20)	edge			(11);
	\path[->,>=latex]	(30)	edge			(22);
	\path[->,>=latex]	(22)	edge			(10);
	\path[->,>=latex]	(12)	edge			(00);
	\path[->,>=latex]	(22)	edge			(12);
	\path[->,>=latex]	(21)	edge			(13);
	\path[->,>=latex]	(13)	edge			(00);
	\path[->,>=latex]	(30)	edge			(23);
	\path[->,>=latex]	(23)	edge			(13);
	\path[->,>=latex]	(30)	edge			(24);
	\path[->,>=latex]	(24)	edge			(12);
	\path[->,>=latex]	(24)	edge			(14);
	\path[->,>=latex]	(14)	edge			(00);
	\path[->,>=latex]	(23)	edge			(15);
	\path[->,>=latex]	(15)	edge			(00);
	\path[->,>=latex]	(30)	edge			(25);
	\path[->,>=latex]	(25)	edge			(15);
	\path[->,>=latex]	(30)	edge			(26);
	\path[->,>=latex]	(26)	edge			(14);
	\path[->,>=latex]	(26)	edge			(16);
	\path[->,>=latex]	(16)	edge			(00)
	(10)	edge[densely dotted,white] 		node[fill=lightgray]{$\!\NM\!$}	(11b)
	(11)	edge[densely dotted,white] 		node[fill=lightgray]{$\!\NM\!$}	(13b)
	(10)	edge[densely dotted,white] 		node[fill=lightgray]{$\!\NM\!$}	(12b)
	(12)	edge[thin,densely dotted,gray] 		node[fill=white]{$\!\NM\!$}	(14b)
	(13)	edge[thin,densely dotted,gray] 		node[fill=white]{$\!\NM\!$}	(15b)
	(14)	edge[thin,densely dotted,gray] 		node[fill=white]{$\!\NM\!$}	(16b)
	(20)	edge[densely dotted,white] 		node[fill=lightgray]{$\!\NM\!$}	(21b)
	(21)	edge[densely dotted,white] 		node[fill=lightgray]{$\!\NM\!$}	(23b)
	(20)	edge[densely dotted,white] 		node[fill=lightgray]{$\!\NM\!$}	(22b)
	(22)	edge[densely dotted,white] 		node[fill=lightgray]{$\!\NM\!$}	(24b)
	(23)	edge[thin,densely dotted,gray] 		node[fill=white]{$\!\NM\!$}	(25b)
	(24)	edge[thin,densely dotted,gray] 		node[fill=white]{$\!\NM\!$}	(26b);
\end{tikzpicture}
}
\vspace*{-50pt}
\caption{From an initial $\Gar$-word~$q_1q_2q_3$, one applies normalisations on the first and the second~$2$-factors alternatively up to stabilisation,
beginning either on the first $2$-factor~$q_1q_2$ (right-hand side here) or on the second~$q_2q_3$. The gray zone corresponds to Condition~\home\ as defined in Definition~\ref{def-breadth}.
}%
\label{fig-breadth}
\end{figure}
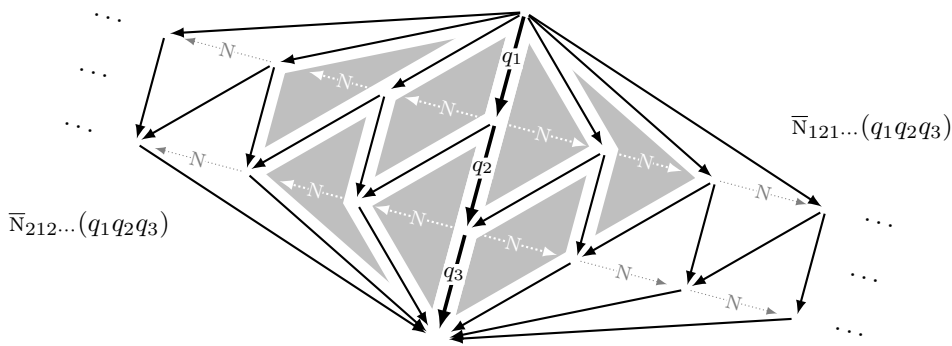

The first main result of~\cite{dehgui} is an axiomatisation of these quadratic normalisations satisfying Condition~\home\ in terms of their restrictions to length-two words: any idempotent map~$\NMbar$ on~$\Gar^2$ that satisfies $\NMbar_{2121} = \NMbar_{121} =\NMbar_{1212}$ extends into a quadratic normalisation~$(\Gar,\NM)$ satisfying Condition~\home.
For larger breadths, a map on length-two words normalising length-three words needs not normalise words of greater length.

The second main result of~\cite{dehgui} involves termination. Every quadratic normalisation~$(\Gar,\NM)$ gives rise to a quadratic rewriting system, namely the one with rules $\word{w}\longrightarrow\NMbar(\word{w})$ for~$\word{w}\in\Gar^2$. By Definition~\ref{def-quad}, such a rewriting system is confluent and normalising, meaning that, for every initial word, there exists a finite sequence of rewriting steps leading to a unique $\NM$-normal word, but its convergence, meaning that \emph{any} sequence of rewriting steps is finite, is a quite different problem.

\begin{theorem}{\rm\cite{dehgui}}\label{thm-termination}
If $(\Gar,\NM)$ is a quadratic normalisation satisfying Condition~\home, then the associated rewriting system is convergent.
\end{theorem}

More precisely, every rewriting sequence starting from a word of~$\Gar^p$ has length at most~$\frac{\pp(\pp-1)}{2}$ (\resp $2^\pp - \pp - 1$)
in the case of a breadth~$(3,3)$ (\resp either~$(3,4)$ or~$(4,3)$).
Theorem~\ref{thm-termination} is essentially optimal since there exist nonconvergent rewriting systems with breadth~$(4,4)$.

\bigbreak
The rest of the current section describes a tiny fragment of Garside theory (see \cite{dehornoyTheory} for its foundations).
Garside families were recently introduced as a general framework guaranteeing the existence of normal forms.
While this notion is not necessary for the understanding of the main result, its proof, and the whole of Section~\ref{sec-main},
several examples of~Section~\ref{sec-ex} could rely on it.

\medbreak
Let $\MM$ be a monoid. For~$\ff,\gg,\hh\in\MM$, $\ff$ is a \emph{left-divisor} of~$\gg$ or, equivalently, $\gg$ is a \emph{right-multiple} of~$\ff$ if $\gg = \ff\gg'$ holds for some~$\gg'$ in~$\MM$;
moreover, $\gg$ is a minimal common right-multiple, or \emph{right-mcm}, of~$\ff$ and~$\hh$
if $\gg$ is a right-multiple of~$\ff$ and~$\hh$, and no proper left-divisor of~$\gg$ is a right-multiple of~$\ff$ and~$\hh$.

Furthermore, $\MM$ is said to be \emph{right-cancellative} whenever, for all~$\ff,\gg,\hh\in\MM$, $\ff\gg=\hh\gg$ implies $\ff=\hh$,
and $\MM$ admits no nontrivial \emph{invertible element} whenever $\ff\gg = \ttunit$ implies~$\ff = \gg = \ttunit$.
\emph{Right-divisor}, \emph{left-mcm}, and \emph{left-cancellativity} are defined symmetrically.

\begin{definition}\label{def-family} If $\MM$ is a right-cancellative monoid with no nontrivial invertible element,
a \emph{(right-)Garside family} for~$\MM$ is a generating set closed under left-divisor and under left-mcm.
\end{definition}

Various practical characterisations of Garside families are known, depending in particular on the specific properties of the considered monoid.
The following is especially relevant here.

\begin{theorem}{\rm\cite{dehgui}}\label{thm-garhome}
Assume that $\MM$ is a right-cancellative monoid with no nontrivial invertible element and with a finite (right-)Garside family~$\Gar$.
Then the normalisation $(\Gar,\NM)$ defined by~$\NM(ab)=cd$ for~$a,b,c,d\in\Gar$
with~$d$ maximal, satisfies Condition~\home.
\end{theorem}

This characterisation would help to handle tiny cases like Example~\ref{ex-bsOneZero},
now it will reveal all of its strength for Examples~\ref{ex-akram}, \ref{ex-bsThreeTwo}, or \ref{ex-braid}:
dedicated procedures allow to compute in a trice the closures displayed on Figures~\ref{fig-akram} and~\ref{fig-bsThreeTwo}.
However, Examples~\ref{ex-bicyclic} and~\ref{ex-plactic} are out of its range,
and, for such profiles, new efficient tools are to be built (see~\cite{darla}). Whatever the way a quadratic normalisation is obtained (brute force, Garside theory, Knuth--Bendix completion, etc), to compute its breadth and to check Condition~\home\ remain low-cost.

\medbreak 

The results of Section~\ref{sec-main} rely on the special Condition~\home. As mentioned, this condition was already
outlined by Dehornoy and Guiraud (see~\cite{dehgui}).
However, none of their results (in particular Theorem~\ref{thm-termination} and~\ref{thm-garhome}
given above for the sake of completeness) is either applied or needed to establish ours.
The current work and its exposition are thus self-contained and our constructions never require any of their stronger hypotheses
(neither cancellativity nor absence of nontrivial invertible elements). We want here to emphasise that Condition~\home\ happens to appear as a common denominator from different approaches.

\section{From an automatic structure to a self-similar structure}\label{sec-main}

All the ingredients are now in place to effectively and naturally interpret
as an automaton monoid any automatic monoid
admitting a special language of normal forms---namely,
a quadratic normalisation satisfying Condition~\home.
The point is to construct a Mealy automaton encoding the behaviour
of its language of normal forms under one-sided multiplication.

\begin{definition}\label{def-msqn}
Assume that~$\MM$ is a semigroup admitting a quadratic normalisation~$(\Gar, \NM)$. 
We define the Mealy automaton~$\Mea_{\MM,\Gar,\NM}=(\Gar,\Gar,\tau,\sigma)$
such that, for every~$(a,b)\in\Gar^2$, $\sigma_b(a)$ is the rightmost element of~$\Gar$
in the normal form~$\NM(ab)$ of~$ab$ and $\tau_a(b)$ is the left one:
\[\NM(ab)=\tau_a(b)\sigma_b(a).\]
\end{definition}

The latter correspondence can be simply interpreted via square-diagram \emph{vs} cross-diagram:
\begin{center}
\begin{tikzpicture}[thick,node distance=13mm]
	\node (00) at (0,2) {};
	\node (01) [right of=00] {};
	\node (10) [below of=00] {};
	\node (11) [right of=10] {};
	\path[->,>=latex]
		(01)	edge 		node[above]{$a$}	(00)
		(00)	edge 		node[left]{$b$}		(10)
		(01)	edge 		node[right]{$\tau_a(b)$}	(11)
		(11)	edge 		node[below]{$\sigma_b(a)$}	(10)
		(00)	edge[thin,densely dotted,black!50] 		node[fill=white]{$\!\NM\!$}	(11);
	\begin{scope}[thick,node distance=16mm,xshift=50mm]
	\node (01) at (.75,2) {$a$};
	\node (10) at (0,1.25) {$b$};
	\node (12) [node distance=18mm,right of=10] {$\tau_a(b)$};
	\node (21) [below of=01] {$\sigma_b(a)$};
	\path[->,>=latex]
		(01)	edge  (21)
		(10)	edge  (12);
	\end{scope}
\end{tikzpicture}
\end{center}

For~$\word{s}=s_n\cdots s_1$, $\word{t}=t_n\cdots t_1$, $\sigma_q(\word{s})=\word{t}$, and~$\tau_\word{s}(q)=r$,
we obtain diagrammatically:
\begin{center}
\begin{tikzpicture}[thick,node distance=15mm]
	\node (00) at (0,2) {};
	\node (01) [right of=00] {};
	\node (02) [right of=01] {};
	\node (03) [right of=02] {};
	\node (04) [right of=03] {};
	\node (10) [below of=00] {};
	\node (11) [right of=10] {};
	\node (12) [right of=11] {};
	\node (13) [right of=12] {};
	\node (14) [right of=13] {};
	\path[->,>=latex]
		(01)	edge 		node[above]{$\ss_1$}	(00)
		(02)	edge 		node[above]{$\ss_2$}	(01)
		(04)	edge 		node[above]{$\ss_n$}	(03)
		(00)	edge 		node[left]{$q$}		(10);
	\path[->,>=latex,black!50]
		(00)	edge[thin,densely dotted] 		node[fill=white]{$\!\NM\!$}	(11)
		(01)	edge[thin,densely dotted] 		node[fill=white]{$\!\NM\!$}	(12)
		(03)	edge[thin,densely dotted] 		node[fill=white]{}	(14)
		(01)	edge 		node[right]{$q_1$}	(11)
		(02)	edge 		node[right]{$q_2$}	(12)
		(03)	edge 		node[right]{$q_{n-1}$}	(13)
		(04)	edge 		node[right]{$r$}	(14)
		(11)	edge 		node[below]{$\tt_1$}	(10)
		(12)	edge 		node[below]{$\tt_2$}	(11)
		(14)	edge 		node[below]{$\tt_n$}	(13);
	\path
		(03)	edge[dotted] 						(02);
	\path[black!50]
		(13)	edge[dotted] 						(12);
\end{tikzpicture}
\end{center}
\vspace*{-10pt}
\begin{center}
\begin{tikzpicture}[thick,node distance=16mm]
	\node (01) at (.75,2) {$\ss_1$};
	\node (03) [right of=01] {$\ss_2$};
	\node (05) [right of=03] {};
	\node (07) [right of=05] {$\ss_n$};
	\node (10) at (0,1.25) {$q$};
	\node (12) [right of=10] {$q_1$};
	\node (14) [right of=12] {$q_2$};
	\node (14right) [right of=14,node distance=4mm] {};
	\node (16) [right of=14,inner sep=0pt] {$q_{n-1}$};
	\node (16left) [left of=16,node distance=6mm] {};
	\node (18) [right of=16] {$r$};
	\node (21) [below of=01] {$\tt_1$};
	\node (23) [right of=21] {$\tt_2$};
	\node (25) [right of=23] {};
	\node (27) [right of=25] {$\tt_n$};
	\path[->,>=latex]
		(01)	edge  (21)
		(10)	edge  (12)
		(03)	edge  (23)
		(12)	edge  (14)
		(16)	edge  (18)
		(07)	edge	  (27);
	\path
		(14right)	edge[dotted]	(16left);
\end{tikzpicture}
\end{center}

We choose on purpose to always draw a normalisation square-diagram backward,
such that it coincides with the associated cross-diagram.
The function~$\sigma_q$ induced by the state~$q$ should
map any word~$\word{s}$ (read backward)
to some word~$\word{t}$ (read backward) with $\NM(\word{s}q)=\NM(r\word{t})$.

\medbreak We now aim to strike reasonable (most often optimal) hypotheses for a quadratic normalisation~$(\Gar,\NM)$
associated with an original semigroup~$\MM$ to generate a semigroup~$\press{~\Mea_{\MM,\Gar,\NM}~}$
that approximates~$\MM$ as sharply as possible. Since the generating sets coincide by Definition~\ref{def-msqn}, we shall focus on the case
where $\MM$ is a quotient of~$\Mea_{\MM,\Gar,\NM}$ (top-approximation, Lemma~\ref{lem-easy}), and next, on the case where~$\Mea_{\MM,\Gar,\NM}$
is a quotient of~$\MM$ (bottom-approximation, Proposition~\ref{prop-home}).

\medbreak Before establishing our top-approximation statement (Lemma~\ref{lem-easy}),
we first recall how semigroups could appear much more difficult to handle,
especially when it comes to automaticity (see~\cite{HoffmannPhD}) or self-similarity (see~\cite{BroughCain15,BroughCain16}).
Let~$\MM$ be a semigroup with a quadratic normalisation~$(\Gar,\NM)$: 
two situations occur.
First, if $\MM$ is a monoid with unit~$\ttunit$, it admits a quadratic normalisation satisfying~$\NM(\ttunit)=\ttunit$ and
\begin{equation}
\NM(\ttunit q)=\NM(q\ttunit)=\ttunit q\tag{\unit}
\end{equation}
for each~$q\in\Gar$.
Second, if $\MM$ does not admit a unit, one can adjoin a unit~$\ttunit$ to obtain a monoid (if needed) with a quadratic normalisation satisfying~Condition\cunit.
The choice made for such a condition becomes natural
whenever we think of the (adjoined or not) unit~$\ttunit$
as some \emph{dummy} element that escapes from the normalisation and simply ensures its length-preserving property.

\begin{lemma}\label{lem-easy}
If $\MM$ is a monoid with a quadratic normalisation~$(\Gar,\NM)$ satisfying~Condition~\cunit,
then the Mealy automaton~$\Mea_{\MM,\Gar,\NM}$ generates a monoid of which~$\MM$ is a quotient.
\end{lemma}

\begin{figure}[b]
\begin{center}
\begin{tikzpicture}[thick,node distance=11mm,inner sep=.4mm]
\begin{scope}[]
	\node (00) at (0,2) {};
	\node (01) [right of=00] {};
	\node (02) [right of=01] {};
	\node (03) [right of=02] {};
	\node (04) [right of=03] {};
	\node (05) [right of=04] {};
	\node (06) [node distance=6mm,right of=05] {};
	\node (10) [below of=00] {};
	\node (11) [right of=10] {};
	\node (12) [right of=11] {};
	\node (13) [right of=12] {};
	\node (14) [right of=13] {};
	\node (15) [right of=14] {};
	\node (16) [node distance=6mm,right of=15] {};
	\node (20) [below of=10] {};
	\node (21) [right of=20] {};
	\node (22) [right of=21] {};
	\node (23) [right of=22] {};
	\node (24) [right of=23] {};
	\node (25) [right of=24] {};
	\node (26) [node distance=6mm,right of=25] {};
	\node (30) [below of=20] {};
	\node (31) [right of=30] {};
	\node (32) [right of=31] {};
	\node (33) [right of=32] {};
	\node (34) [right of=33] {};
	\node (35) [right of=34] {};
	\node (36) [node distance=6mm,right of=35] {};
	\node (40) [below of=30] {};
	\node (41) [right of=40] {};
	\node (42) [right of=41] {};
	\node (43) [right of=42] {};
	\node (44) [right of=43] {};
	\node (45) [right of=44] {};
	\node (46) [node distance=6mm,right of=45] {};
	\begin{scope}[node distance=5.5mm]
		\node (005) [right of=00,black!50] {$\ttunit$};
		\node (015) [right of=01,black!50] {$\ttunit$};
		\node (025) [right of=02,black!50] {$\ttunit$};
		\node (035) [right of=03,black!50] {};
		\node (045) [right of=04,black!50] {$\ttunit$};
		\node (055) [right of=05,black!50] {};
		\node (065) [right of=06,black!50] {};
		\node (105) [right of=10,inner sep=-1pt] {$\phantom{'}p_1\phantom{'}$};
		\node (115) [right of=11,inner sep=1pt,black!50] {$\ttunit$};
		\node (125) [right of=12,inner sep=1pt,black!50] {$\ttunit$};
		\node (135) [right of=13,inner sep=1pt,black!50] {};
		\node (145) [right of=14,inner sep=1pt,black!50] {$\ttunit$};
		\node (155) [right of=15,inner sep=1pt,black!50] {};
		\node (165) [right of=16,inner sep=1pt,black!50] {};
		\node (205) [right of=20,inner sep=0pt,yshift=-4pt] {$\phantom{'}p_1'\phantom{'}$};
		\node (215) [right of=21,inner sep=0pt,yshift=-4pt] {$\phantom{'}p_2'\phantom{'}$};
		\node (225) [right of=22,inner sep=2pt,yshift=-4pt,black!50] {$\ttunit$};
		\node (235) [right of=23,inner sep=1pt,black!50] {};
		\node (245) [right of=24,inner sep=2pt,yshift=-4pt,black!50] {$\ttunit$};
		\node (255) [right of=25,inner sep=1pt,black!50] {};
		\node (265) [right of=26,inner sep=1pt,black!50] {};
		\node (305) [right of=30,inner sep=0pt,yshift=2pt] {$\phantom{'}p_1^{(k-2)}\phantom{'}$};
		\node (315) [right of=31,inner sep=0pt,yshift=2pt] {$\phantom{'}p_2^{(k-2)}\phantom{'}$};
		\node (325) [right of=32,inner sep=-2pt] {};
		\node (335) [right of=33,inner sep=2pt,black!50] {$\ttunit$};
		\node (345) [right of=34,inner sep=2pt,black!50] {$\ttunit$};
		\node (355) [right of=35,inner sep=-2pt] {};
		\node (405) [right of=40,inner sep=0pt,yshift=-6pt] {$\phantom{'}p_1^{(k-1)}\phantom{'}$};
		\node (415) [right of=41,inner sep=0pt,yshift=-6pt] {$\phantom{'}p_2^{(k-1)}\phantom{'}$};
		\node (425) [right of=42,inner sep=1pt] {};
		\node (435) [right of=43,inner sep=0pt,yshift=-6pt] {$\phantom{'}p_k^{(k-1)}\phantom{'}$};
		\node (445) [right of=44,inner sep=4pt,yshift=-6pt,black!50] {$\ttunit$};
		\node (500) [below of=00,inner sep=-2pt] {$\phantom{'}p_1\phantom{'}$};
		\node (510) [below of=10,inner sep=-2pt] {$\phantom{'}p_1\phantom{'}$};
		\node (520) [below of=20,inner sep=-2pt] {};
		\node (530) [below of=30,inner sep=-2pt] {$\phantom{'}p_k\phantom{'}$};
		\node (501) [below of=01,inner sep=1pt,black!50] {$\ttunit$};
		\node (511) [below of=11,inner sep=-2pt] {$\phantom{'}p_2'\phantom{'}$};
		\node (521) [below of=21,inner sep=-2pt] {};
		\node (531) [below of=31,inner sep=-2pt] {$\phantom{'}p_k'\phantom{'}$};
		\node (502) [below of=02,inner sep=1pt,black!50] {$\ttunit$};
		\node (512) [below of=12,inner sep=1pt,black!50] {$\ttunit$};
		\node (522) [below of=22,inner sep=-2pt] {};
		\node (532) [below of=32,inner sep=-2pt] {$\phantom{'}p_k''\phantom{'}$};
		\node (503) [below of=03,inner sep=1pt,black!50] {$\ttunit$};
		\node (513) [below of=13,inner sep=1pt,black!50] {$\ttunit$};
		\node (523) [below of=23,inner sep=-2pt] {};
		\node (533) [below of=33,inner sep=-10pt] {$\phantom{'}p_k^{(k-1)}\phantom{'}$};
		\node (533bis) [node distance=4mm,left of=533] {};
		\node (504) [below of=04,inner sep=1pt,black!50] {$\ttunit$};
		\node (514) [below of=14,inner sep=1pt,black!50] {$\ttunit$};
		\node (524) [below of=24,inner sep=-2pt] {};
		\node (534) [below of=34,inner sep=1pt,black!50] {$\ttunit$};
		\node (505) [below of=05,inner sep=1pt,black!50] {$\ttunit$};
		\node (515) [below of=15,inner sep=1pt,black!50] {$\ttunit$};
		\node (525) [below of=25,inner sep=-2pt] {};
		\node (535) [below of=35,inner sep=1pt,black!50] {$\ttunit$};
		\node (506) [below of=06,inner sep=1pt,black!50] {};
		\node (516) [below of=16,inner sep=1pt,black!50] {};
		\node (536) [below of=36,inner sep=1pt,black!50] {};
	\end{scope}
	\path[->,>=latex]
		(005)	edge[black]	(105)
		(500)	edge[black]	(501)
		(015)	edge[black]	(115)
		(501)	edge[black]	(502)
		(025)	edge[black]	(125)
		(502)	edge[black]	(503)
		(503)	edge[-,dotted,black!50]	(504)
		(045)	edge[black]	(145)
		(504)	edge[black]	(505)
		(105)	edge[black]	(205)
		(510)	edge[black]	(511)
		(115)	edge[black]	(215)
		(511)	edge[black]	(512)
		(125)	edge[black]	(225)
		(512)	edge[black]	(513)
		(513)	edge[-,dotted,black!50]	(514)
		(145)	edge[black]	(245)
		(514)	edge[black]	(515)
		(205)	edge[-,dotted,black!50]	(305)
		%
		(215)	edge[-,dotted,black!50]	(315)
		%
		%
		(245)	edge[-,dotted,black!50]	(345)
		(305)	edge[black]	(405)
		(530)	edge[black]	(531)
		(315)	edge[black]	(415)
		(531)	edge[black]	(532)
		(532)	edge[-,dotted,black!50]	(533bis)
		(335)	edge[black]	(435)
		(533)	edge[black]	(534)
		(345)	edge[black]	(445)
		(534)	edge[black]	(535)
		(505)	edge[-,dotted,black!50]	(506)
		(515)	edge[-,dotted,black!50]	(516)
		(535)	edge[-,dotted,black!50]	(536);

\end{scope}
\begin{scope}[xshift=75mm]
	\node (00) at (0,2) {};
	\node (01) [right of=00] {};
	\node (02) [right of=01] {};
	\node (03) [right of=02] {};
	\node (04) [right of=03] {};
	\node (05) [right of=04] {};
	\node (06) [node distance=6mm,right of=05] {};
	\node (10) [below of=00] {};
	\node (11) [right of=10] {};
	\node (12) [right of=11] {};
	\node (13) [right of=12] {};
	\node (14) [right of=13] {};
	\node (15) [right of=14] {};
	\node (16) [node distance=6mm,right of=15] {};
	\node (20) [below of=10] {};
	\node (21) [right of=20] {};
	\node (22) [right of=21] {};
	\node (23) [right of=22] {};
	\node (24) [right of=23] {};
	\node (25) [right of=24] {};
	\node (26) [node distance=6mm,right of=25] {};
	\node (30) [below of=20] {};
	\node (31) [right of=30] {};
	\node (32) [right of=31] {};
	\node (33) [right of=32] {};
	\node (34) [right of=33] {};
	\node (35) [right of=34] {};
	\node (36) [node distance=6mm,right of=35] {};
	\node (40) [below of=30] {};
	\node (41) [right of=40] {};
	\node (42) [right of=41] {};
	\node (43) [right of=42] {};
	\node (44) [right of=43] {};
	\node (45) [right of=44] {};
	\node (46) [node distance=6mm,right of=45] {};
	\path[->,>=latex,black!50]
		(00)	edge[thin,densely dotted] 		node[fill=white]{\scriptsize$\!\NM\!$}	(11)
		(01)	edge[thin,densely dotted] 		node[fill=white]{\scriptsize$\!\NM\!$}	(12)
		(02)	edge[thin,densely dotted] 		node[fill=white]{\scriptsize$\!\NM\!$}	(13)
		(04)	edge[thin,densely dotted] 		node[fill=white]{\scriptsize$\!\NM\!$}	(15)
		(10)	edge[thin,densely dotted] 		node[fill=white]{\scriptsize$\!\NM\!$}	(21)
		(11)	edge[thin,densely dotted] 		node[fill=white]{\scriptsize$\!\NM\!$}	(22)
		(12)	edge[thin,densely dotted] 		node[fill=white]{\scriptsize$\!\NM\!$}	(23)
		(14)	edge[thin,densely dotted] 		node[fill=white]{\scriptsize$\!\NM\!$}	(25)
		(30)	edge[thin,densely dotted] 		node[fill=white]{\scriptsize$\!\NM\!$}	(41)
		(31)	edge[thin,densely dotted] 		node[fill=white]{\scriptsize$\!\NM\!$}	(42)
		(33)	edge[thin,densely dotted] 		node[fill=white]{\scriptsize$\!\NM\!$}	(44)
		(34)	edge[thin,densely dotted] 		node[fill=white]{\scriptsize$\!\NM\!$}	(45);
	\path[->,>=latex]
		(01)	edge[black!50]	node[above]{$\ttunit$}	(00)
		(02)	edge[black!50]	node[above]{$\ttunit$}	(01)
		(03)	edge[black!50]	node[above]{$\ttunit$}	(02)
		(04)	edge[-,dotted,black!50] 			(03)
		(05)	edge[black!50]	node[above]{$\ttunit$}	(04)
		(06)	edge[-,dotted,black!50] 			(05)
		(00)	edge 		node[left]{$p_1\phantom{'}$}	(10)
		(01)	edge[black!50]	node[right]{$\ttunit$}	(11)
		(02)	edge[black!50]	node[right]{$\ttunit$}	(12)
		(03)	edge[black!50]	node[right]{$\ttunit$}	(13)
		(04)	edge[black!50]	node[right]{$\ttunit$}	(14)
		(05)	edge[black!50]	node[right]{$\ttunit$}	(15)
		(11)	edge 		node[below]{$\phantom{'}p_1\phantom{'}$}	(10)
		(12)	edge[black!50]	node[below]{$\phantom{'}1\phantom{'}$}		(11)
		(13)	edge[black!50]	node[below]{$\phantom{'}1\phantom{'}$}		(12)
		(14)	edge[-,dotted,black!50] 								(13)
		(15)	edge[black!50]	node[below]{$\phantom{'}1\phantom{'}$}		(14)
		(16)	edge[-,dotted,black!50] 								(15)
		(10)	edge 		node[left]{$p_2\phantom{'}$}	(20)
		(11)	edge 		node[right]{$p'_2$}	(21)
		(12)	edge[black!50]	node[right]{$\ttunit$}	(22)
		(13)	edge[black!50]	node[right]{$\ttunit$}	(23)
		(14)	edge[black!50]	node[right]{$\ttunit$}	(24)
		(15)	edge[black!50]	node[right]{$\ttunit$}	(25)
		(21)	edge 		node[below]{$p'_1$}	(20)
		(22)	edge 		node[below]{$p'_2$}	(21)
		(23)	edge[black!50]	node[below]{$\ttunit$}	(22)
		(24)	edge[-,dotted,black!50] 			(23)
		(25)	edge[black!50]	node[below]{$\ttunit$}	(24)
		(26)	edge[-,dotted,black!50] 			(25)
		(20)	edge[-,dotted] 					(30)
		(21)	edge[-,dotted] 					(31)
		(22)	edge[-,dotted] 					(32)
		(23)	edge[-,dotted,black!50] 			(33)
		(24)	edge[-,dotted,black!50] 			(34)
		(25)	edge[-,dotted,black!50] 			(35)
		(31)	edge 		(30)
		(32)	edge 		(31)
		(33)	edge[-,dotted]	(32)
		(34)	edge[black!50]	node[below]{$\ttunit$}	(33)
		(35)	edge[black!50]	(34)
		(36)	edge[-,dotted,black!50] 			(35)
		(30)	edge 		node[left]{$p_k\phantom{'}$}	(40)
		(31)	edge 		node[left]{$p'_k$}	(41)
		(32)	edge 		node[left]{$p''_k$}	(42)
		(33)	edge 		node[fill=white]{$p^{\tiny(k-1)}_k$}	(43)
		(34)	edge[black!50]	node[right]{$\ttunit$}	(44)
		(35)	edge[black!50]	node[right]{$\ttunit$}	(45)
		(41)	edge 		node[below]{$p^{\tiny(k-1)}_1$}	(40)
		(42)	edge 		node[below]{$p^{\tiny(k-1)}_2$}	(41)
		(43)	edge[-,dotted]							(42)
		(44)	edge			node[below]{$p^{\tiny(k-1)}_k$}	(43)
		(45)	edge[black!50]	node[below]{$\phantom{'}1\phantom{'}$}	(44)
		(46)	edge[-,dotted,black!50] 					(45);
\end{scope}
\end{tikzpicture}
\end{center}
\vspace*{-15pt}
\caption{Proof of Lemma~\ref{lem-easy}: any $\Gar$-words inducing a same action (\emph{e.g.} on~$\ttunit^\omega$) are $\NM$-equivalent.}%
\label{fig-easy}
\end{figure}

\begin{proof}Let~$\MM=\Gar^*/\!\equiv_{\NM}$
and~$\Mea_{\MM,\Gar,\NM}=(\Gar,\Gar,\tau,\sigma)$ as in Definition~\ref{def-msqn}.
We have to prove that any relation in~$\presm{~\Mea_{\MM,\Gar,\NM}~}$ is a relation in~$\MM$, thereby implying for any $\Gar$-words~$\word{u}$ and~$\word{v}$:
\begin{equation*}
\sigma_{\word{u}}=\sigma_{\word{v}} \ \Longrightarrow\ \word{u}\equiv_{\NM}\word{v}.
\end{equation*}
Let~$\sigma_{p_1}\cdots\sigma_{p_k}=\sigma_{q_1}\cdots\sigma_{q_{k+\ell}}$ be some relation in~$\presm{~\Mea_{\MM,\Gar,\NM}~}$ with~$p_i\in\Gar$ for~$0\leq i\leq k$, and~$q_j\in\Gar$ for~$0\leq j\leq k+\ell$ and~$\ell\geq 0$. Any $\Gar$-word~$\word{w}$ admits hence the same image under the action of~$\sigma_{p_1}\cdots\sigma_{p_k}$ and under the action of~$\sigma_{q_1}\cdots\sigma_{q_{k+\ell}}$.
By taking for~$\word{w}$ the special word~$\ttunit^{k+\ell}$ (or any sufficiently long power of~$\ttunit$),
such a common image corresponds to some $\Gar$-word which happens to be $\NM$-equivalent
to both 
$p_1\cdots p_k$ and~$q_1\cdots q_{k+\ell}$ (see~Figure~\ref{fig-easy}).
Indeed, we define the sequence~$\left(p_i^{(j)}\right)_{\raisebox{1.3ex}{\scriptsize${0<i\leq k}\atop{0\leq j<k}$}}$ over~$\Gar$ by~$p_i^{(0)}=p_i$, $\sigma_{p_{i+1}^{(j)}}(p_{j+1}^{(i-1)})=p_{j+1}^{(i)}$,
and~$\tau_{p_{j+1}^{(i-1)}}(p_{i+1}^{(j)})=p_{i+1}^{(j+1)}$, or equivalently by~$\NMbar(p_{j+1}^{(i-1)}p_{i+1}^{(j)})=p_{i+1}^{(j+1)}p_{j+1}^{(i)}$, for~$0\leq j<i\leq k$.
The sequence~$\left(q_i^{(j)}\right)_{\raisebox{1.3ex}{\scriptsize${0<i\leq k+\ell}\atop{0\leq j<k+\ell}$}}$ is similarly defined. According to Condition~\cunit, we obtain precisely
\[\ttunit^\ell p_1\cdots p_k 
\equiv_{\NM}\ttunit^\ell  p^{\tiny(k-1)}_k\cdots p^{\tiny(k-1)}_1 
= q^{\tiny(k+\ell-1)}_{k+\ell}\cdots q^{\tiny(k+\ell-1)}_1
\equiv_{\NM} q_1\cdots q_{k+\ell}.\]
Therefore the three corresponding $\Gar$-words~$\ttunit^\ell p_1\cdots p_k$, $p_1\cdots p_k$, and~$q_1\cdots q_\ell$
represent a same element in~$\MM$ by definition.
\end{proof}

Although specific to a monoidal framework and then requiring the innocuous Condition~\cunit,
the previous straightforward proof relies only on the definition of a quadratic normalisation
and on the well-fitted associated Mealy automaton (Definition~\ref{def-msqn}).
For the bottom-approximation statement, we consider an extra assumption, which happens to be necessary and sufficient.

\begin{proposition}\label{prop-home} Assume that $\MM$ is a semigroup
with a quadratic normalisation~$(\Gar,\NM)$.
If Condition~\home\ is satisfied, then the Mealy automaton~$\Mea_{\MM,\Gar,\NM}$ generates a semigroup quotient of~$\MM$.
The converse holds provided that Condition~\cunit\ is satisfied.
\end{proposition}

\begin{proof} Let~$\MM=\Gar^+/\!\equiv_\NM$ and  $\Mea_{\MM,\Gar,\NM}=(\Gar,\Gar,\tau,\sigma)$ as in Definition~\ref{def-msqn}.

\smallskip
$(\Leftarrow)$ Assume that Condition~\home\ is satisfied and that there exists~$(a,b,c,d)\in\Gar^4$ with~$ab\equiv_\NM cd$.
To show that~$\press{~\Mea_{\MM,\Gar,\NM}~}$ is quotient of~$\MM$, it suffices to prove $\sigma_{ab}=\sigma_{cd}$.
Without loss of generality, the word~$ab$ can be supposed to be~$\NM$-normal, that is, we can set
\begin{equation}\NM(ab)=\NM(cd)=ab.\label{eq-0}\end{equation}
Let~$\word{u}=q\word{v}\in\Gar^n$ for some~$n>0$ and~$q\in\Gar$.
We shall prove both $\sigma_{ab}(\word{u})=\sigma_{cd}(\word{u})$ (letterwise)
and~$\tau_{\word{u}}(ab)\equiv_\NM\tau_{\word{u}}(cd)$ by induction on~$n>0$.
For~$n=1$, we obtain the two square-diagrams on Figure~\ref{fig-proof} (left) (reproduced on page~\pageref{fig-proof}),
that is,
\begin{align}
\NM(qa)&=a'q_0',\label{eq-1}\\
\NM(qc)&=c'q_1',\label{eq-3}\\
\NM(q_0'b)&=b'q_0'',\label{eq-2}\\
\hbox{ and }\NM(q_1'c)&=d'q_1''\label{eq-4}
\end{align}
for some~$a',b',c',d',q_0',q_0'',q_1'q_1''\in\Gar$.
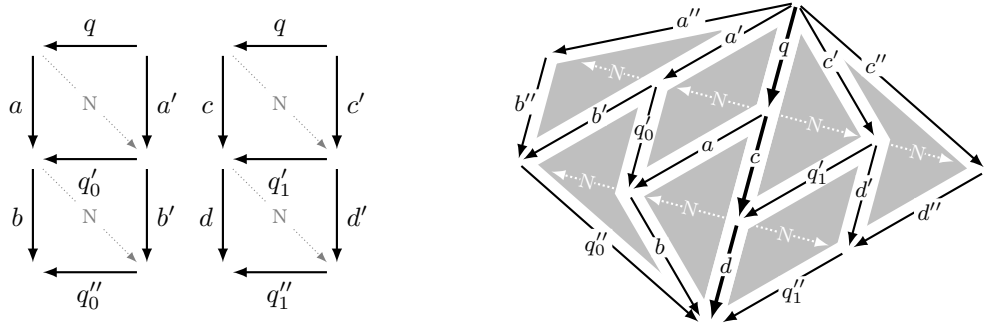
\begin{figure}[b]
\vspace*{-25pt}
\centering
\begin{tikzpicture}[thick,node distance=15mm]
	\node (00) at (0.5,2) {};
	\node (01) [right of=00] {};
	\node (10) [below of=00] {};
	\node (11) [right of=10] {};
	\node (20) [below of=10] {};
	\node (21) [right of=20] {};
	\path[->,>=latex]
		(01)	edge 		node[above]{$q$}	(00)
		(00)	edge 		node[left]{$\phantom{'}a$}		(10)
		(10)	edge 		node[left]{$\phantom{'}b$}		(20)
		(01)	edge 		node[right]{$a'$}	(11)
		(11)	edge 		node[below]{$q_0'$}	(10)
		(11)	edge 		node[right]{$b'$}	(21)
		(21)	edge 		node[below]{$q_0''$}	(20)
		(00)	edge[thin,densely dotted,black!50] 		node[fill=white]{$\!\NM\!$}	(11)
		(10)	edge[thin,densely dotted,black!50] 		node[fill=white]{$\!\NM\!$}	(21)
		;
	\begin{scope}[thick,node distance=15mm,xshift=30mm]
	\node (00) at (0,2) {};
	\node (01) [right of=00] {};
	\node (10) [below of=00] {};
	\node (11) [right of=10] {};
	\node (20) [below of=10] {};
	\node (21) [right of=20] {};
	\path[->,>=latex]
		(01)	edge 		node[above]{$q$}	(00)
		(00)	edge 		node[left]{$\phantom{'}c$}		(10)
		(10)	edge 		node[left]{$\phantom{'}d$}		(20)
		(01)	edge 		node[right]{$c'$}	(11)
		(11)	edge 		node[below]{$q_1'$}	(10)
		(11)	edge 		node[right]{$d'$}	(21)
		(21)	edge 		node[below]{$q_1''$}	(20)
		(00)	edge[thin,densely dotted,black!50] 		node[fill=white]{$\!\NM\!$}	(11)
		(10)	edge[thin,densely dotted,black!50] 		node[fill=white]{$\!\NM\!$}	(21)
		;
	\end{scope}
	\node at (9.6,.5) {
	 \rotatebox{-15}{
 \begin{tikzpicture}[thick,node distance=15mm,inner sep=1.3pt]
	\node[inner sep=4pt] (00) at (0,0) {};
	\node (10) [above of=00] {};
	\node (11) [left of=10] 	{};
	\node (12) [right of=10] 	{};
	\node (13) [left of=11] 	{};
	\node (14) [right of=12] 	{};
	\node (20) [above of=10] {};
	\node (21) [above of=11] {};
	\node (22) [above of=12] {};
	\node (23) [above of=13] {};
	\node (24) [above of=14] {};
	\node (30) [above of=20] {};
	\node (11b) [node distance=13mm,left of=10] {};
	\node (12b) [node distance=13mm,right of=10] {};
	\node (13b) [node distance=11mm,left of=11] {};
	\node (21b) [node distance=13mm,left of=20] {};
	\node (22b) [node distance=13mm,right of=20] {};
	\node (23b) [node distance=11mm,left of=21] {};
	\node (24b) [node distance=11mm,right of=22] {};
	\fill[color=lightgray] (0,0) -- (-3,1.5) -- (-3,3) -- (0,4.5) -- (0,0);
	\fill[color=lightgray] (0,0) -- (3,3) -- (0,4.5) -- (0,0);
	\draw[color=white,line width=9pt] (0,0) -- (-3,1.5) -- (-3,3) -- (0,4.5) -- (0,0);
	\draw[color=white,line width=9pt] (0,0) -- (3,3) -- (0,4.5) -- (0,0);
	\draw[color=white,line width=9pt] (0,1.5) -- (1.5,3);
	\draw[color=white,line width=9pt] (1.5,1.5) -- (1.5,3) -- (0,4.5);
	\draw[color=white,line width=9pt] (-3,1.5) -- (-1.5,3) -- (0,4.5);
	\draw[color=white,line width=9pt] (-1.5,1.5) -- (0,3);
	\draw[color=white,line width=9pt] (0,0) -- (-1.5,1.5) -- (-1.5,3);
	\path[->,>=latex,line width=1.4pt]	(30)	edge 		node[pos=.42,fill=white,rotate=15]{\footnotesize$\!q\!$} (20);
	\path[->,>=latex,line width=1.4pt]	(20)	edge			node[pos=.42,fill=white,rotate=15]{\footnotesize$\!c\!$} (10);
	\path[->,>=latex,line width=1.4pt]	(10)	edge			node[pos=.42,fill=white,rotate=15]{\footnotesize$\!d\!$} (00);
	\path[->,>=latex] 	(30)	edge		node[pos=.42,fill=white,rotate=15,text height=5pt,pos=.4,left=-.07]{\footnotesize$a'\!$} (21);
	\path[->,>=latex]	(21)	edge		node[pos=.42,fill=white,rotate=15,text height=6pt]{\footnotesize$\!\!q_0'\!\!$} (11);
	\path[->,>=latex]	(11)	edge		node[pos=.42,fill=white,rotate=15,text height=5pt]{\footnotesize$b$} (00);
	\path[->,>=latex]	(20)	edge		node[pos=.42,fill=white,rotate=15]{\footnotesize$a$} (11);
	\path[->,>=latex]	(30)	edge		node[pos=.42,fill=white,rotate=15,text height=5pt]{\footnotesize$\!c'\!$} (22);
	\path[->,>=latex]	(22)	edge		node[pos=.42,fill=white,rotate=15,text height=3pt,left=-.095,white,pos=.4]{\footnotesize$\!w$}
								node[pos=.42,rotate=15,text height=3pt,left=-.095]{\footnotesize$\!q_1'\hspace*{-3.05pt}$} (10);
	\path[->,>=latex]	(12)	edge		node[pos=.42,fill=white,rotate=15,below=-.1]{\footnotesize$q_1''$} (00);
	\path[->,>=latex]	(22)	edge		node[pos=.42,fill=white,rotate=15]{\footnotesize$\!d'\!$} (12);
	\path[->,>=latex]	(21)	edge		node[pos=.42,fill=white,rotate=15,text height=4pt,pos=.4,left=-.1]{\footnotesize$b'\hspace*{-2.05pt}$} (13);
	\path[->,>=latex]	(13)	edge		node[pos=.42,fill=white,rotate=15,below=-.05]{\footnotesize$q_0''$} (00);
	\path[->,>=latex]	(30)	edge		node[pos=.42,fill=white,rotate=15,above=-.1]{\footnotesize$a''$} (23);
	\path[->,>=latex]	(23)	edge		node[pos=.42,fill=white,rotate=15,left=-.1]{\footnotesize$b''$} (13);
	\path[->,>=latex]	(30)	edge		node[pos=.42,fill=white,rotate=15,above=-.05]{\footnotesize$c''$} (24);
	\path[->,>=latex]	(24)	edge		node[pos=.42,fill=white,rotate=15,below=-.1]{\footnotesize$d''$} (12);
	\path[->,>=latex]
	(10)	edge[densely dotted,white] 		node[fill=lightgray]{$\!\NM\!$}	(11b)
	(10)	edge[densely dotted,white] 		node[fill=lightgray]{$\!\NM\!$}	(12b)
	(11)	edge[densely dotted,white] 		node[fill=lightgray]{$\!\NM\!$}	(13b)
	(20)	edge[densely dotted,white] 		node[fill=lightgray]{$\!\NM\!$}	(21b)
	(20)	edge[densely dotted,white] 		node[fill=lightgray]{$\!\NM\!$}	(22b)
	(21)	edge[densely dotted,white] 		node[fill=lightgray]{$\!\NM\!$}	(23b)
	(22)	edge[densely dotted,white] 		node[fill=lightgray]{$\!\NM\!$}	(24b);
\end{tikzpicture}
\hspace*{-20pt}
}
};
\end{tikzpicture}
\vspace*{-30pt}
\caption{Proof of Proposition~\ref{prop-home}: initial data (left) can be pasted into Condition~\home\ (right).}%
\label{fig-proof}
\end{figure}
As illustrated in Figure~\ref{fig-proof} (right), we obtain
\[\NMbar_{2121}(qcd)\stackrel{\eqref{eq-0}}{=}\NMbar_{121}(qab)\stackrel{\eqref{eq-1}}{=}\NMbar_{21}(a'q_0'b)\stackrel{\eqref{eq-2}}{=}\NMbar_{1}(a'b'q_0'')\stackrel{\eqref{eq-5}}{=}a''b''q_0'',\]
and\[\NMbar_{121}(qcd)\stackrel{\eqref{eq-3}}{=}\NMbar_{21}(c'q_1'd)\stackrel{\eqref{eq-4}}{=}\NMbar_{1}(c'd'q_1'')\stackrel{\eqref{eq-6}}{=}c''d''q_1'',\]
where~$a'',b'',c'',d''\in\Gar$ are defined by
\begin{align}
\NM(a'b')&=a''b''\label{eq-5}\\
\hbox{ and }\NM(c'd')&=c''d''.\label{eq-6}
\end{align}
Condition~\home\ implies~$\NMbar_{2121}(qcd)=\NMbar_{121}(qcd)$, that is, $a''b''q_0''=c''d''q_1''$.
This means that~$q_0''=q_1''$, $a''=c''$, and~$b''=d''$ hold.
On the one hand, the conjunction of the latter two gives~$\NM(a'b')=\NM(c'd')$ by~Equations~\eqref{eq-5} and~\eqref{eq-6}, that is
\begin{equation}a'b'\equiv_\NM c'd'.\label{eq-nxt}\end{equation}
On the other hand, $q_0''=q_1''$ means~$\sigma_{ab}(q)=\sigma_{cd}(q)$, which simply concludes the case~$n=1$.
Then Equation~\eqref{eq-nxt} allows to proceed the induction and to finally prove the implication~$(\Leftarrow)$:
by induction hypothesis, $a'b'\equiv_\NM c'd'$ implies~$\sigma_{a'b'}(\word{v})=\sigma_{c'd'}(\word{v})$ and~$\tau_{a'b'}(\word{v})\equiv_\NM\tau_{c'd'}(\word{v})$. From the first equality, we obtain~$\sigma_{ab}(\word{u})=\sigma_{cd}(\word{u})$ after left-appending~$\sigma_{ab}(q)=q_0''=q_1''=\sigma_{cd}(q)$.
From the second equivalence, we conclude~$\tau_{ab}(\word{u})=\tau_{a'b'}(\word{v})\equiv_\NM\tau_{c'd'}(\word{v})=\tau_{cd}(\word{u})$.

\medskip
$(\Rightarrow)$ Assume that $\press{~\Mea_{\MM,\Gar,\NM}~}$ is a quotient of~$\MM$, that is, $\word{u}\equiv_\NM\word{v}$ implies~$\sigma_{\word{u}}=\sigma_{\word{v}}$ for any words~$\word{u}$ and~$\word{v}$ over~$\Gar$. Consider an arbitrary length~3 word over~$\Gar$, say~$qcd\in\Gar^3$.
Let~$a,b$ denote the elements in~$\Gar$ satisfying
\begin{equation}\NM(cd)=ab.\label{eq-abcd}\end{equation}By definition, we deduce~$ab\equiv_\NM cd$.
This implies~$\sigma_{ab}=\sigma_{cd}$ by hypothesis.
In particular, the images of any nonempty word~$q\word{v}$ under~$\sigma_{ab}$ and under~$\sigma_{cd}$ coincide (letterwise). Now
$\sigma_{ab}(q\word{v})=\sigma_{cd}(q\word{v})$ decomposes into
\begin{equation}\sigma_{ab}(q)=q''_0=q''_1=\sigma_{cd}(q)\label{eq-q}
\end{equation}
and
\begin{equation}\sigma_{\tau_q(ab)}(\word{v})=\sigma_{a'b'}(\word{v})=\sigma_{c'd'}(\word{v})=\sigma_{\tau_q(cd)}(\word{v}),\label{eq-last}
\end{equation}
with (see Figure~\ref{fig-proof} (left) again)
\vspace*{-5pt}

\qquad\begin{minipage}{.45\textwidth}
\begin{align}
a'&=\tau_q(a),		&q_0'&=\sigma_a(q),\label{eq-A}\\
c'&=\tau_q(c),		&q_1'&=\sigma_c(q),	\label{eq-B}
\end{align}
\end{minipage}
\hfill\begin{minipage}{.45\textwidth}
\begin{align}
b'&=\tau_{q_0'}(b),	&q_0''&=\sigma_b(q_0'),\label{eq-C}\\
d'&=\tau_{q_1'}(d),	&q_1''&=\sigma_d(q_1').\label{eq-D}
\end{align}
\end{minipage}

~

Equality~\eqref{eq-last} holds for any original word~$\word{v}\in\Gar^*$ and therefore implies~$\sigma_{a'b'}=\sigma_{c'd'}$.
Now, whenever Condition~\cunit\ is satisfied, we deduce
\begin{equation}
\NM(a'b')=\NM(c'd')\label{eq-prime}
\end{equation}
according to Lemma~\ref{lem-easy}. For any such arbitrary word~$qcd\in\Gar^3$, we obtain
\[\begin{array}{rcccccccc}
\NMbar_{121}(qcd)&
\stackrel{\eqref{eq-B}}{=}		&\NMbar_{21}(c'q'_1d)&
\stackrel{\eqref{eq-D}}{=}		&\NMbar_{1}(c'd'q''_1)&
\stackrel{\eqref{eq-prime}}{=}	&\NMbar(a'b')q_1''\\
&&&&&&\raisebox{1.3ex}{\rotatebox{270}{$\stackrel{\rotatebox{90}{\scriptsize\eqref{eq-q}}}{=\!=}$}}\\
\NMbar_{2121}(qcd)&
\stackrel{\eqref{eq-abcd}}{=}	&\NMbar_{121}(qab)&
\stackrel{\eqref{eq-A}}{=}	&\NMbar_{21}(a'q_0'b)&
\stackrel{\eqref{eq-C}}{=}	&\NMbar_{1}(a'b'q_0'')&
\end{array}\]
Therefore $(\Gar,\NM)$ satisfies Condition~\home.
\end{proof}

\newpage
{Gathering Lemma~\ref{lem-easy} and Proposition~\ref{prop-home},
we obtain the following main result.
\par\nobreak
\begin{theorem}\label{thm-main} Assume that $\MM$ is a monoid
with a quadratic normalisation~$(\Gar,\NM)$ satisfying Conditions~\cunit\ and~\home.
Then the Mealy automaton~$\Mea_{\MM,\Gar,\NM}$ generates a monoid isomorphic to~$\MM$.
\end{theorem}}

\begin{proof}By construction, $\MM$ and $\presm{~\Mea_{\MM,\Gar,\NM}~}$ share a same generating subset~$\Gar$.
Now, any defining relation for~$\MM$ maps to a defining relation for~$\presm{~\Mea_{\MM,\Gar,\NM}~}$ by Proposition~\ref{prop-home},
and conversely by Lemma~\ref{lem-easy}.
\end{proof}

\begin{corollary}\label{cor-resid} Any monoid with a quadratic normalisation satisfying Conditions~\cunit\ and~\home\ is residually finite.
\end{corollary}

To conclude this main section, we come back to that remark (following Definition~\ref{def-automaticsem})
about the transducer approach by~Thurston.

\begin{definition}\label{def-thurston} 
With any quadratic normalisation~$(\Gar,\NM)$ is associated its \emph{Thurston transducer}
defined as the Mealy automaton~$\Thu_{\Gar,\NM}$ with stateset~$\Gar$, alphabet~$\Gar$, 
and transitions as follows:
\par\nobreak
\begin{center}
\vspace*{-5pt}
\begin{tikzpicture}[thick,node distance=13mm]
	\node (00) at (0,0) {};
	\node (01) [right of=00] {};
	\node (10) [below of=00] {};
	\node (11) [right of=10] {};
	\path[->,>=latex]
		(01)	edge 		node[above]{$a$}	(00)
		(00)	edge 		node[left]{$b$}		(10)
		(01)	edge 		node[right]{$d$}	(11)
		(11)	edge 		node[below]{$c$}	(10)
		(00)	edge[thin,densely dotted] 		node[fill=white]{$\!\NM\!$}	(11);
	\node[draw,circle,text width=3mm,align=center,inner sep=1.2pt] (quad-n-transi-left) at (5,-.8) {$a$};
	\node[draw,circle,text width=3mm,align=center,inner sep=1.2pt,node distance=18mm] (quad-n-transi-right) [right of=quad-n-transi-left] {$c$};
	\path[->,>=latex] (quad-n-transi-left) edge[bend left] node[align=center,above]{\(\!b\,|\,d\!\)} (quad-n-transi-right);
\end{tikzpicture}
\end{center}
\end{definition}

\begin{corollary}\label{cor-duality}Assume that $\MM$ is a monoid 
with a quadratic normalisation~$(\Gar,\NM)$ satisfying Conditions~\cunit\ and~\home.
The Thurston transducer~$\Thu_{\Gar,\NM}$ and the Mealy automaton~$\Mea_{\MM,\Gar,\NM}$ being dual automaton,
$\MM$ possesses both the explicitly dual properties of automaticity and self-similarity.
\end{corollary}

These rather unexpected results provide 
the very first bridge between two 
fundamental areas that have always been widely seen as irreconcilable: automatic semigroups \emph{vs} automaton semigroups. We choose to conclude by gathering several carefully selected examples, counterexamples, and open problems.

\newcommand\sepon{
\,\scalebox{.1}{\begin{tikzpicture}[scale=1.3]
\draw[-,line width=12pt,loosely dashed] (0,0) -- (0,2.05);
\end{tikzpicture}
}\,}

\newcommand\BRone{1}
\newcommand\BRa{\sigma_{\!1}}
\newcommand\BRb{\sigma_{\!2}}
\newcommand\BRab{\sigma_{\!1}\hspace*{-1pt}\sigma_{\!2}}
\newcommand\BRba{\sigma_{\!2}\hspace*{-.4pt}\sigma_{\!1}}
\newcommand\BRd{\Delta}

\colorlet{couleurTransition}{orange!60!black}
\newcommand\sepp{~~}
\renewcommand\sepp{
~\scalebox{.1}{\begin{tikzpicture}[rotate=-90,scale=1.3]
\draw[-,line width=12pt,loosely dashed] (0,0) -- (2.2,0);
\end{tikzpicture}
}~}
\newcommand\BRICd{
\scalebox{.1}{\begin{tikzpicture}[rotate=-90,scale=1.3]
\braid[-,line width=6pt, style strands={1}{orange}, style strands={2}{violet}, style strands={3}{gray},number of strands=3,height=12mm] s_1^{-1} s_2^{-1} s_1^{-1};
\end{tikzpicture}
}}
\newcommand\BRICone{
\scalebox{.1}{\begin{tikzpicture}[rotate=-90,scale=1.3]
\braid[-,line width=6pt, style strands={1}{orange}, style strands={2}{violet}, style strands={3}{gray},number of strands=3,border height=21mm] ;
\end{tikzpicture}
}}
\newcommand\BRICa{
\scalebox{.1}{\begin{tikzpicture}[rotate=-90,scale=1.3]
\braid[-,line width=6pt, style strands={1}{orange}, style strands={2}{violet}, style strands={3}{gray},number of strands=3,height=36mm] at (0,0) s_1^{-1};
\end{tikzpicture}
}}
\newcommand\BRICb{
\scalebox{.1}{\begin{tikzpicture}[rotate=-90,scale=1.3]
\braid[-,line width=6pt, style strands={1}{orange}, style strands={2}{violet}, style strands={3}{gray},number of strands=3,height=36mm] at (0,0) s_2^{-1};
\end{tikzpicture}
}}
\newcommand\BRICab{
\scalebox{.1}{\begin{tikzpicture}[rotate=-90,scale=1.3]
\braid[-,line width=6pt, style strands={1}{orange}, style strands={2}{violet}, style strands={3}{gray},number of strands=3,height=18mm] at (0,0) s_1^{-1}s_2^{-1};
\end{tikzpicture}
}}
\newcommand\BRICba{
\scalebox{.1}{\begin{tikzpicture}[rotate=-90,scale=1.3]
\braid[-,line width=6pt, style strands={1}{orange}, style strands={2}{violet}, style strands={3}{gray},number of strands=3,height=18mm] at (0,0) s_2^{-1}s_1^{-1};
\end{tikzpicture}
}}
\newcommand\BRICdS{
\scalebox{.1}{\begin{tikzpicture}[scale=1.6]
\braid[-,line width=6pt, style strands={1}{orange}, style strands={2}{violet}, style strands={3}{gray},number of strands=3,height=6mm] s_1^{-1} s_2^{-1} s_1^{-1};
\end{tikzpicture}
}}
\newcommand\BRIConeS{
\scalebox{.1}{\begin{tikzpicture}[scale=1.6]
\braid[-,line width=6pt, style strands={1}{orange}, style strands={2}{violet}, style strands={3}{gray},number of strands=3,border height=10.5mm] ;
\end{tikzpicture}
}}
\newcommand\BRICaS{
\scalebox{.1}{\begin{tikzpicture}[scale=1.6]
\braid[-,line width=6pt, style strands={1}{orange}, style strands={2}{violet}, style strands={3}{gray},number of strands=3,height=18mm] s_1^{-1};
\end{tikzpicture}
}}
\newcommand\BRICbS{
\scalebox{.1}{\begin{tikzpicture}[scale=1.6]
\braid[-,line width=6pt, style strands={1}{orange}, style strands={2}{violet}, style strands={3}{gray},number of strands=3,height=18mm] s_2^{-1};
\end{tikzpicture}
}}
\newcommand\BRICabS{
\scalebox{.1}{\begin{tikzpicture}[scale=1.6]
\braid[-,line width=6pt, style strands={1}{orange}, style strands={2}{violet}, style strands={3}{gray},number of strands=3,height=9mm] s_1^{-1}s_2^{-1};
\end{tikzpicture}
}}
\newcommand\BRICbaS{
\scalebox{.1}{\begin{tikzpicture}[scale=1.6]
\braid[-,line width=6pt, style strands={1}{orange}, style strands={2}{violet}, style strands={3}{gray},number of strands=3,height=9mm] s_2^{-1}s_1^{-1};
\end{tikzpicture}
}}

\newcommand\BThreeAutomatic{
\scalebox{.51}{
\begin{tikzpicture}[node distance=35mm,inner sep=1.5pt]
	 	\node[draw=orange,thin,circle,text width=6mm,align=center] (1) {$\BRICone$};
	 	\node[draw=orange,thin,circle,text width=6mm,align=center] (a) [above left of=1] {$\BRICa$};
	 	\node[draw=orange,thin,circle,text width=6mm,align=center] (b) [above right of=1] {$\BRICb$};
	 	\node[draw=orange,thin,circle,text width=6mm,align=center] (ab) [above of=a] {$\BRICab$};
	 	\node[draw=orange,thin,circle,text width=6mm,align=center] (ba) [above of=b] {$\BRICba$};
	 	\node[draw=orange,thin,circle,text width=6mm,align=center] (d) [above left of=ba] {$\BRICd$};
\begin{scope}[draw=couleurTransition,thick,couleurTransition]
\path[->,>=latex] (1)	edge[loop below]		node[scale=.8]{$\BRIConeS\sepp\BRIConeS$}	 (1);
\path[->,>=latex] (1)	edge			node[scale=.8,left=.2,pos=.4]{$\BRICaS\sepp\BRIConeS$}
							node[scale=.8,right=1,pos=.63]{$\BRICdS\sepp\BRIConeS$}		(a);
\path[->,>=latex] (1)	edge			node[scale=.8,right=.15,pos=.4]{$\BRICbS\sepp\BRIConeS$}	 (b);
\draw[->,>=latex] (1) ..  controls  ($(1)+(-6cm,1cm)$) and ($(ab)+(-4cm,-4cm)$) ..  (ab);%
\draw[->,>=latex] (1) ..  controls  ($(1)+(6cm,1cm)$) and ($(ba)+(4cm,-4cm)$) ..  (ba);%
\draw[->,>=latex] (1) ..	controls  ($(1)+(0cm,1cm)$) and ($(b)+(-1.5cm,-1cm)$) ..  ($(b)+(-1.5cm,0cm)$)
				..	controls  ($(b)+(-1.5cm,1cm)$) and ($(ab)+(1.5cm,-1cm)$) ..  ($(ab)+(1.5cm,0cm)$)
				..	controls  ($(ab)+(1.5cm,1cm)$) and ($(d)+(0cm,-1cm)$) ..  (d);%
\path[->,>=latex] (a)	edge[out=240,in=210,looseness=8]	node[scale=.8,align=center,left=.05,pos=.7]{$\BRIConeS\sepp\BRIConeS$}
											node[scale=.8,align=center,below=.05,pos=.5]{$\BRICaS\sepp\BRICaS$}	 (a);
\path[->,>=latex] (a)	edge		node[scale=.8,align=center,left=.05,pos=.57]{$\BRICbS\sepp\BRIConeS$}
						node[scale=.8,align=center,left=.05,pos=.43]{$\BRICabS\sepp\BRICaS$}
						node[scale=.8,align=center,left=1.65,pos=-0.18]{$\BRICabS\sepp\BRIConeS$}
						node[scale=.8,align=center,left=1.65,pos=1.25]{$\BRICdS\sepp\BRICaS$\\$\BRICbaS\sepp\BRIConeS$} (ab);	
\draw[->,>=latex] (a) ..  controls  ($(a)+(-4cm,4cm)$) and ($(d)+(-6cm,-1cm)$) ..  (d);%
\path[->,>=latex] (b)	edge[out=300,in=330,looseness=8]	node[scale=.8,align=center,right=.05,pos=.7]{$\BRIConeS\sepp\BRIConeS$}
											node[scale=.8,align=center,below=.05,pos=.5]{$\BRICbS\sepp\BRICbS$}	 (b);
\path[->,>=latex] (b)	edge		node[scale=.8,align=center,right=.05,pos=.57]{$\BRICaS\sepp\BRIConeS$}
						node[scale=.8,align=center,right=.05,pos=.43]{$\BRICbaS\sepp\BRICbS$}
						node[scale=.8,align=center,right=1.65,pos=-0.18]{$\BRICbaS\sepp\BRIConeS$}
						node[scale=.8,align=center,right=1.65,pos=1.25]{$\BRICdS\sepp\BRICbS$\\$\BRICabS\sepp\BRIConeS$} (ba);
\draw[->,>=latex] (b) ..  controls  ($(b)+(4cm,4cm)$) and ($(d)+(+6cm,-1cm)$) ..  (d);%
\path[->,>=latex] (ab)	edge[out=130,in=150,looseness=8]	(ab);
\path[->,>=latex] (ab)	edge	[out=-50,in=160]			node[scale=.8,align=center,left=.17,pos=.37]{$\BRICbS\sepp\BRICabS$}	 (b);
\path[->,>=latex] (ab)	edge[bend left=20]		node[scale=.8,align=center,below]{$\BRICbaS\sepp\BRICabS$}	 (ba);
\path[->,>=latex] (ab)	edge					node[scale=.8,align=center,left=.1,pos=.52]{$\BRICabS\sepp\BRICaS$}
									node[scale=.8,align=center,left=.08,pos=.30]{$\BRICaS\sepp\BRIConeS$}
									node[scale=.8,align=center,left=.1,pos=.74]{$\BRICdS\sepp\BRICabS$}
									node[scale=.8,align=center,left=1,pos=.22]{$\BRIConeS\sepp\BRIConeS$}	(d);
\path[->,>=latex] (ba)	edge[out=50,in=30,looseness=8]	(ba);
\path[->,>=latex] (ba)	edge	[out=-130,in=20]			node[scale=.8,align=center,right=.17,pos=.37]{$\BRICaS\sepp\BRICbaS$}	 (a);
\path[->,>=latex] (ba)	edge[bend left=20]		node[scale=.8,align=center,above]{$\BRICabS\sepp\BRICbaS$}	 (ab);
\path[->,>=latex] (ba)	edge					node[scale=.8,align=center,right=.1,pos=.52]{$\BRICbaS\sepp\BRICbS$}
									node[scale=.8,align=center,right=.08,pos=.30]{$\BRICbS\sepp\BRIConeS$}
									node[scale=.8,align=center,right=.1,pos=.74]{$\BRICdS\sepp\BRICbaS$}
									node[scale=.8,align=center,right=1,pos=0.22]{$\BRIConeS\sepp\BRIConeS$}  (d);
\path[->,>=latex] (d)	edge[loop above]	node[scale=.8,align=center,above=.08]{$\BRICaS\sepp\BRICbS$\\$\BRICbS\sepp\BRICaS$}
								node[scale=.8,align=center,left=.2,pos=.12,opacity=0]{$\BRIConeS\sepp\BRIConeS$\\$\BRICdS\sepp\BRICdS$}
								node[scale=.8,align=center,right=.2,pos=.88,opacity=0]{$\BRICabS\sepp\BRICbaS$\\$\BRICbaS\sepp\BRICabS$}	 (d);
\path[->,>=latex,opacity=0] (1)	edge[loop below]	node[scale=.8,align=center,below=.08]{$\BRone\sepon\BRone$\\$\BRd\sepon\BRd$} (1);
\end{scope}
\end{tikzpicture}
}
}

%
\renewcommand\BRd{
\scalebox{.1}{\begin{tikzpicture}[rotate=-90,scale=1.3]
\braid[-,line width=6pt, style strands={1}{orange}, style strands={2}{violet}, style strands={3}{gray},number of strands=3,height=8mm] s_1^{-1} s_2^{-1} s_1^{-1};
\end{tikzpicture}
}}
\renewcommand\BRone{
\scalebox{.1}{\begin{tikzpicture}[rotate=-90,scale=1.3]
\braid[-,line width=6pt, style strands={1}{orange}, style strands={2}{violet}, style strands={3}{gray},number of strands=3,border height=14mm] ;
\end{tikzpicture}
}}
\renewcommand\BRa{
\scalebox{.1}{\begin{tikzpicture}[rotate=-90,scale=1.3]
\braid[-,line width=6pt, style strands={1}{orange}, style strands={2}{violet}, style strands={3}{gray},number of strands=3,height=24mm] at (0,0) s_1^{-1};
\end{tikzpicture}
}}
\renewcommand\BRb{
\scalebox{.1}{\begin{tikzpicture}[rotate=-90,scale=1.3]
\braid[-,line width=6pt, style strands={1}{orange}, style strands={2}{violet}, style strands={3}{gray},number of strands=3,height=24mm] at (0,0) s_2^{-1};
\end{tikzpicture}
}}
\renewcommand\BRab{
\scalebox{.1}{\begin{tikzpicture}[rotate=-90,scale=1.3]
\braid[-,line width=6pt, style strands={1}{orange}, style strands={2}{violet}, style strands={3}{gray},number of strands=3,height=12mm] at (0,0) s_1^{-1}s_2^{-1};
\end{tikzpicture}
}}
\renewcommand\BRba{
\scalebox{.1}{\begin{tikzpicture}[rotate=-90,scale=1.3]
\braid[-,line width=6pt, style strands={1}{orange}, style strands={2}{violet}, style strands={3}{gray},number of strands=3,height=12mm] at (0,0) s_2^{-1}s_1^{-1};
\end{tikzpicture}
}}
\newcommand\BRdS{
\scalebox{.1}{\begin{tikzpicture}[scale=1.6]
\braid[-,line width=6pt, style strands={1}{orange}, style strands={2}{violet}, style strands={3}{gray},number of strands=3,height=8mm] s_1^{-1} s_2^{-1} s_1^{-1};
\end{tikzpicture}
}}
\newcommand\BRoneS{
\scalebox{.1}{\begin{tikzpicture}[scale=1.6]
\braid[-,line width=6pt, style strands={1}{orange}, style strands={2}{violet}, style strands={3}{gray},number of strands=3,border height=14mm] ;
\end{tikzpicture}
}}
\newcommand\BRaS{
\scalebox{.1}{\begin{tikzpicture}[scale=1.6]
\braid[-,line width=6pt, style strands={1}{orange}, style strands={2}{violet}, style strands={3}{gray},number of strands=3,height=24mm] s_1^{-1};
\end{tikzpicture}
}}
\newcommand\BRbS{
\scalebox{.1}{\begin{tikzpicture}[scale=1.6]
\braid[-,line width=6pt, style strands={1}{orange}, style strands={2}{violet}, style strands={3}{gray},number of strands=3,height=24mm] s_2^{-1};
\end{tikzpicture}
}}
\newcommand\BRabS{
\scalebox{.1}{\begin{tikzpicture}[scale=1.6]
\braid[-,line width=6pt, style strands={1}{orange}, style strands={2}{violet}, style strands={3}{gray},number of strands=3,height=12mm] s_1^{-1}s_2^{-1};
\end{tikzpicture}
}}
\newcommand\BRbaS{
\scalebox{.1}{\begin{tikzpicture}[scale=1.6]
\braid[-,line width=6pt, style strands={1}{orange}, style strands={2}{violet}, style strands={3}{gray},number of strands=3,height=12mm] s_2^{-1}s_1^{-1};
\end{tikzpicture}
}}

\renewcommand\BRd{
\scalebox{.1}{\begin{tikzpicture}[rotate=-90,scale=1.3]
\braid[-,line width=6pt, style strands={1}{orange}, style strands={2}{violet}, style strands={3}{gray},number of strands=3,height=10mm] s_1^{-1} s_2^{-1} s_1^{-1};
\end{tikzpicture}
}}
\renewcommand\BRone{
\scalebox{.1}{\begin{tikzpicture}[rotate=-90,scale=1.3]
\braid[-,line width=6pt, style strands={1}{orange}, style strands={2}{violet}, style strands={3}{gray},number of strands=3,border height=17.5mm] ;
\end{tikzpicture}
}}
\renewcommand\BRa{
\scalebox{.1}{\begin{tikzpicture}[rotate=-90,scale=1.3]
\braid[-,line width=6pt, style strands={1}{orange}, style strands={2}{violet}, style strands={3}{gray},number of strands=3,height=30mm] at (0,0) s_1^{-1};
\end{tikzpicture}
}}
\renewcommand\BRb{
\scalebox{.1}{\begin{tikzpicture}[rotate=-90,scale=1.3]
\braid[-,line width=6pt, style strands={1}{orange}, style strands={2}{violet}, style strands={3}{gray},number of strands=3,height=30mm] at (0,0) s_2^{-1};
\end{tikzpicture}
}}
\renewcommand\BRab{
\scalebox{.1}{\begin{tikzpicture}[rotate=-90,scale=1.3]
\braid[-,line width=6pt, style strands={1}{orange}, style strands={2}{violet}, style strands={3}{gray},number of strands=3,height=15mm] at (0,0) s_1^{-1}s_2^{-1};
\end{tikzpicture}
}}
\renewcommand\BRba{
\scalebox{.1}{\begin{tikzpicture}[rotate=-90,scale=1.3]
\braid[-,line width=6pt, style strands={1}{orange}, style strands={2}{violet}, style strands={3}{gray},number of strands=3,height=15mm] at (0,0) s_2^{-1}s_1^{-1};
\end{tikzpicture}
}}
\renewcommand\BRdS{
\scalebox{.1}{\begin{tikzpicture}[scale=1.6]
\braid[-,line width=6pt, style strands={1}{orange}, style strands={2}{violet}, style strands={3}{gray},number of strands=3,height=10mm] s_1^{-1} s_2^{-1} s_1^{-1};
\end{tikzpicture}
}}
\renewcommand\BRoneS{
\scalebox{.1}{\begin{tikzpicture}[scale=1.6]
\braid[-,line width=6pt, style strands={1}{orange}, style strands={2}{violet}, style strands={3}{gray},number of strands=3,border height=17.5mm] ;
\end{tikzpicture}
}}
\renewcommand\BRaS{
\scalebox{.1}{\begin{tikzpicture}[scale=1.6]
\braid[-,line width=6pt, style strands={1}{orange}, style strands={2}{violet}, style strands={3}{gray},number of strands=3,height=30mm] s_1^{-1};
\end{tikzpicture}
}}
\renewcommand\BRbS{
\scalebox{.1}{\begin{tikzpicture}[scale=1.6]
\braid[-,line width=6pt, style strands={1}{orange}, style strands={2}{violet}, style strands={3}{gray},number of strands=3,height=30mm] s_2^{-1};
\end{tikzpicture}
}}
\renewcommand\BRabS{
\scalebox{.1}{\begin{tikzpicture}[scale=1.6]
\braid[-,line width=6pt, style strands={1}{orange}, style strands={2}{violet}, style strands={3}{gray},number of strands=3,height=15mm] s_1^{-1}s_2^{-1};
\end{tikzpicture}
}}
\renewcommand\BRbaS{
\scalebox{.1}{\begin{tikzpicture}[scale=1.6]
\braid[-,line width=6pt, style strands={1}{orange}, style strands={2}{violet}, style strands={3}{gray},number of strands=3,height=15mm] s_2^{-1}s_1^{-1};
\end{tikzpicture}
}}

\newcommand\BRabSpec{
\scalebox{.1}{\begin{tikzpicture}[scale=1.6]
\braid[-,line width=6pt, style strands={1}{violet}, style strands={2}{gray}, style strands={3}{orange},number of strands=3,height=15mm] s_1^{-1}s_2^{-1};
\end{tikzpicture}
}}
\newcommand\BRdSpec{
\scalebox{.1}{\begin{tikzpicture}[scale=1.6]
\braid[-,line width=6pt, style strands={1}{violet}, style strands={2}{orange}, style strands={3}{gray},number of strands=3,height=15mm] s_1^{-1}s_2^{-1};
\end{tikzpicture}
}}

\newcommand\BThreeAutomaton{
\scalebox{.51}{
\begin{tikzpicture}[node distance=35mm,inner sep=1.5pt]
	 	\node[draw=orange,thin,circle,text width=4mm,align=center] (1) {$\BRoneS$};
	 	\node[draw=orange,thin,circle,text width=4mm,align=center] (a) [above left of=1] {$\BRaS$};
	 	\node[draw=orange,thin,circle,text width=4mm,align=center] (b) [above right of=1] {$\BRbS$};
	 	\node[draw=orange,thin,circle,text width=4mm,align=center] (ab) [above of=a] {$\BRabS$};
	 	\node[draw=orange,thin,circle,text width=4mm,align=center] (ba) [above of=b] {$\BRbaS$};
	 	\node[draw=orange,thin,circle,text width=4mm,align=center] (d) [above left of=ba] {$\BRdS$};
\begin{scope}[draw=couleurTransition,thick,couleurTransition]
\path[->,>=latex] (1)	edge[loop below]	node[scale=.8,align=center,below=.08]{$\BRone\sepon\BRone$\\$\BRd\sepon\BRd$}
								node[scale=.8,align=center,left=.5,pos=.12]{$\BRa\sepon\BRa$\\$\BRb\sepon\BRb$}
								node[scale=.8,align=center,right=.5,pos=.88]{$\BRab\sepon\BRab$\\$\BRba\sepon\BRba$}	(1);	
\path[<-,>=latex] (1)	edge			node[scale=.8,align=center,left=.08,pos=.20]{$\BRab\sepon\BRd$}
							node[scale=.8,align=center,left=.1,pos=.42]{$\BRb\sepon\BRba$}
							node[scale=.8,align=center,left=.1,pos=.64]{$\BRone\sepon\BRa$}
							node[scale=.8,align=center,left=1.2,pos=.7]{$\BRa\sepon\BRa$}
							node[scale=.8,right=.72,pos=.5]{$\BRone\sepon\BRd$}			 (a);
\path[<-,>=latex] (1)	edge			node[scale=.8,align=center,right=.08,pos=.20]{$\BRba\sepon\BRd$}
							node[scale=.8,align=center,right=.1,pos=.42]{$\BRa\sepon\BRab$}
							node[scale=.8,align=center,right=.1,pos=.64]{$\BRone\sepon\BRb$}
							node[scale=.8,align=center,right=1.2,pos=.7]{$\BRb\sepon\BRb$}	 (b);
\draw[<-,>=latex] (1) ..  controls  ($(1)+(-6cm,1cm)$) and ($(ab)+(-4cm,-4cm)$) ..  (ab);%
\draw[<-,>=latex] (1) ..  controls  ($(1)+(6cm,1cm)$) and ($(ba)+(4cm,-4cm)$) ..  (ba);%
\draw[<-,>=latex] (1) ..	controls  ($(1)+(0cm,1cm)$) and ($(b)+(-1.2cm,-1.2cm)$) ..  ($(b)+(-1.2cm,0cm)$)
				..	controls  ($(b)+(-1.2cm,1.2cm)$) and ($(ab)+(1.5cm,-1cm)$) ..  ($(ab)+(1.5cm,0cm)$)
				..	controls  ($(ab)+(1.5cm,1cm)$) and ($(d)+(0cm,-1cm)$) ..  (d);%
\path[->,>=latex] (a)	edge[out=240,in=210,looseness=8]	(a);
\path[->,>=latex] (a)	edge[bend left=20]		node[scale=.8,align=center,below]{$\BRd\sepon\BRd$}	 (b);
\path[<-,>=latex] (a)	edge		node[scale=.8,align=center,left=.05,pos=.57]{$\BRa\sepon\BRab$}
						node[scale=.8,align=center,left=.05,pos=.43]{$\BRab\sepon\BRd$}
						node[scale=.8,align=center,left=1.61,pos=-0.25]{$\BRone\sepon\BRab$\\$\BRb\sepon\BRd$}
						node[scale=.8,align=center,left=1.65,pos=1.25]{$\BRa\sepon\BRd$} (ab);
\draw[<-,>=latex] (a) ..  controls  ($(a)+(-4cm,4cm)$) and ($(d)+(-6cm,-1cm)$) ..  (d);%
\path[->,>=latex] (b)	edge[bend left=20]		node[scale=.8,align=center,above]{$\BRd\sepon\BRd$}	 (a);
\path[->,>=latex] (b)	edge[out=300,in=330,looseness=8]	(b);
\path[<-,>=latex] (b)	edge		node[scale=.8,align=center,right=.05,pos=.57]{$\BRb\sepon\BRba$}
						node[scale=.8,align=center,right=.05,pos=.43]{$\BRba\sepon\BRd$}
						node[scale=.8,align=center,right=1.61,pos=-0.25]{$\BRb\sepon\BRba$\\$\BRba\sepon\BRd$}
						node[scale=.8,align=center,right=1.65,pos=1.25]{$\BRb\sepon\BRd$} (ba);
\draw[<-,>=latex] (b) ..  controls  ($(b)+(4cm,4cm)$) and ($(d)+(+6cm,-1cm)$) ..  (d);%
\path[<-,>=latex] (ab)	edge	[out=-50,in=130]	node[scale=.8,align=center,left=.3,pos=.37]{$\BRab\sepon\BRb$}	 (b);
\path[->,>=latex] (ab)	edge[bend left=20]		node[scale=.8,align=center]{$\BRba\sepon\BRab$\\$\BRd\sepon\BRd$}	 (ba);
\path[<-,>=latex] (ab)	edge					node[scale=.8,align=center,left=.2,pos=.6]{$\BRab\sepon\BRd$}(d); 	
\path[<-,>=latex] (ba)	edge	[out=-130,in=50]	node[scale=.8,align=center,right=.3,pos=.37]{$\BRba\sepon\BRa$}	 (a);
\path[->,>=latex] (ba)	edge[bend left=20]		node[scale=.8,align=center]{$\BRab\sepon\BRba$\\$\BRd\sepon\BRd$}	 (ab);
\path[<-,>=latex] (ba)	edge					node[scale=.8,align=center,right=.2,pos=.6]{$\BRba\sepon\BRd$}		(d);
\path[->,>=latex] (d)	edge[loop above]	node[scale=.8,align=center,above=.08]{$\BRd\sepon\BRd$} (d);
\path[->,>=latex,opacity=0] (d)	edge[loop above]	node[scale=.8,align=center,above=.08]{$\BRICaS\sepp\BRICbS$\\$\BRICbS\sepp\BRICaS$} (d);
\end{scope}
\end{tikzpicture}
}
}

\section{Examples and counterexamples}\label{sec-ex}

Our very first example is straightforward, but enlightening.

\begin{example}\label{ex-finite} Every finite monoid~$\Fin$ (in particular every finite group) is an \emph{automaticon} monoid, that is, both an automatic and an automaton monoid.
Consider its quadratic normalisation~$(\Fin,\NM)$ with~$\NM(ab)=1(ab)$ for every~$(a,b)\in\Fin^2$.
Figure~\ref{fig-finite} shows how to compute its breadth~$(3,2)$, witness of Condition~\home\ for applying Theorem~\ref{thm-main}.
\end{example}

\begin{figure}[h]
\vspace*{-20pt}
\begin{minipage}[b]{.46\linewidth}
  \centering \rotatebox{-15}{
 \hspace*{-20pt}
 \begin{tikzpicture}[thick,node distance=15mm,inner sep=1.3pt]
	\node[inner sep=4pt] (00) at (0,0) {};
	\node (10) [above of=00] {};
	\node (11) [left of=10] 	{};
	\node (12) [right of=10] 	{};
	\node (13) [left of=11] 	{};
	\node (20) [above of=10] {};
	\node (21) [above of=11] {};
	\node (22) [above of=12] {};
	\node (30) [above of=20] {};
	\fill[color=lightgray] (0,0) -- (-3,1.5) -- (-3,3) -- (0,4.5) -- (0,0);
	\fill[color=lightgray] (0,0) -- (3,3) -- (0,4.5) -- (0,0);
	\draw[color=white,line width=9pt] (0,0) -- (-3,1.5) -- (-3,3) -- (0,4.5) -- (0,0);
	\draw[color=white,line width=9pt] (0,0) -- (3,3) -- (0,4.5) -- (0,0);
	\draw[color=white,line width=9pt] (0,1.5) -- (1.5,3);
	\draw[color=white,line width=9pt] (1.5,1.5) -- (1.5,3) -- (0,4.5);
	\draw[color=white,line width=9pt] (-3,1.5) -- (-1.5,3) -- (0,4.5);
	\draw[color=white,line width=9pt] (-1.5,1.5) -- (0,3);
	\draw[color=white,line width=9pt] (0,0) -- (-1.5,1.5) -- (-1.5,3);
	\path[->,>=latex,line width=1.4pt]	(30)	edge 		node[pos=.42,fill=white,rotate=15]{\footnotesize$\!a\!$} (20);
	\path[->,>=latex,line width=1.4pt]	(20)	edge			node[pos=.42,fill=white,rotate=15]{\footnotesize$\!b\!$} (10);
	\path[->,>=latex,line width=1.4pt]	(10)	edge			node[pos=.42,fill=white,rotate=15]{\footnotesize$\!c\!$} (00);
	\path[->,>=latex] 	(30)	edge			node[pos=.42,fill=white,rotate=15]{\footnotesize$\!1\!$}(21);
	\path[->,>=latex]	(21)	edge			node[pos=.42,fill=white,rotate=15]{\footnotesize$\!a\!$}(11);
	\path[->,>=latex]	(11)	edge			node[pos=.42,fill=white,rotate=15]{\footnotesize$\!bc\!$}(00);
	\path[->,>=latex]	(20)	edge			node[pos=.42,fill=white,rotate=15]{\footnotesize$\!1\!$}(11);
	\path[->,>=latex]	(30)	edge			node[pos=.42,fill=white,rotate=15]{\footnotesize$\!1\!$}(22);
	\path[->,>=latex]	(22)	edge			node[pos=.42,fill=white,rotate=15]{\footnotesize$\!ab\!$}(10);
	\path[->,>=latex]	(12)	edge			node[pos=.42,fill=white,rotate=15,below=-.06]{\footnotesize$abc$}(00);
	\path[->,>=latex]	(22)	edge			node[pos=.42,fill=white,rotate=15]{\footnotesize$\!1\!$}(12);
	\path[->,>=latex]	(21)	edge			node[pos=.42,fill=white,rotate=15]{\footnotesize$\!1\!$}(13);
	\path[->,>=latex]	(13)	edge			node[pos=.42,fill=white,rotate=15,below=-.03]{\footnotesize$abc$}(00);
	\path[->,>=latex]
	(10)	edge[densely dotted,white] 		node[fill=lightgray]{$\!\NM\!$}	(11b)
	(10)	edge[densely dotted,white] 		node[fill=lightgray]{$\!\NM\!$}	(12b)
	(11)	edge[densely dotted,white] 		node[fill=lightgray]{$\!\NM\!$}	(13b)
	(20)	edge[densely dotted,white] 		node[fill=lightgray]{$\!\NM\!$}	(21b)
	(20)	edge[densely dotted,white] 		node[fill=lightgray]{$\!\NM\!$}	(22b);
\end{tikzpicture}
 \hspace*{-20pt}
}
\vspace*{-30pt}
\caption{Computing the breadth~$(3,2)$ for any finite monoid~$\Fin$ as in~Example~\ref{ex-finite}.\label{fig-finite}}
\end{minipage} \hfill
 \begin{minipage}[b]{.46\linewidth}
  \centering\rotatebox{-15}{
 \hspace*{-60pt}
 \begin{tikzpicture}[thick,node distance=15mm,inner sep=1.3pt]
	\node[inner sep=4pt] (00) at (0,0) {};
	\node (10) [above of=00] {};
	\node (11) [left of=10] 	{};
	\node (12) [right of=10] 	{};
	\node (13) [left of=11] 	{};
	\node (14) [right of=12] 	{};
	\node (14b) [node distance=13mm,right of=12] {};
	\node (15) [left of=13] 	{};
	\node (20) [above of=10] {};
	\node (21) [above of=11] {};
	\node (22) [above of=12] {};
	\node (23) [above of=13] {};
	\node (24) [above of=14] {};
	\node (25) [above of=15] {};
	\node (30) [above of=20] {};
	\fill[color=lightgray] (0,0) -- (-3,1.5) -- (-3,3) -- (0,4.5) -- (0,0);
	\fill[color=lightgray] (0,0) -- (3,3) -- (0,4.5) -- (0,0);
	\draw[color=white,line width=9pt] (0,0) -- (-3,1.5) -- (-3,3) -- (0,4.5) -- (0,0);
	\draw[color=white,line width=9pt] (0,0) -- (3,3) -- (0,4.5) -- (0,0);
	\draw[color=white,line width=9pt] (0,1.5) -- (1.5,3);
	\draw[color=white,line width=9pt] (1.5,1.5) -- (1.5,3) -- (0,4.5);
	\draw[color=white,line width=9pt] (-3,1.5) -- (-1.5,3) -- (0,4.5);
	\draw[color=white,line width=9pt] (-1.5,1.5) -- (0,3);
	\draw[color=white,line width=9pt] (0,0) -- (-1.5,1.5) -- (-1.5,3);
	\path[->,>=latex,line width=1.4pt]	(30)	edge 		node[pos=.42,fill=white,rotate=15]{\footnotesize$\!x\!$} (20);
	\path[->,>=latex,line width=1.4pt]	(20)	edge			node[pos=.42,fill=white,rotate=15]{\footnotesize$\!\tta\!$} (10);
	\path[->,>=latex,line width=1.4pt]	(10)	edge			node[pos=.42,fill=white,rotate=15]{\footnotesize$\!\ttb\!$} (00);
	\path[->,>=latex] 	(30)	edge 		node[pos=.42,fill=white,rotate=15]{\footnotesize$\!\ttunit\!$} (21);
	\path[->,>=latex]	(21)	edge 		node[pos=.42,fill=white,rotate=15]{\footnotesize$x$} (11);
	\path[->,>=latex]	(11)	edge 		node[pos=.42,fill=white,rotate=15]{\footnotesize$\!\ttunit\!$} (00);
	\path[->,>=latex]	(20)	edge 		node[pos=.42,fill=white,rotate=15]{\footnotesize$\!\ttunit\!$} (11);
	\path[->,>=latex]	(30)	edge 		node[pos=.42,fill=white,rotate=15]{\footnotesize$x$} (22);
	\path[->,>=latex]	(22)	edge 		node[pos=.42,fill=white,rotate=15]{\footnotesize$\tta$} (10);
	\path[->,>=latex]	(12)	edge 		node[pos=.42,fill=white,rotate=15]{\footnotesize$\!\ttunit\!$} (00);
	\path[->,>=latex]	(22)	edge 		node[pos=.42,fill=white,rotate=15]{\footnotesize$\!\ttunit\!$} (12);
	\path[->,>=latex]	(21)	edge 		node[pos=.42,fill=white,rotate=15]{\footnotesize$\!\ttunit\!$} (13);
	\path[->,>=latex]	(13)	edge 		node[pos=.42,fill=white,rotate=15]{\footnotesize$x$} (00);
	\path[->,>=latex]	(30)	edge 		node[pos=.42,fill=white,rotate=15]{\footnotesize$\!\ttunit\!$} (24);
	\path[->,>=latex]	(24)	edge 		node[pos=.42,fill=white,rotate=15]{\footnotesize$x$} (12);
	\path[->,>=latex]	(24)	edge 		node[pos=.42,fill=white,rotate=15]{\footnotesize$\!\ttunit\!$} (14);
	\path[->,>=latex]	(14)	edge 		node[pos=.42,fill=white,rotate=15]{\footnotesize$x$} (00);
	\path[->,>=latex]
	(10)	edge[densely dotted,white] 		node[fill=lightgray]{$\!\NM\!$}	(11b)
	(10)	edge[densely dotted,white] 		node[fill=lightgray]{$\!\NM\!$}	(12b)
	(11)	edge[densely dotted,white] 		node[fill=lightgray]{$\!\NM\!$}	(13b)
	(12)	edge[thin,densely dotted,gray] 		node[fill=white]{$\!\NM\!$}	(14b)
	(20)	edge[densely dotted,white] 		node[fill=lightgray]{$\!\NM\!$}	(21b)
	(20)	edge[densely dotted,white] 		node[fill=lightgray]{$\!\NM\!$}	(22b)
	(22)	edge[densely dotted,white] 		node[fill=lightgray]{$\!\NM\!$}	(24b)
	(22)	edge[densely dotted,white] 		node[fill=lightgray]{$\!\NM\!$}	(24b);
\end{tikzpicture}
 \hspace*{-20pt}
}
\vspace*{-28pt}
\caption{Computing the breadth~$(3,4)$
for the bicylic monoid~$\Bic$ from~Example~\ref{ex-bicyclic}.
\label{fig-bicyclic}}
 \end{minipage}
\end{figure}

As mentioned in Section~\ref{sec-intro} and appearing on~Figure~\ref{fig-picture}, there exist automatic semigroups that cannot be automaton semigroups.

\begin{example}\label{ex-bicyclic} The bicyclic monoid~$\Bic=\presm{~\tta,\ttb:\tta\ttb=\ttunit~}$ is known to be automatic and not residually finite,
hence cannot be an automaton monoid.
Choose for~$\Bic$ the quadratic normalisation~$(\{\tta,\ttb,\ttunit\},\NM)$ with~$\NM(\tta\ttb)=\ttunit\ttunit$, $\NM(x\ttunit)=\ttunit x$ for~$x\in\{\tta,\ttb\}$,
and~$\NM(xy)=xy$ otherwise. Figure~\ref{fig-bicyclic} illustrates the computation (on the witness word~$x\tta\ttb$ with~$x\in\{\tta,\ttb\}$) of its breadth~$(3,4)$.
The Condition~\home\ is hence not satisfied and Theorem~\ref{thm-main} cannot apply. Precisely, according to the proof of Proposition~\ref{prop-home} and Figure~\ref{fig-proof},
we have $\sigma_{\tta\ttb}(x)=\ttunit\not=x=\sigma_{\ttunit\ttunit}(x)$ for~$x\in\{\tta,\ttb\}$, hence $\sigma_{\tta\ttb}\not=\sigma_{\ttunit}=\sigma_{\ttunit\ttunit}$.
\end{example}

By contrast, one of the simplest nontrivial examples could be the following.

\begin{example}\label{ex-bsOneZero} The automatic monoid~$\presm{~\tta,\ttb:\tta\ttb=\tta~}$
admits the quadratic normalisation~$(\{\tta,\ttb,\ttunit\},\NM)$ with~$\NM(\tta\ttb)=\ttunit\tta$, $\NM(x\ttunit)=\ttunit x$ for~$x\in\{\tta,\ttb\}$,
and~$\NM(xy)=xy$ otherwise. Condition~\cunit\ is satisfied, and the breadth is~(3,3) according to the graph on Figure~\ref{fig-bsOneZero}.
By Theorem~\ref{thm-main}, it is therefore an automaton monoid, generated by the Mealy automaton displayed on Figure~\ref{fig-bsOneZero}.

The latter happens to be the common smallest nontrivial member of the family of Baumslag--Solitar monoids (see \cite{HoffmannPhD} for instance), namely~$\BS^{\ttunit}_{\hspace*{-1pt}+}\!(1,0)$, and
of a wide family of right-cancellative semigroups, that we readily call~\emph{Artin--Krammer}
monoids and that have been introduced and studied in~\cite{kram} (see also~\cite{Hes, HeOz, Ozo}), namely~$\AK^{\ttunit}_{\hspace*{-1pt}+}\!(\Gamma)$ associated with
the Coxeter-like matrix~$\Gamma=\ $\scalebox{.7}{$\begin{bmatrix}1&1\\2&1\end{bmatrix}$}.

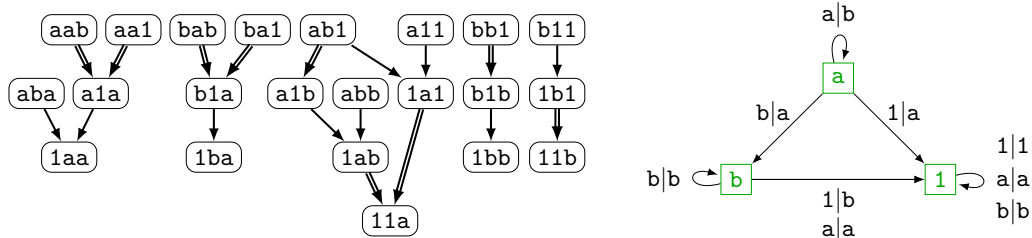
\begin{figure}[h]
\centering
\vspace*{-10pt}
	\scalebox{.95}{
	\begin{tikzpicture}
	\begin{scope}[>=latex,xscale=.4,yscale=.35]
  \begin{scope}[inner sep=3pt,rounded corners,minimum height=3.5mm]
  \node (uau) at (404.0bp,162.0bp) [draw] {$\ttunit\tta\ttunit$};
  \node (uaa) at (54.0bp,90.0bp) [draw] {$\ttunit\tta\tta$};
  \node (uab) at (340.0bp,90.0bp) [draw] {$\ttunit\tta\ttb$};
  \node (aua) at (85.0bp,162.0bp) [draw] {$\tta\ttunit\tta$};
  \node (aba) at (22.0bp,162.0bp) [draw] {$\tta\ttb\tta$};
  \node (abb) at (340.0bp,162.0bp) [draw] {$\tta\ttb\ttb$};
  \node (aub) at (276.0bp,162.0bp) [draw] {$\tta\ttunit\ttb$};
  \node (auu) at (404.0bp,234.0bp) [draw] {$\tta\ttunit\ttunit$};
  \node (abu) at (308.0bp,234.0bp) [draw] {$\tta\ttb\ttunit$};
  \node (bbu) at (468.0bp,234.0bp) [draw] {$\ttb\ttb\ttunit$};
  \node (ubu) at (533.0bp,162.0bp) [draw] {$\ttunit\ttb\ttunit$};
  \node (buu) at (533.0bp,234.0bp) [draw] {$\ttb\ttunit\ttunit$};
  \node (uba) at (196.0bp,90.0bp) [draw] {$\ttunit\ttb\tta$};
  \node (ubb) at (468.0bp,90.0bp) [draw] {$\ttunit\ttb\ttb$};
  \node (bub) at (468.0bp,162.0bp) [draw] {$\ttb\ttunit\ttb$};
  \node (bua) at (196.0bp,162.0bp) [draw] {$\ttb\ttunit\tta$};
  \node (uua) at (369.0bp,18.0bp) [draw] {$\ttunit\ttunit\tta$};
  \node (uub) at (533.0bp,90.0bp) [draw] {$\ttunit\ttunit\ttb$};
  \node (bau) at (244.0bp,234.0bp) [draw] {$\ttb\tta\ttunit$};
  \node (aab) at (54.0bp,234.0bp) [draw] {$\tta\tta\ttb$};
  \node (aau) at (117.0bp,234.0bp) [draw] {$\tta\tta\ttunit$};
  \node (bab) at (180.0bp,234.0bp) [draw] {$\ttb\tta\ttb$};
  \draw [double,->,thick] (bbu) ..controls (468.0bp,207.98bp) and (468.0bp,198.71bp)  .. (bub);
  \draw [->,thick] (aba) ..controls (33.259bp,136.37bp) and (38.03bp,125.93bp)  .. (uaa);
  \draw [->,thick] (auu) ..controls (404.0bp,207.98bp) and (404.0bp,198.71bp)  .. (uau);
  \draw [double,->,thick] (aab) ..controls (64.907bp,208.37bp) and (69.529bp,197.93bp)  .. (aua);
  \draw [->,thick] (bua) ..controls (196.0bp,135.98bp) and (196.0bp,126.71bp)  .. (uba);
  \draw [->,thick] (abb) ..controls (340.0bp,135.98bp) and (340.0bp,126.71bp)  .. (uab);
  \draw [double,->,thick] (ubu) ..controls (533.0bp,135.98bp) and (533.0bp,126.71bp)  .. (uub);
  \draw [->,thick] (aua) ..controls (74.093bp,136.37bp) and (69.471bp,125.93bp)  .. (uaa);
  \draw [double,->,thick] (bau) ..controls (227.36bp,208.73bp) and (219.31bp,197.0bp)  .. (bua);
  \draw [double,->,thick] (aau) ..controls (105.74bp,208.37bp) and (100.97bp,197.93bp)  .. (aua);
  \draw [->,thick] (buu) ..controls (533.0bp,207.98bp) and (533.0bp,198.71bp)  .. (ubu);
  \draw [->,thick] (bub) ..controls (468.0bp,135.98bp) and (468.0bp,126.71bp)  .. (ubb);
  \draw [->,thick] (abu) ..controls (339.41bp,210.1bp) and (362.14bp,193.53bp)  .. (uau);
  \draw [double,->,thick] (uau) ..controls (393.85bp,119.81bp) and (382.7bp,74.568bp)  .. (uua);
  \draw [double,->,thick] (uab) ..controls (350.15bp,64.491bp) and (354.4bp,54.232bp)  .. (uua);
  \draw [->,thick] (aub) ..controls (297.89bp,137.06bp) and (309.94bp,123.88bp)  .. (uab);
  \draw [double,->,thick] (bab) ..controls (185.68bp,208.14bp) and (187.88bp,198.54bp)  .. (bua);
  \draw [double,->,thick] (abu) ..controls (296.74bp,208.37bp) and (291.97bp,197.93bp)  .. (aub);
\end{scope}
\end{scope}
\begin{scope}[->,>=latex,node distance=20mm,xshift=10cm,yshift=8mm]
	 	\node[etatDualDual] (two) {\(\ttb\)};
	 	\node[etatDualDual] (one) [above right of=two] {\(\tta\)};
	 	\node[etatDualDual] (thr) [below right of=one] {\(\ttunit\)};
	\path (one) edge[loop above] 	node[text width=6mm,align=center,above]{\(\tta|\ttb\)} (one);
	\path (one) edge 			node[text width=6mm,align=center,above=.1,pos=.7]{\(\ttb|\tta\)} (two);
	\path (one) edge 			node[text width=9mm,align=center,above=.1,pos=.7]{\(\ttunit|\tta\)} (thr);
	\path (two) edge[loop left] 		node[text width=6mm,align=center,left]{\(\ttb|\ttb\)} (two);
	\path (two) edge	 		node[text width=7mm,align=center,below]{\(\ttunit|\ttb\)} (thr);
	\path (two) edge[opacity=1] 	node[text width=7mm,align=center,below=.4,opacity=1]{\(\tta|\tta\)} (thr);
	\path (thr)  edge[loop right] 	node[text width=6mm,align=center,right]{\(\ttunit|\ttunit\)\\ \(\tta|\tta\)\\ \(\ttb|\ttb\)} (thr);
\end{scope}
\end{tikzpicture}
	}
\vspace*{-10pt}
\caption{The $\NMbar$-graph for the quadratic normalisation associated with~$\presm{~\tta,\ttb:\tta\ttb=\tta~}$ from~Example~\ref{ex-bsOneZero}: simple arrows correspond to~$\NMbar_1$ and double arrows to~$\NMbar_2$, while loops are simply omitted for better readability. The breadth is $(3,3)$ as well. On the right is the associated Mealy automaton.
}%
\label{fig-bsOneZero}
\end{figure}
\end{example}

Examples~\ref{ex-akram} and~\ref{ex-bsThreeTwo} describe important members from both these families.

\begin{example}\label{ex-akram} The following Artin--Krammer monoid is emblematic:%
\[\AK^{\ttunit}_{\hspace*{-1pt}+}\!\Big(\scalebox{.54}{
$\begin{bmatrix}1&3&2\\
{\color{orange}4}&1&3\\
2&{\color{orange}4}&1\end{bmatrix}$}\Big)
=\Bigg\langle \tta,\ttb,\ttc:
\scalebox{1}{$\begin{matrix}{\color{orange}\tta}\ttb\tta\ttb=\tta\ttb\tta\\
\tta\ttc=\ttc\tta\\
{\color{orange}\ttb}\ttc\ttb\ttc=\ttb\ttc\ttb\end{matrix}$}\ \Bigg\rangle^{\ttunit}_{\hspace*{-1pt}+}.
\]As displayed on Figure~\ref{fig-akram}, its minimal so-called Garside family forms like a flint which encodes its whole combinatorics
and, according to Theorem~\ref{thm-main}, makes it an automaticon monoid.
\end{example}

\begin{figure}[t]
 \begin{minipage}[b]{.46\linewidth}
  \centering \scalebox{0.63}{
\begin{tikzpicture}[<-,>=latex,scale=.6,rotate=-90]
\node (0) at (94.0bp,486.0bp) [fill=orange,ellipse] {};
\node (a) at (49.0bp,342.0bp) [fill=orange,ellipse] {};
\node (b) at (150.0bp,390.0bp) [fill=orange,ellipse] {};
\node (c) at (94.0bp,438.0bp) [fill=orange,ellipse] {};
\draw [orange,very thick] (0) ..controls (89.469bp,473.3bp) and (83.479bp,457.53bp)  .. (79.0bp,444.0bp) .. controls (66.811bp,407.19bp) and (54.236bp,362.08bp)  .. (a);
\draw [violet,very thick] (0) ..controls (105.53bp,465.64bp) and (137.98bp,411.18bp)  .. (b);
\draw [black,opacity=.5,very thick,dotted] (0) ..controls (94.0bp,471.6bp) and (94.0bp,452.36bp)  .. (c);
\node (ab) at (107.0bp,246.0bp) [draw,ellipse] {};
\node (aba) at (121.0bp,198.0bp) [fill=orange,ellipse] {};
\node (bab) at (145.0bp,246.0bp) [draw,ellipse] {};
\node (ba) at (150.0bp,342.0bp) [fill=orange,ellipse] {};
\node (ac) at (35.0bp,294.0bp) [fill=orange,ellipse] {};
\node (bcb) at (237.0bp,294.0bp) [fill=orange,ellipse] {};
\node (bc) at (208.0bp,342.0bp) [draw,ellipse] {};
\node (cbc) at (179.0bp,342.0bp) [draw,ellipse] {};
\node (cb) at (121.0bp,390.0bp) [fill=orange,ellipse] {};
\draw [orange,very thick] (ab) ..controls (111.08bp,231.6bp) and (116.93bp,212.36bp)  .. (aba);
\draw [orange,very thick] (bab) ..controls (138.19bp,231.95bp) and (127.67bp,211.79bp)  .. (aba);
\draw [violet,very thick] (a) ..controls (61.033bp,321.5bp) and (95.169bp,266.17bp)  .. (ab);
\draw [violet,very thick] (ba) ..controls (148.92bp,320.62bp) and (146.11bp,267.82bp)  .. (bab);
\draw [orange,very thick] (b) ..controls (150.0bp,375.6bp) and (150.0bp,356.36bp)  .. (ba);
\draw [orange,very thick] (c) ..controls (90.055bp,417.31bp) and (79.861bp,371.04bp)  .. (63.0bp,336.0bp) .. controls (55.824bp,321.09bp) and (43.662bp,305.48bp)  .. (ac);
\draw [black,opacity=.5,very thick,dotted] (a) ..controls (44.923bp,327.6bp) and (39.067bp,308.36bp)  .. (ac);
\draw [violet,very thick] (bc) ..controls (215.95bp,328.39bp) and (229.06bp,307.59bp)  .. (bcb);
\draw [violet,very thick] (cbc) ..controls (193.29bp,329.67bp) and (222.51bp,306.49bp)  .. (bcb);
\draw [black,opacity=.5,very thick,dotted] (b) ..controls (164.29bp,377.67bp) and (193.51bp,354.49bp)  .. (bc);
\draw [black,opacity=.5,very thick,dotted] (cb) ..controls (135.29bp,377.67bp) and (164.51bp,354.49bp)  .. (cbc);
\draw [violet,very thick] (c) ..controls (101.66bp,423.95bp) and (113.5bp,403.79bp)  .. (cb);
\node (acba) at (34.0bp,198.0bp) [fill=orange,ellipse] {};
\node (acb) at (34.0bp,246.0bp) [draw,ellipse] {};
\node (cbab) at (63.0bp,246.0bp) [draw,ellipse] {};
\node (cba) at (121.0bp,342.0bp) [fill=orange,ellipse] {};
\node (bacba) at (268.0bp,150.0bp) [fill=orange,ellipse] {};
\node (bacb) at (208.0bp,246.0bp) [draw,ellipse] {};
\node (bcbab) at (275.0bp,198.0bp) [draw,ellipse] {};
\node (cbacb) at (179.0bp,246.0bp) [draw,ellipse] {};
\node (bac) at (208.0bp,294.0bp) [draw,ellipse] {};
\node (cbac) at (179.0bp,294.0bp) [draw,ellipse] {};
\node (bcba) at (275.0bp,246.0bp) [draw,ellipse] {};
\draw [orange,very thick] (acb) ..controls (34.0bp,231.6bp) and (34.0bp,212.36bp)  .. (acba);
\draw [orange,very thick] (cbab) ..controls (55.049bp,232.39bp) and (41.939bp,211.59bp)  .. (acba);
\draw [violet,very thick] (ac) ..controls (34.709bp,279.6bp) and (34.29bp,260.36bp)  .. (acb);
\draw [violet,very thick] (cba) ..controls (108.97bp,321.5bp) and (74.831bp,266.17bp)  .. (cbab);
\draw [orange,very thick] (cb) ..controls (121.0bp,375.6bp) and (121.0bp,356.36bp)  .. (cba);
\draw [orange,very thick] (bacb) ..controls (221.75bp,237.23bp) and (244.67bp,222.76bp)  .. (256.0bp,204.0bp) .. controls (265.56bp,188.17bp) and (267.51bp,165.51bp)  .. (bacba);
\draw [orange,very thick] (cbacb) ..controls (190.25bp,234.89bp) and (208.04bp,218.63bp)  .. (222.0bp,204.0bp) .. controls (238.99bp,186.19bp) and (257.81bp,163.49bp)  .. (bacba);
\draw [orange,very thick] (bcbab) -- (bacba);
\draw [violet,very thick] (bac) ..controls (208.0bp,279.6bp) and (208.0bp,260.36bp)  .. (bacb);
\draw [black,opacity=.5,very thick,dotted] (ba) ..controls (164.29bp,329.67bp) and (193.51bp,306.49bp)  .. (bac);
\draw [orange,very thick] (bc) ..controls (208.0bp,327.6bp) and (208.0bp,308.36bp)  .. (bac);
\draw [violet,very thick] (cbac) ..controls (179.0bp,279.6bp) and (179.0bp,260.36bp)  .. (cbacb);
\draw [orange,very thick] (cbc) ..controls (179.0bp,327.6bp) and (179.0bp,308.36bp)  .. (cbac);
\draw [black,opacity=.5,very thick,dotted] (cba) ..controls (135.29bp,329.67bp) and (164.51bp,306.49bp)  .. (cbac);
\draw [violet,very thick] (bcba) -- (bcbab);
\draw [orange,very thick] (bcb) ..controls (247.05bp,280.83bp) and (264.77bp,259.39bp)  .. (bcba);
\node (abac) at (121.0bp,150.0bp) [draw,ellipse] {};
\node (abacb) at (123.0bp,54.0bp) [draw,ellipse] {};
\node (abacba) at (108.0bp,6.0bp) [fill=orange,ellipse] {};
\node (abc) at (92.0bp,198.0bp) [draw,ellipse] {};
\node (abcb) at (60.0bp,150.0bp) [draw,ellipse] {};
\node (abcba) at (60.0bp,102.0bp) [draw,ellipse] {};
\node (abcbab) at (65.0bp,54.0bp) [draw,ellipse] {};
\node (acbac) at (31.0bp,150.0bp) [draw,ellipse] {};
\node (acbacb) at (31.0bp,102.0bp) [draw,ellipse] {};
\node (acbc) at (5.0bp,198.0bp) [draw,ellipse] {};
\node (babc) at (150.0bp,198.0bp) [draw,ellipse] {};
\node (babcb) at (150.0bp,150.0bp) [draw,ellipse] {};
\node (babcba) at (157.0bp,102.0bp) [draw,ellipse] {};
\node (babcbab) at (154.0bp,54.0bp) [draw,ellipse] {};
\node (bacbac) at (204.0bp,102.0bp) [draw,ellipse] {};
\node (bacbacb) at (193.0bp,54.0bp) [draw,ellipse] {};
\node (bacbc) at (208.0bp,198.0bp) [draw,ellipse] {};
\node (cbabc) at (63.0bp,198.0bp) [draw,ellipse] {};
\node (cbabcb) at (89.0bp,150.0bp) [draw,ellipse] {};
\node (cbabcba) at (89.0bp,102.0bp) [draw,ellipse] {};
\node (cbabcbab) at (94.0bp,54.0bp) [draw,ellipse] {};
\node (cbacbc) at (179.0bp,198.0bp) [draw,ellipse] {};
\draw [black,opacity=.5,very thick,dotted] (ab) ..controls (102.63bp,231.6bp) and (96.357bp,212.36bp)  .. (abc);
\draw [black,opacity=.5,very thick,dotted] (bab) ..controls (146.46bp,231.6bp) and (148.55bp,212.36bp)  .. (babc);
\draw [violet,very thick] (babc) ..controls (150.0bp,183.6bp) and (150.0bp,164.36bp)  .. (babcb);
\draw [orange,very thick] (abc) ..controls (99.951bp,184.39bp) and (113.06bp,163.59bp)  .. (abac);
\draw [violet,very thick] (bacbc) ..controls (193.71bp,185.67bp) and (164.49bp,162.49bp)  .. (babcb);
\draw [violet,very thick] (cbabc) ..controls (70.378bp,183.95bp) and (81.773bp,163.79bp)  .. (cbabcb);
\draw [orange,very thick] (babcb) ..controls (152.04bp,135.6bp) and (154.97bp,116.36bp)  .. (babcba);
\draw [violet,very thick] (cbacbc) ..controls (159.49bp,187.03bp) and (108.94bp,161.19bp)  .. (cbabcb);
\draw [orange,very thick] (babcbab) ..controls (142.29bp,41.287bp) and (120.17bp,19.168bp)  .. (abacba);
\draw [black,opacity=.5,very thick,dotted] (bacb) ..controls (208.0bp,231.6bp) and (208.0bp,212.36bp)  .. (bacbc);
\draw [violet,very thick] (acbac) ..controls (31.0bp,135.6bp) and (31.0bp,116.36bp)  .. (acbacb);
\draw [orange,very thick] (acbacb) ..controls (33.328bp,86.687bp) and (38.29bp,63.409bp)  .. (50.0bp,48.0bp) .. controls (65.071bp,28.168bp) and (92.47bp,14.106bp)  .. (abacba);
\draw [orange,very thick] (abcb) ..controls (60.0bp,135.6bp) and (60.0bp,116.36bp)  .. (abcba);
\draw [black,opacity=.5,very thick,dotted] (acba) ..controls (33.126bp,183.6bp) and (31.871bp,164.36bp)  .. (acbac);
\draw [orange,very thick] (bacbacb) ..controls (173.67bp,42.54bp) and (127.38bp,17.49bp)  .. (abacba);
\draw [black,opacity=.5,very thick,dotted] (cbacb) ..controls (179.0bp,231.6bp) and (179.0bp,212.36bp)  .. (cbacbc);
\draw [orange,very thick] (cbabcbab) ..controls (98.077bp,39.604bp) and (103.93bp,20.363bp)  .. (abacba);
\draw [orange,very thick] (cbabc) ..controls (54.226bp,184.39bp) and (39.761bp,163.59bp)  .. (acbac);
\draw [violet,very thick] (abcba) ..controls (61.456bp,87.604bp) and (63.548bp,68.363bp)  .. (abcbab);
\draw [orange,very thick] (abacb) ..controls (118.63bp,39.604bp) and (112.36bp,20.363bp)  .. (abacba);
\draw [black,opacity=.5,very thick,dotted] (bacba) ..controls (252.12bp,137.59bp) and (219.36bp,114.04bp)  .. (bacbac);
\draw [orange,very thick] (abcbab) ..controls (76.452bp,40.749bp) and (96.819bp,18.962bp)  .. (abacba);
\draw [orange,very thick] (acbc) ..controls (12.378bp,183.95bp) and (23.773bp,163.79bp)  .. (acbac);
\draw [orange,very thick] (babc) ..controls (142.05bp,184.39bp) and (128.94bp,163.59bp)  .. (abac);
\draw [violet,very thick] (abac) ..controls (121.43bp,128.62bp) and (122.56bp,75.823bp)  .. (abacb);
\draw [violet,very thick] (acbc) ..controls (18.55bp,185.67bp) and (46.263bp,162.49bp)  .. (abcb);
\draw [violet,very thick] (abc) ..controls (83.226bp,184.39bp) and (68.761bp,163.59bp)  .. (abcb);
\draw [black,opacity=.5,very thick,dotted] (aba) ..controls (121.0bp,183.6bp) and (121.0bp,164.36bp)  .. (abac);
\draw [violet,very thick] (cbabcba) ..controls (90.456bp,87.604bp) and (92.548bp,68.363bp)  .. (cbabcbab);
\draw [orange,very thick] (cbabcb) ..controls (89.0bp,135.6bp) and (89.0bp,116.36bp)  .. (cbabcba);
\draw [orange,very thick] (cbacbc) ..controls (184.42bp,176.62bp) and (198.46bp,123.82bp)  .. (bacbac);
\draw [orange,very thick] (bacbc) ..controls (207.13bp,176.62bp) and (204.89bp,123.82bp)  .. (bacbac);
\draw [black,opacity=.5,very thick,dotted] (cbab) ..controls (63.0bp,231.6bp) and (63.0bp,212.36bp)  .. (cbabc);
\draw [violet,very thick] (bacbac) ..controls (200.8bp,87.604bp) and (196.2bp,68.363bp)  .. (bacbacb);
\draw [black,opacity=.5,very thick,dotted] (acb) ..controls (26.049bp,232.39bp) and (12.939bp,211.59bp)  .. (acbc);
\draw [violet,very thick] (a) ..controls (61.033bp,321.5bp) and (95.169bp,266.17bp)  .. (ab);
\draw [violet,very thick] (babcba) ..controls (156.13bp,87.604bp) and (154.87bp,68.363bp)  .. (babcbab);
\end{tikzpicture}
\hspace*{-30pt}
}
\vspace*{-10pt}
\caption{The minimal Garside family of the monoid~$\AK^{\ttunit}_{\hspace*{-1pt}+}\!\Big(\scalebox{.54}{
$\begin{bmatrix}1&3&2\\
{\color{orange}4}&1&3\\
2&{\color{orange}4}&1\end{bmatrix}$}\Big)$ from~Example~\ref{ex-akram}.\label{fig-akram}}
\end{minipage} \hfill
\begin{minipage}[b]{.46\linewidth}
 \centering\scalebox{.94}{
	\begin{tikzpicture}[->,>=latex,node distance=12mm,inner sep=1.8pt]
	 	\node[circle,text width=4mm,align=center,fill=orange] (0) {\(\ttunit\)};
	 	\node[circle,text width=4mm,align=center,fill=orange] (a) [above left of=0] {\(\tta\)};
	 	\node[circle,text width=4mm,align=center,fill=orange] (b) [below left of=0] {\(\ttb\)};
	 	\node[circle,text width=4mm,align=center,fill=orange] (ab) [below left of=b] {\(\tta\ttb\)};
	 	\node[circle,text width=4mm,align=center,fill=orange,draw=white] (bb) [above left of=b] {\(\ttb^2\)};
	 	\node[circle,text width=4mm,align=center,fill=orange] (abb) [below left of=bb] {\(\!\tta\ttb^2\!\)};
	 	\node[circle,text width=4mm,align=center,draw] 	 (bbb) [above left of=abb] {\(\ttb^3\)};
	 	\node[circle,text width=4mm,align=center,draw] 	 (ba) [above left of=bbb] {\(\ttb\tta\)};
	 	\node[circle,text width=4mm,align=center,draw] 	 (bbbb) [below left of=bbb] {\(\ttb^4\)};
	 	\node[circle,text width=4mm,align=center,draw] 	 (bab) [below left of=bbbb] {\(\!\ttb\tta\ttb\!\!\)};
	 	\node[circle,text width=4mm,align=center,fill=orange] (bba) [above left of=bbbb] {\(\!\tta\ttb^3\!\)};
	 	\node[circle,text width=4mm,align=center,fill=orange] (abbbb) [above left of=bab] {\(\!\tta\ttb^4\!\)};
	\path[line width=1pt]
		(a) edge[orange]	(0)
		(b) edge[violet]		(0)
		(ab) edge[orange]	(b)
		(bb) edge[violet]	(b)
		(ba) edge[violet]	(a)
		(bba) edge[violet]	(ba)
		(bbb) edge[violet]	(bb)
		(bba) edge[orange]	(bbb)
		(bab) edge[violet]	(ab)
		(abb) edge[orange]	(bb)
		(bbbb) edge[violet]	(bbb)
		(abbbb) edge[orange]	(bbbb)
		(abbbb) edge[violet]	(bab);
	\end{tikzpicture}
	}
\vspace*{-6pt}
\caption{\label{fig-bsThreeTwo}The minimal Garside family of the monoid~$\BS^{\ttunit}_{\hspace*{-1pt}+}\!(3,2)$
from~Example~\ref{ex-bsThreeTwo}.\phantom{\scalebox{.54}{
$\begin{bmatrix}1\\
3\\
1\end{bmatrix}$}
}}
 \end{minipage}
\end{figure}

\begin{example}\label{ex-bsThreeTwo}
Consider the Baumslag--Solitar monoid~$\BS^{\ttunit}_{\hspace*{-1pt}+}\!(3,2)=\presm{~\tta,\ttb:\tta\ttb^3=\ttb^2\tta~}$.
Displayed on Figure~\ref{fig-bsThreeTwo}, its minimal Garside family contains eight elements (orange vertices)
and makes it an automaticon monoid. This is an example of a group-embeddable automaton 
monoid whose enveloping group is not an automaton group. Indeed, the Baumslag--Solitar group
$\BS(3,2)$ is precisely known as an example of non-residually finite group, hence cannot
be an automaton group. The question remains open for those
automaton semigroups whose enveloping group is a group of fractions.
\end{example}

Concerning again group-embeddability, the following gives now an example
of a cancellative automaton semigroup which is not group-embeddable.

\begin{example}\label{ex-malcev}
The monoid~$\Malcev=\presm{~\tta,\ttb,\ttc,\ttd,\tta',\ttb',\ttc',\ttd': \tta\ttb = \ttc\ttd,\tta'\ttb'=\ttc'\ttd',\tta'\ttd=\ttc'\ttb~}$
is known (by Malcev work~\cite{malcev37,malcev39,malcev40}) to be cancellative but not group-embeddable:
from these three relations, we cannot deduce the relation~$\tta\ttd'=\ttc\ttb'$ that holds in the enveloping group.
The quadratic normalisation~$(\{\tta,\ttb,\ttc,\ttd,\tta',\ttb',\ttc',\ttd'\},\NM)$
defined by~$\NM(\tta\ttb)=\ttc\ttd$, $\NM(\tta'\ttb')=\ttc'\ttd'$, and~$\NM(\tta'\ttd)=\ttc'\ttb$
for instance has breadth~$(3,3)$,
hence satisfies Condition~\home\ and Theorem~\ref{thm-main} applies.
This answers in particular a question by Cain \cite{cain-perso}.
\end{example}

Some classes of neither left- nor right-cancellative monoids have been studied and shown
to admit nice normal forms yielding biautomatic structures:

\begin{example}\label{ex-plactic} According to Sch\"utzenberger~\cite{pplax}, plactic monoids are among the most fundamental monoids: they are monoids of Young tableaux.
The rank~2 plactic monoid \hbox{is~${\mathbf P\!}_2=\presm{~\tta,\ttb:\tta\ttb\tta=\ttb\tta\tta,\ttb\tta\ttb=\ttb\ttb\tta~}$.}
As noted in~\cite{dehDLT,dehgui}, ${\mathbf P\!}_2$ admits the quadratic normalisation~$(\Gar,\NM)$ with~$\Gar=\{\ttunit,\tta,\ttb,\ttb\tta\}$,
$\NM(\ttb\tta)=\ttunit(\ttb\tta)$, $\NM((\ttb\tta)\tta)=\tta(\ttb\tta)$,
$\NM((\ttb\tta)\ttb)=\ttb(\ttb\tta)$, $\NM(\ttunit x)=x\ttunit$ for~$x\in\Gar$, and~$\NM(xy)=xy$ otherwise. The latter has a breadth~$(3,3)$, hence satisfies Condition~\home\ and Theorem~\ref{thm-main} ensures that~${\mathbf P\!}_2$ is an automaton monoid.
Note that, for a higher rank plactic monoid~${\mathbf P\!}_X$,
it suffices to take again for~$\Gar$ the set of \emph{columns}, that is,
the strictly decreasing products of elements of~$X$.

\smallbreak The Chinese monoid of rank~$n$ is (see~\cite{chinese})
\[{\mathbf C}_n = \presm{~1<\cdots<n:\ttz\tty\ttx = \ttz\ttx\tty\  = \tty\ttz\ttx\quad \text{for} \ \ttx\le\tty\le\ttz~}.\]
According to~\cite{CGM}, ${\mathbf C}_n$ is also generated by~$\{\ttx:n\ge\ttx\ge 1\}\cup\{\tty\ttx:\tty>\ttx\}$ from which one can deduce an automatic structure. According to~\cite[Example~5.8]{dehgui}, ${\mathbf C}_n$ admits a quadratic normalisation on~$\Gar=\{\ttx:n\ge\ttx\ge 1\}\cup\{\tty\ttx:\tty>\ttx\}\,\cup\,\{\ttx^2:n>\ttx>1\}$ with breadth~$(4,4)$.
We independently compute a quadratic normalisation always on~$\Gar$ now with breadth~$(4,3)$, hence satisfying Condition~\home,
and Theorem~\ref{thm-main} ensures that ${\mathbf C}_n$ is automaticon monoid~\cite{pHdR}.

\end{example}

\begin{example}\label{appendix-max}
The complexity measure of a quadratic normalisation we called \emph{breadth} is crucial.
We have seen that, for a somehow limited breadth, say~$(4,3)$, \emph{aka} Condition~\home,
various semigroups can be reached by our approach and~Theorem~\ref{thm-main}.
Here we precise how a large breadth need not mean specially intricate semigroups.

One can build quadratic normalisations with a (finite) breadth arbitrarily large (see~\cite
{dehDLT,dehgui}).
A~natural question would be to know what is the maximal breadth for a fixed size of~$\Gar$.
For instance, the semigroup
$\WW=\press{~\tta,\ttb,\ttc:\tta\tta=\ttc\ttc=\ttb\ttc , \ttb\tta=\ttc\ttb=\tta\ttb , \ttb\ttb=\ttc\tta=\tta\ttc~}$
admits a quadratic normalisation of breadth~$(11,10)$, see Figures~\ref{fig-breadth1110} and~\ref{fig-ngraph}, that happens to correspond with the maximal breadth for~$\card{\Gar}=3$.

\begin{figure}[h]
\vspace*{-35pt}
\centering
 \rotatebox{-15}{
 \begin{tikzpicture}[thick,node distance=12mm,inner sep=1.3pt]
	\node (00) at (0,0) {};
	\node (100) [above of=00] {};
	\begin{scope}[node distance=11mm]
	\node (101) [left of=100] 	{};
	\node (102) [right of=100] 	{};
	\node (103) [left of=101] 	{};
	\node (104) [right of=102] 	{};
	\node (105) [left of=103] 	{};
	\node (106) [right of=104] 	{};
	\node (107) [left of=105] 	{};
	\node (108) [right of=106] 	{};
	\node (109) [left of=107] 	{};
	\node (110) [right of=108] 	{};
	\node (111) [left of=109] 	{};
	\node (112) [right of=110] 	{};
	\end{scope}
	\node (200) [above of=100] {};
	\node (201) [above of=101] {};
	\node (202) [above of=102] {};
	\node (203) [above of=103] {};
	\node (204) [above of=104] {};
	\node (205) [above of=105] {};
	\node (206) [above of=106] {};
	\node (207) [above of=107] {};
	\node (208) [above of=108] {};
	\node (209) [above of=109] {};
	\node (210) [above of=110] {};
	\node (211) [above of=111] {};
	\node (212) [above of=112] {};
	\node (30) [above of=200] {};
	%
	\path[->,>=latex,line width=1.4pt]	(30)	edge 		node[pos=.62,fill=white,rotate=15]{\footnotesize$\!\ttc\!$} (200);
	\path[->,>=latex,line width=1.4pt]	(200)	edge			node[pos=.49,fill=white,rotate=15]{\footnotesize$\!\ttb\!$} (100);
	\path[->,>=latex,line width=1.4pt]	(100)	edge			node[pos=.32,fill=white,rotate=15]{\footnotesize$\!\tta\!$} (00);
	\path[->,>=latex] 	(30)	edge			node[pos=.62,fill=white,rotate=15]{\footnotesize$\!\tta\!$} (201);
	\path[->,>=latex]	(30)	edge			node[pos=.62,fill=white,rotate=15]{\footnotesize$\!\tta\!$} (202);
	\path[->,>=latex]	(30)	edge			node[pos=.62,fill=white,rotate=15]{\footnotesize$\!\ttb\!$} (203);
	\path[->,>=latex]	(30)	edge			node[pos=.62,fill=white,rotate=15]{\footnotesize$\!\ttb\!$} (204);
	\path[->,>=latex]	(30)	edge			node[pos=.62,fill=white,rotate=15]{\footnotesize$\!\tta\!$} (205);
	\path[->,>=latex]	(30)	edge			node[pos=.62,fill=white,rotate=15]{\footnotesize$\!\tta\!$} (206);
	\path[->,>=latex]	(30)	edge			node[pos=.62,fill=white,rotate=15]{\footnotesize$\!\ttb\!$} (207);
	\path[->,>=latex]	(30)	edge			node[pos=.62,fill=white,rotate=15]{\footnotesize$\!\ttb\!$} (208);
	\path[->,>=latex]	(30)	edge			node[pos=.62,fill=white,rotate=15]{\footnotesize$\!\tta\!$} (209);
	\path[->,>=latex]	(30)	edge			node[pos=.62,fill=white,rotate=15]{\footnotesize$\!\tta\!$} (210);
	\path[->,>=latex]	(200)	edge			node[pos=.49,fill=white,rotate=15]{\footnotesize$\!\tta\!$} (101);
	\path[->,>=latex]	(201)	edge			node[pos=.49,fill=white,rotate=15]{\footnotesize$\!\ttc\!$} (101);
	\path[->,>=latex]	(201)	edge			node[pos=.49,fill=white,rotate=15]{\footnotesize$\!\tta\!$} (103);
	\path[->,>=latex]	(202)	edge			node[pos=.49,fill=white,rotate=15]{\footnotesize$\!\ttb\!$} (100);
	\path[->,>=latex]	(202)	edge			node[pos=.49,fill=white,rotate=15]{\footnotesize$\!\tta\!$} (102);
	\path[->,>=latex]	(203)	edge			node[pos=.49,fill=white,rotate=15]{\footnotesize$\!\ttc\!$} (103);
	\path[->,>=latex]	(203)	edge			node[pos=.49,fill=white,rotate=15]{\footnotesize$\!\tta\!$} (105);
	\path[->,>=latex]	(204)	edge			node[pos=.49,fill=white,rotate=15]{\footnotesize$\!\ttc\!$} (102);
	\path[->,>=latex]	(204)	edge			node[pos=.49,fill=white,rotate=15]{\footnotesize$\!\tta\!$} (104);
	\path[->,>=latex]	(205)	edge			node[pos=.49,fill=white,rotate=15]{\footnotesize$\!\ttb\!$} (105);
	\path[->,>=latex]	(205)	edge			node[pos=.49,fill=white,rotate=15]{\footnotesize$\!\tta\!$} (107);	
	\path[->,>=latex]	(206)	edge			node[pos=.49,fill=white,rotate=15]{\footnotesize$\!\ttb\!$} (104);
	\path[->,>=latex]	(206)	edge			node[pos=.49,fill=white,rotate=15]{\footnotesize$\!\tta\!$} (106);
	\path[->,>=latex]	(207)	edge			node[pos=.49,fill=white,rotate=15]{\footnotesize$\!\ttc\!$} (107);
	\path[->,>=latex]	(207)	edge			node[pos=.49,fill=white,rotate=15]{\footnotesize$\!\ttb\!$} (109);
	\path[->,>=latex]	(208)	edge			node[pos=.49,fill=white,rotate=15]{\footnotesize$\!\ttc\!$} (106);
	\path[->,>=latex]	(208)	edge			node[pos=.49,fill=white,rotate=15]{\footnotesize$\!\ttb\!$} (108);
	\path[->,>=latex]	(209)	edge			node[pos=.49,fill=white,rotate=15]{\footnotesize$\!\ttc\!$} (109);
	\path[->,>=latex]	(209)	edge			node[pos=.49,fill=white,rotate=15]{\footnotesize$\!\ttb\!$} (111);
	\path[->,>=latex]	(210)	edge			node[pos=.49,fill=white,rotate=15]{\footnotesize$\!\ttc\!$} (108);	
	\path[->,>=latex]	(210)	edge			node[pos=.49,fill=white,rotate=15]{\footnotesize$\!\ttb\!$} (110);
	\path[->,>=latex]	(101)	edge			node[pos=.32,fill=white,rotate=15]{\footnotesize$\!\ttb\!$} (00);
	\path[->,>=latex]	(102)	edge			node[pos=.32,fill=white,rotate=15]{\footnotesize$\!\ttb\!$} (00);
	\path[->,>=latex]	(103)	edge			node[pos=.32,fill=white,rotate=15]{\footnotesize$\!\ttb\!$} (00);
	\path[->,>=latex]	(104)	edge			node[pos=.32,fill=white,rotate=15]{\footnotesize$\!\ttb\!$} (00);
	\path[->,>=latex]	(105)	edge			node[pos=.32,fill=white,rotate=15]{\footnotesize$\!\ttb\!$} (00);
	\path[->,>=latex]	(106)	edge			node[pos=.32,fill=white,rotate=15]{\footnotesize$\!\ttc\!$} (00);
	\path[->,>=latex]	(107)	edge			node[pos=.32,fill=white,rotate=15]{\footnotesize$\!\ttc\!$} (00);
	\path[->,>=latex]	(108)	edge			node[pos=.32,fill=white,rotate=15]{\footnotesize$\!\ttc\!$} (00);
	\path[->,>=latex]	(109)	edge			node[pos=.32,fill=white,rotate=15]{\footnotesize$\!\ttc\!$} (00);
	\path[->,>=latex]	(110)	edge			node[pos=.32,fill=white,rotate=15]{\footnotesize$\!\ttc\!$} (00);
	\path[->,>=latex]	(111)	edge			node[pos=.32,fill=white,rotate=15]{\footnotesize$\!\ttc\!$} (00);
\end{tikzpicture}
}
\vspace*{-50pt}
\caption{From the initial word~$\ttc\ttb\tta$, one applies normalisations from~$\WW$ on the first and the second~$2$-factors alternatively up to stabilisation,
beginning on the first $2$-factor, namely~$\ttc\ttb$ (on the right-hand side here) or on the second, namely~$\ttb\tta$. This witnesses the large breadth~$(11,10)$ for~$\WW$.
}%
\label{fig-breadth1110}
\end{figure}
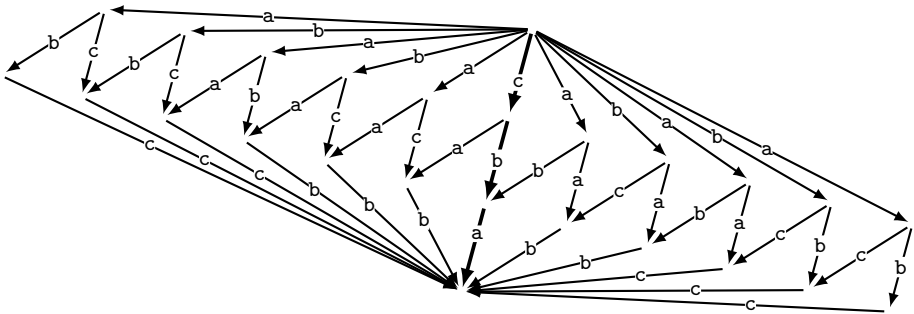

Such a large breadth corresponds with a great height of the associated $\NMbar$-graph as displayed on Figure~\ref{fig-ngraph}.

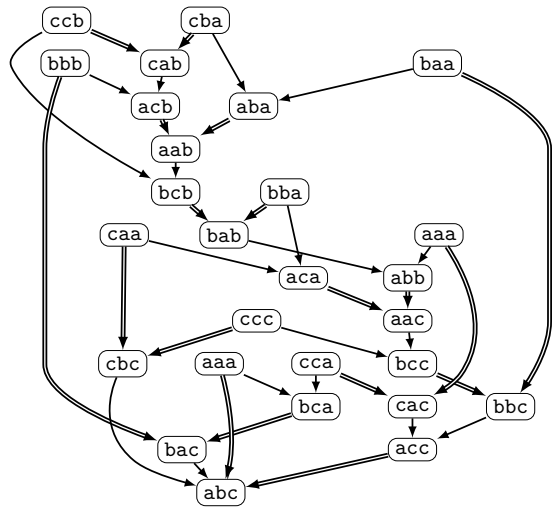
\begin{figure}[h]
\centering
\scalebox{.83}{
\begin{tikzpicture}[>=latex,xscale=.7,yscale=.27]
  \begin{scope}[inner sep=3pt,rounded corners,minimum height=3.5mm]
  \node (cbb) at (275.12bp,450.0bp) [draw] 	{$	\tta\tta\tta	$}	;
  \node (cbc) at (72.118bp,234.0bp) [draw] 	{$	\ttc\ttb\ttc	$}	;
  \node (cba) at (125.12bp,810.0bp) [draw] 	{$	\ttc\ttb\tta	$}	;
  \node (aba) at (156.12bp,666.0bp) [draw] 	{$	\tta\ttb\tta	$}	;
  \node (abb) at (255.12bp,378.0bp) [draw] 	{$	\tta\ttb\ttb	$}	;
  \node (abc) at (135.12bp,18.0bp) [draw] 		{$	\tta\ttb\ttc	$}	;
  \node (bbc) at (321.12bp,162.0bp) [draw] 	{$	\ttb\ttb\ttc	$}	;
  \node (bbb) at (35.118bp,738.0bp) [draw] 	{$	\ttb\ttb\ttb	$}	;
  \node (bba) at (176.12bp,522.0bp) [draw] 	{$	\ttb\ttb\tta	$}	;
  \node (acc) at (258.12bp,90.0bp) [draw] 		{$	\tta\ttc\ttc	$}	;
  \node (acb) at (93.118bp,666.0bp) [draw] 	{$	\tta\ttc\ttb	$}	;
  \node (aca) at (188.12bp,378.0bp) [draw] 	{$	\tta\ttc\tta	$}	;
  \node (aaa) at (134.12bp,234.0bp) [draw] 	{$	\tta\tta\tta	$}	;
  \node (aac) at (255.12bp,306.0bp) [draw] 	{$	\tta\tta\ttc	$}	;
  \node (aab) at (106.12bp,594.0bp) [draw] 	{$	\tta\tta\ttb	$}	;
  \node (bca) at (196.12bp,162.0bp) [draw] 	{$	\ttb\ttc\tta	$}	;
  \node (bcb) at (106.12bp,522.0bp) [draw] 	{$	\ttb\ttc\ttb	$}	;
  \node (bcc) at (258.12bp,234.0bp) [draw] 	{$	\ttb\ttc\ttc	$}	;
  \node (bab) at (137.12bp,450.0bp) [draw] 	{$	\ttb\tta\ttb	$}	;
  \node (bac) at (110.12bp,90.0bp) [draw] 		{$	\ttb\tta\ttc	$}	;
  \node (baa) at (274.12bp,738.0bp) [draw] 	{$	\ttb\tta\tta	$}	;
  \node (cac) at (258.12bp,162.0bp) [draw] 	{$	\ttc\tta\ttc	$}	;
  \node (cab) at (99.118bp,738.0bp) [draw] 	{$	\ttc\tta\ttb	$}	;
  \node (caa) at (73.118bp,450.0bp) [draw] 	{$	\ttc\tta\tta	$}	;
  \node (cca) at (196.12bp,234.0bp) [draw] 	{$	\ttc\ttc\tta	$}	;
  \node (ccc) at (158.12bp,306.0bp) [draw] 	{$	\ttc\ttc\ttc	$}	;
  \node (ccb) at (36.118bp,810.0bp) [draw] 	{$	\ttc\ttc\ttb	$}	;
  \end{scope}
  \draw [->,thick] (aab) ..controls (106.12bp,567.98bp) and (106.12bp,558.71bp)  .. (bcb);				
  \draw [double,->,thick] (aca) ..controls (210.68bp,353.43bp) and (224.03bp,339.48bp)  .. (aac);				
  \draw [double,->,thick] (ccb) ..controls (57.662bp,785.06bp) and (69.524bp,771.88bp)  .. (cab);				
  \draw [->,thick] (baa) ..controls (237.1bp,715.04bp) and (205.02bp,696.01bp)  .. (aba);				
  \draw [double,->,thick] (bcb) ..controls (117.02bp,496.37bp) and (121.65bp,485.93bp)  .. (bab);				
  \draw [->,thick] (cac) ..controls (258.12bp,135.98bp) and (258.12bp,126.71bp)  .. (acc);				
  \draw [double,->,thick] (bca) ..controls (167.54bp,137.73bp) and (148.3bp,122.08bp)  .. (bac);				
  \draw [double,->,thick] (cca) ..controls (217.11bp,209.3bp) and (229.05bp,195.82bp)  .. (cac);				
  \draw [double,->,thick] (bcc) ..controls (279.66bp,209.06bp) and (291.52bp,195.88bp)  .. (bbc);				
  \draw [->,thick] (bbb) ..controls (55.068bp,712.92bp) and (65.617bp,700.19bp)  .. (acb);				
  \draw [double,->,thick] (ccc) ..controls (129.9bp,282.03bp) and (110.39bp,266.15bp)  .. (cbc);				
  \draw [double,->,thick] (cba) ..controls (115.91bp,784.22bp) and (112.17bp,774.14bp)  .. (cab);				
  \draw [->,thick] (ccc) ..controls (190.2bp,282.54bp) and (215.01bp,265.18bp)  .. (bcc);				
  \draw [->,thick] (bbc) ..controls (299.57bp,137.06bp) and (287.71bp,123.88bp)  .. (acc);				
  \draw [->,thick] (bab) ..controls (174.34bp,426.92bp) and (206.06bp,408.1bp)  .. (abb);				
  \draw [->,thick] (ccb) ..controls (16.498bp,784.66bp) and (7.2294bp,770.38bp)  .. (3.1179bp,756.0bp) .. controls (-1.2799bp,740.62bp) and (-0.5581bp,735.57bp)  .. (3.1179bp,720.0bp) .. controls (19.221bp,651.78bp) and (64.313bp,581.21bp)  .. (bcb);				
  \draw [->,thick] (bac) ..controls (118.97bp,64.216bp) and (122.57bp,54.14bp)  .. (abc);				
  \draw [->,thick] (cbb) ..controls (268.02bp,424.14bp) and (265.27bp,414.54bp)  .. (abb);				
  \draw [double,->,thick] (acb) ..controls (97.697bp,640.35bp) and (99.427bp,631.03bp)  .. (aab);				
  \draw [double,->,thick] (abb) ..controls (255.12bp,351.98bp) and (255.12bp,342.71bp)  .. (aac);				
  \draw [double,->,thick] (caa) ..controls (72.865bp,394.83bp) and (72.455bp,307.18bp)  .. (cbc);				
  \draw [double,->,thick] (cbb) ..controls (289.64bp,393.42bp) and (310.42bp,294.62bp)  .. (289.12bp,216.0bp) .. controls (286.23bp,205.35bp) and (280.61bp,194.65bp)  .. (cac);				
  \draw [->,thick] (cba) ..controls (134.11bp,767.81bp) and (143.99bp,722.57bp)  .. (aba);				
  \draw [->,thick] (caa) ..controls (109.2bp,427.04bp) and (140.46bp,408.01bp)  .. (aca);				
  \draw [->,thick] (cab) ..controls (96.974bp,711.98bp) and (96.179bp,702.71bp)  .. (acb);				
  \draw [double,->,thick] (baa) ..controls (307.7bp,698.63bp) and (346.12bp,646.69bp)  .. (346.12bp,595.0bp) .. controls (346.12bp,595.0bp) and (346.12bp,595.0bp)  .. (346.12bp,305.0bp) .. controls (346.12bp,264.38bp) and (336.07bp,218.2bp)  .. (bbc);				
  \draw [double,->,thick] (bba) ..controls (162.43bp,496.43bp) and (156.33bp,485.49bp)  .. (bab);				
  \draw [->,thick] (bba) ..controls (179.61bp,479.67bp) and (183.37bp,435.21bp)  .. (aca);				
  \draw [double,->,thick] (bbb) ..controls (28.761bp,693.4bp) and (22.118bp,640.3bp)  .. (22.118bp,595.0bp) .. controls (22.118bp,595.0bp) and (22.118bp,595.0bp)  .. (22.118bp,233.0bp) .. controls (22.118bp,183.11bp) and (61.852bp,136.15bp)  .. (bac);				
  \draw [->,thick] (cbc) ..controls (63.881bp,185.62bp) and (57.046bp,120.36bp)  .. (79.118bp,72.0bp) .. controls (85.737bp,57.497bp) and (98.317bp,45.018bp)  .. (abc);				
  \draw [->,thick] (aaa) ..controls (155.0bp,209.42bp) and (166.74bp,196.17bp)  .. (bca);				
  \draw [double,->,thick] (aaa) ..controls (138.97bp,186.07bp) and (144.07bp,124.24bp)  .. (141.12bp,72.0bp) .. controls (140.64bp,63.549bp) and (139.74bp,54.376bp)  .. (abc);				
  \draw [double,->,thick] (aba) ..controls (138.89bp,640.87bp) and (130.21bp,628.73bp)  .. (aab);				
  \draw [->,thick] (cca) ..controls (196.12bp,207.98bp) and (196.12bp,198.71bp)  .. (bca);				
  \draw [->,thick] (aac) ..controls (256.19bp,279.98bp) and (256.59bp,270.71bp)  .. (bcc);				
  \draw [double,->,thick] (acc) ..controls (220.18bp,67.408bp) and (185.67bp,47.767bp)  .. (abc);				
\end{tikzpicture}
}
\vspace*{-5pt}
\caption{The $\NMbar$-graph for the quadratic normalisation associated with~$\WW$: simple arrows correspond to~$\NMbar_1$ and double arrows to~$\NMbar_2$, while loops are simply omitted for better readability.}
\label{fig-ngraph}
\end{figure}

It becomes clear that the larger the breadth, the higher the $\NMbar$-graph, the most the semigroup approximates~$\press{~a:~~}$.
This can be compared with a null breadth, that is, with a quadratic normalisation with no rules at all, which generates the rank~$\card{\Gar}$ free semigroup.
\end{example}

To conclude, we would like to illustrate the duality between
"being an automatic semigroup" and "being an automaton semigroup"
by highlighting a paradigmatic example.

\begin{example}\label{ex-braid} The braid monoids were chosen by Thurston \cite[Chapter~9]{EpsteinWord} to describe his idea to build
a single transducer that computes the so-called Adjan--Garside--Thurston normal form via iterated runs.
The (classical) $n$-strand braid monoid is
\[\BB^{\ttunit}_{\hspace*{-1pt}n+}=\Bigg\langle \sigma_1,\ldots,\sigma_{n-1}:
\scalebox{1}{$\begin{array}{cr}\sigma_{i}\sigma_{j}\sigma_{i}=\sigma_{j}\sigma_{i}\sigma_{j}&\hbox{ for }|i-j|\leq 1\\
\sigma_{i}\sigma_{j}=\sigma_{j}\sigma_{i}&\hbox{ for }|i-j|>1\end{array}$}\ \Bigg\rangle^{\ttunit}_{\hspace*{-1pt}+}.\]
Figure~\ref{fig-comb} illustrates the combing of some 4-strand braid diagram to obtain its $\NM$-normal form,
mimicking the square diagrams from Definitions~\ref{def-norm} and~\ref{def-msqn}.
A small triangle stands for the beginning of a strand. The initial diagram runs from north-east corner to west.
One step of normalisation applies on the west-most 2-window, and so on. 

\begin{figure}[h]
\begin{center}
\includegraphics[width=.79\textwidth]{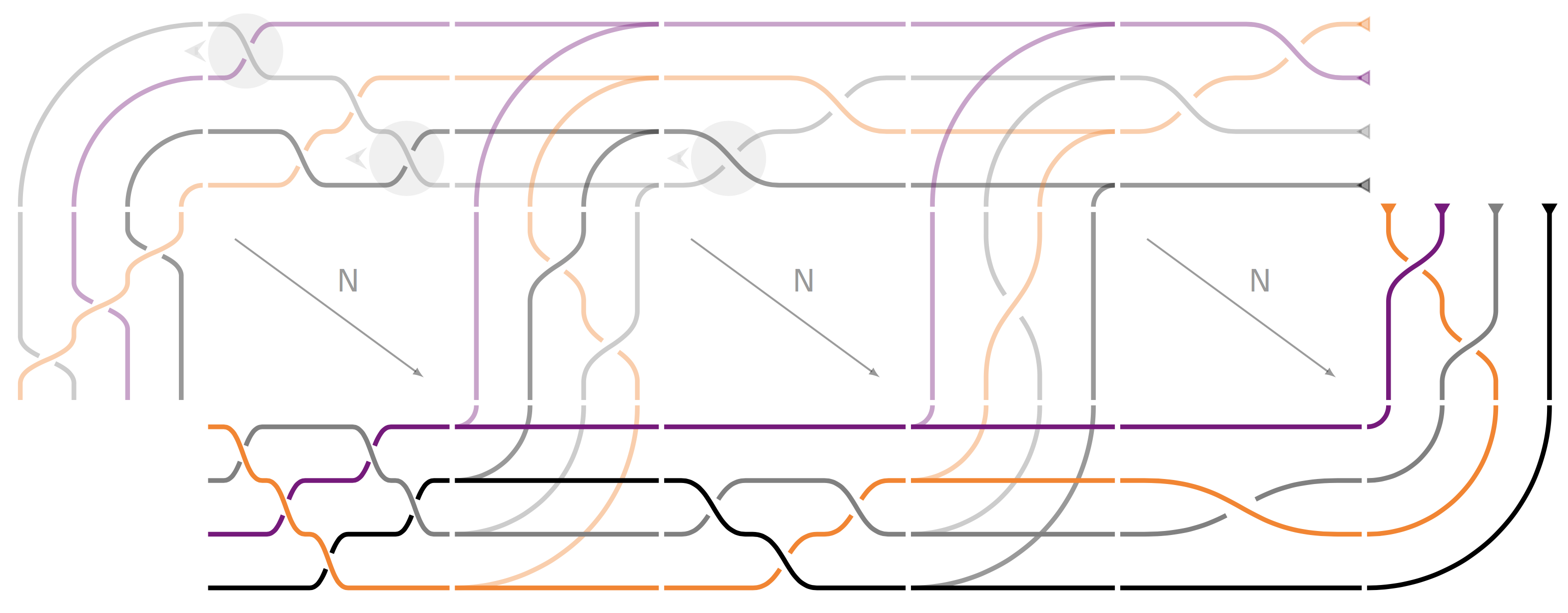}
\end{center}
\vspace*{-12pt}
\caption{The combing of some 4-strand braid diagram to obtain its $\NM$-normal form.
}%
\label{fig-comb}
\end{figure}

Garside theory allows to build a suitable generating set~$\Gar$ of size~$n!$ and a corresponding quadratic normalisation with breadth~$(3,3$).
According to Corollary~\ref{cor-duality}, its Thurston transducer 
and its Mealy automaton make therefore~$\BB^{\ttunit}_{n+}$ both an automatic and an automaton monoid. 
The 3-strand braid monoid is
$\BB^{\ttunit}_{\hspace*{-1pt}3\!+}=\presm{~\scalebox{.8}{\BRICa},\scalebox{.8}{\BRICb}:\scalebox{.8}{\BRICa}\scalebox{.8}{\BRICb}\scalebox{.8}{\BRICa}=\scalebox{.8}{\BRICb}\scalebox{.8}{\BRICa}\scalebox{.8}{\BRICb}~}$.
A fragment of its Thurston transducer is displayed beforehand on Figure~\ref{fig-thurston}.
Its Thurston transducer and Mealy automaton are displayed on Figure~\ref{fig-braid}.

\begin{figure}[h]
\begin{center}
\includegraphics[width=.79\textwidth]{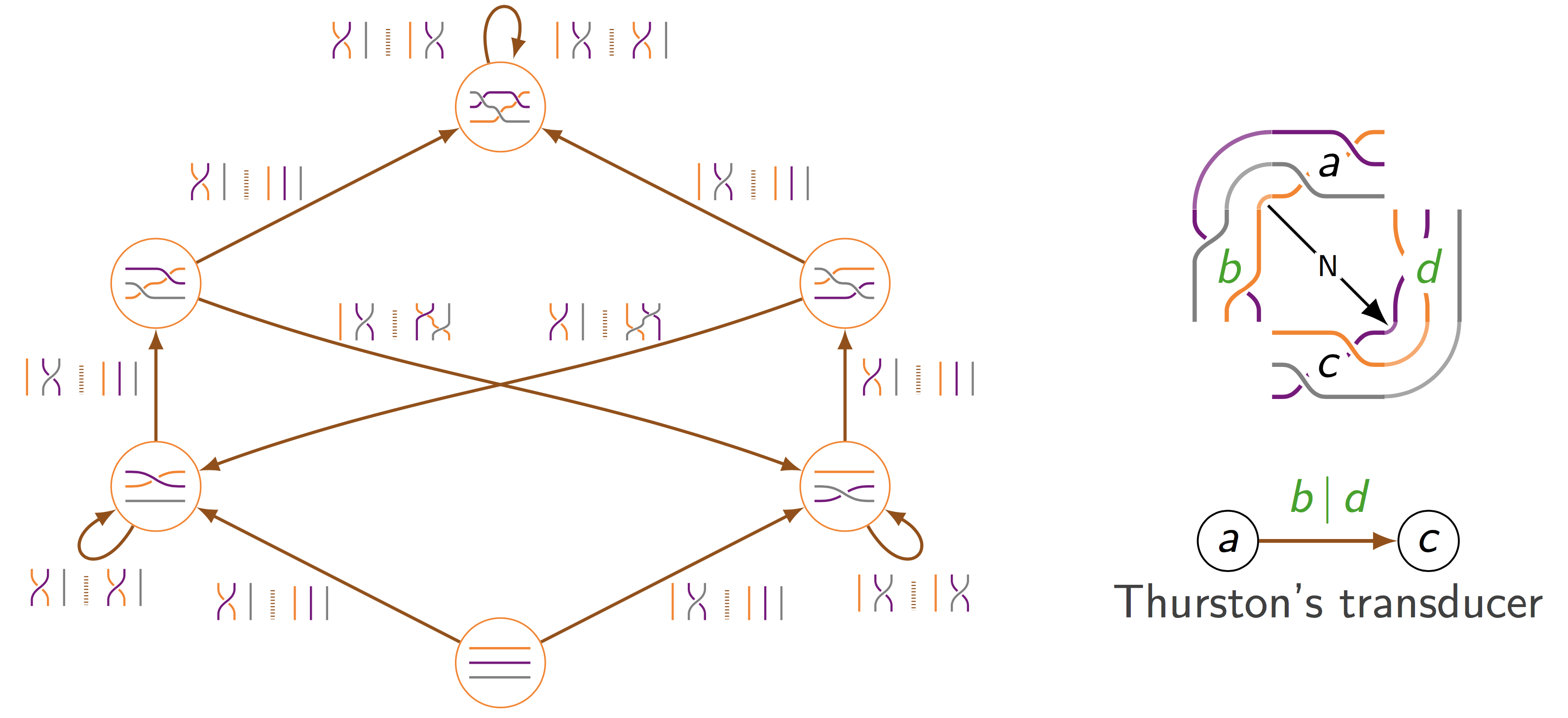}
\end{center}
\vspace*{-12pt}
\caption{A fragment of the Thurston transducer for~$\BB^{\ttunit}_{\hspace*{-1pt}3\!+}$:
while the stateset is complete,
the input alphabet is here restricted
to the two initial generators~$\scalebox{.8}{\BRICaS}$ and~$\scalebox{.8}{\BRICbS}$\ .
}%
\label{fig-thurston}
\end{figure}

Such an approach may hopefully shed some light on the question of whether or not the braid groups are self-similar (Problem~B).
In particular, a positive answer to our following Problem~C would imply a positive answer to Problem~B.
\begin{quote}
{\sffamily\bfseries\raggedright Problem C.} Is the group of fractions of an automaton monoid an automaton group?
\end{quote}

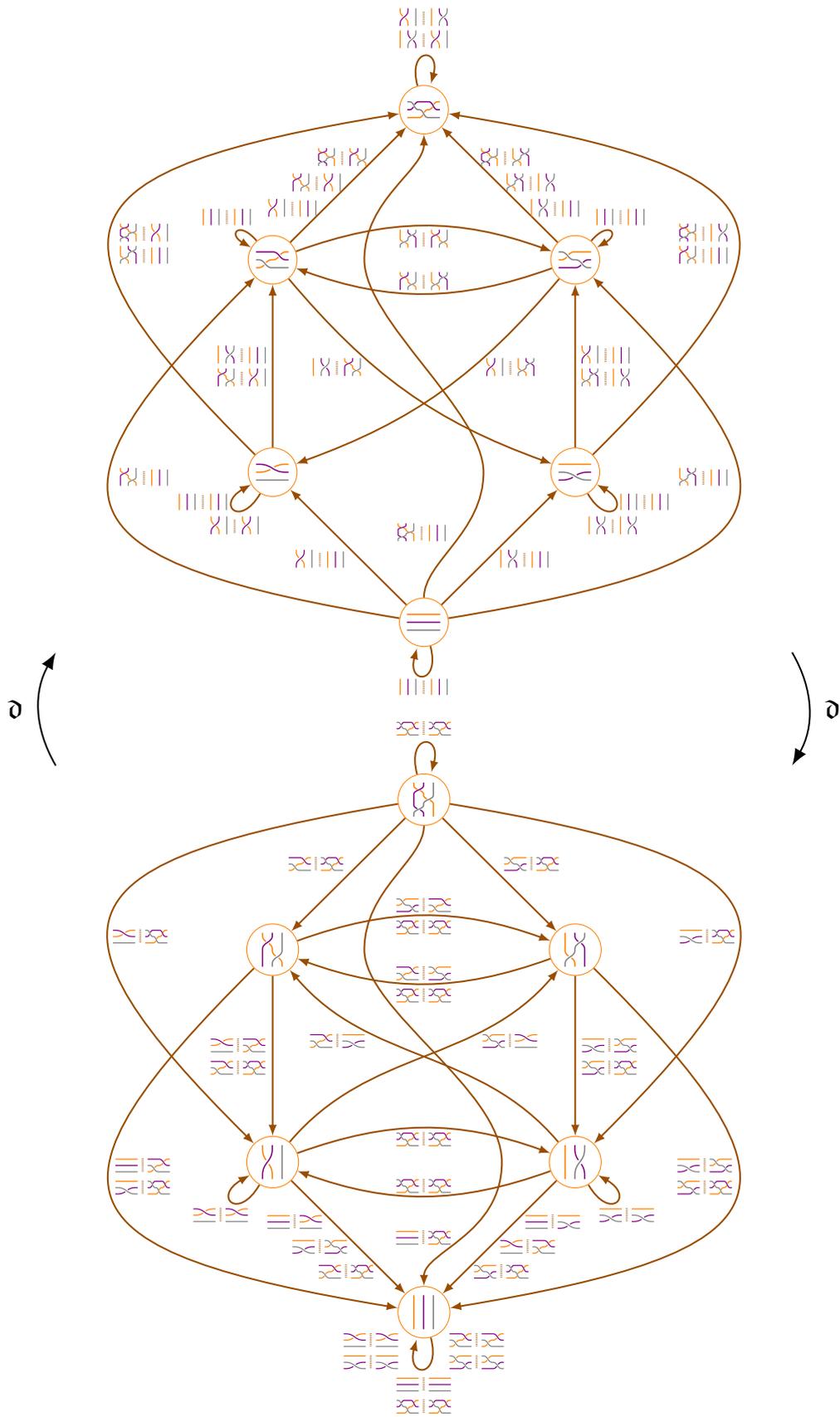
\begin{figure}[H]
\centering
\scalebox{1.9}{
\begin{tikzpicture}
\node (ic) at (0,0){\BThreeAutomatic};
\node (on) at (0,-5.8) {\BThreeAutomaton};
\node (01) at (-3,-2.3) {};
\node (00) at (-3,-3.5) {};
\node (11) at (3,-2.3) {};
\node (10) at (3,-3.5) {};
\path[->,>=latex] (00) edge[bend left] node[left]	{\scriptsize$\dz$} (01);
\path[->,>=latex] (11) edge[bend left] node[right]{\scriptsize$\dz$} (10);
\end{tikzpicture}
}
\vspace*{-10pt}
\caption{The Thurston transducer (top) \emph{vs} the Mealy automaton (bottom) for the $3$-strand braid monoid~$\BB^{\ttunit}_{3\!+}$ from~Example~\ref{ex-braid}.
}%
\label{fig-braid}
\end{figure}

\end{example}

\nocite{pOneRelator,pTransducer,cant,spol,darla}

\bibliography{automaticon}

\end{document}